%% file: arxiv_bi_immune_subgroup.tex
[1994/12/01]  
\documentclass[manuscript]{article}

\title{Combining Determinism and Indeterminism}
\author{Michael Stephen Fiske}

\input{arxiv_preamble.tex}

\begin{document}

\maketitle

\begin{abstract}
Our goal is to construct mathematical operations that combine indeterminism measured from quantum randomness 
with computational determinism so that non-mechanistic behavior is preserved in the computation.    
Some results about operations applied to computably enumerable (c.e.) and bi-immune sets are proven here, 
where the objective is for the operations to preserve bi-immunity.  While developing rearrangement operations on the 
natural numbers, we discovered that bi-immune permutations generate an uncountable subgroup of 
the infinite symmetric group $\Sym(\mathbb{N})$ on the natural numbers $\mathbb{N}$.  Moreover, 
it is shown that the Turing incomputable permutations generate all of $\Sym(\mathbb{N})$.

The subgroup generated by the bi-immune permutations is called the bi-immune symmetric group.  
It contains the finitary symmetric group on 
$\mathbb{N}$, and consequently is highly transitive. 
Furthermore, the bi-immune symmetric group is dense in $\Sym(\mathbb{N})$ with 
respect to the pointwise convergence topology.    
Moreover, it contains permutations that are a 
single infinite cycle on all of $\mathbb{N}$, even though every one of its generators has only finite orbits.
We also prove that the bi-immune symmetric group is a meager subset and a proper subgroup of  $\Sym(\mathbb{N})$, 
using  an invariant called the {\it crossing number}.

Replacing adjacent transpositions by disjoint transpositions across dyadic scales
yields bi-immune dyadic rearrangements that escape the crossing-number obstruction;
together with the original generators, they still generate a proper, meager subgroup of $\Sym(\mathbb{N})$.
Building on a factorization of Bergman and Shelah, it is proven that the bi-immune block
involutions on partitions of $\mathbb{N}$ into finite intervals generate all of
$\Sym(\mathbb{N})$, with every permutation a product of at most six of them, while
those on computable partitions generate exactly the computably bounded permutations.
Lastly, since the symmetric difference of a bi-immune set and a computable set is
bi-immune, a single bi-immune source, such as a quantum random sequence, combined
with computable data determined by the permutation, expresses every computable
permutation as a product of six bi-immune block involutions.
\end{abstract}


\input{bi_immune_intro.tex}

\input{bi_immune_lemmas.tex}

\input{rearrangements.tex}

\input{incomputable_permutations.tex}


\input{bi_immune_permutations.tex}

\input{permutation_metric.tex}

\input{normal_subgroup.tex}


\input{bi_immune_infinite_cycles.tex}

\input{dense_bi-immune_orbits.tex}

\input{eventually_commutative.tex}

\input{proper_bi_immune_subgroups.tex}

\input{bi_immune_density.tex}

\input{bi_immune_crossing_number.tex}

\input{bi_immune_dyadic_subgroup.tex}

\newpage

\input{bi_immune_references.tex}

\end{document}

%% file: arxiv_preamble.tex
\usepackage{amsmath,amsthm,amsfonts,amssymb,mathtools,braket}
\usepackage{marvosym}
\usepackage{graphics}
\usepackage{url}

\chardef\bslash=`\\ 

\theoremstyle{plain}
\newtheorem{thm}{Theorem}[section]
\newtheorem{cor}[thm]{Corollary}
\newtheorem{lem}[thm]{Lemma}
\theoremstyle{definition}
\newtheorem{defn}[thm]{Definition}
\theoremstyle{remark}
\newtheorem{rem}[thm]{Remark}

\newtheorem{example}{Example}

\newcommand{\Sym}{\mathrm{Sym}}
\newcommand{\dcup}{\mathbin{\dot{\cup}}}

\DeclareMathOperator{\supp}{supp}

\newcommand{\eval}[2][\right]{\relax
  \ifx#1\right\relax \left.\fi#2#1\rvert}



%% file: bi_immune_intro.tex

\section{Introduction}
\label{sec:1}

In \cite{fiske_qaem}, a lemma about the symmetric difference operator applied to a computably enumerable (c.e.) set and 
bi-immune set is stated without proof.   Herein lemma  \ref{lem:recursive_xor_bi_immune}  provides a proof;  
this helps  characterize procedure 2's non-mechanistic behavior in  \cite{fiske_TIC}.   
Moreover, the preservation of bi-immunity by the symmetric difference operator helped motivate the 
conception of rearrangement operations.  The preservation of bi-immunity led to the notion of 
generating subgroups of $\Sym(\mathbb{N})$ with one or more bi-immune rearrangements.   
The structure of the infinite subgroups of $\Sym(\mathbb{N})$ is an active area of research in group theory.

This paper is organized as follows.  Section \ref{sect:preserve_indeterminism} proves that
adding or removing finitely many elements, and taking symmetric differences with computably
enumerable sets, preserve bi-immunity.
 Section \ref{sect:rearrangements} defines    
rearrangements  $\sigma_A$ of $\mathbb{N}$ induced by subsets $A \subseteq \mathbb{N}$.  
In section \ref{sect:incomputable_permutations}, the Turing incomputable permutations are shown to generate all of
$\Sym(\mathbb{N})$.  Section  \ref{sect:bi_immune_permutations} defines  
the bi-immune rearrangements that generate the bi-immune symmetric group $G_{\mathfrak{B}}$. 
$G_{\mathfrak{B}}$ contains the finitary symmetric group $FS(\mathbb{N})$, so it is highly
transitive and dense in $\Sym(\mathbb{N})$; moreover, $FS(\mathbb{N})$ is a normal subgroup
of $G_{\mathfrak{B}}$.

Section \ref{sec:infinite_cycles} shows that
$G_{\mathfrak{B}}$ contains permutations consisting of a single infinite cycle, even
though every generator has only finite orbits.  Section \ref{sec:dense_orbits} demonstrates 
dense orbits of bi-immune points in $[0,1)$, using the binary odometer, and relates
rotations to Martin-L\"of randomness.  
Section \ref{sect:proper_bi_immune_subgroups} constructs infinitely many
distinct proper subgroups of $G_{\mathfrak{B}}$ as stabilizers of subsets of
$\mathbb{N}$.  Section \ref{sec:crossing} introduces the crossing number and proves
that $G_{\mathfrak{B}}$ is a meager, proper subgroup of $\Sym(\mathbb{N})$ of index
$2^{\aleph_0}$.

Section \ref{sec:dyadic} defines dyadic rearrangements $\tau_A$.  
Section \ref{sec:dyadic_turing} shows that $A$ and $\tau_A$ have the same Turing degree
for every $A \subseteq \mathbb{N}$.  Section \ref{sec:dyadic_escape} shows that the
$\tau_A$ escape the crossing-number obstruction.  Section \ref{sec:joint} shows that,
together with $\sigma_A$, these rearrangements still generate a proper subgroup.  
Section \ref{sec:interval_partitions} shows that block permutations from finitely many fixed
partitions of $\mathbb{N}$ into finite intervals cannot generate $\Sym(\mathbb{N})$.

Building on a factorization of Bergman and Shelah \cite{bergman_shelah}, sections
\ref{sec:adaptive} and \ref{sec:universal} prove that bi-immune block involutions
generate $\Sym(\mathbb{N})$, with every permutation being a product of at most six bi-immune block involutions.
Lastly, because the symmetric difference of a bi-immune set and a computable set is
bi-immune, a single bi-immune source, such as a quantum random binary sequence, 
together with computable data, expresses every computable permutation as such a product.

\subsection{Symbols and Conventions}\label{sect:symbols}
Since the literature in computer science, computability theory, and group theory  
do not always use the same symbols and conventions, 
we review our symbols and conventions.
$\mathbb{N}$  is the  non-negative integers.      
$\mathbb{E} = \{ 2n: n  \in \mathbb{N} \}$ are the even, non-negative integers. 
$\mathbb{O} = \{ 2n+1: n  \in \mathbb{N} \}$ are the odd, non-negative integers. 
If $S$ is a set, $|S|$ is the cardinality of $S$.   
The symbol $\circ$ represents function composition;   if  $f: X \rightarrow Y$ and $g: Y \rightarrow Z$ then 
$g \circ f$ means apply function $f$ to an element $x$ in $X$, and then apply function $g$ to argument  $f(x)$ in $Y$.

Let $a_0$ $a_1$ $\dots$ $\in \{0, 1\}^{\mathbb{N}}$ be a binary sequence. 
The sequence $a_0$ $a_1$ $\dots$  induces a set $A \subset \mathbb{N}$ and vice versa:  
the {\it standard identification}   of $A$ with the sequence $a_0$ $a_1$ $\dots$ means $k  \in A$ 
if and only if $a_k = 1$.  Also, $k \notin A$ iff $a_k = 0$.   
$\overline{A} = \{ n \in \mathbb{N}: n \notin A \}$ is the complement of set $A$ in $\mathbb{N}$.  
The relative complement is  $A- B  = \{x \in A:  x \notin B \}$.  In computer 
science, $\oplus$ is exclusive-or:  $0 \oplus 0 = 1 \oplus 1 = 0$ and $1 \oplus 0 = 0 \oplus 1 = 1.$ 
Due to the identification between a binary sequence and a subset of 
$\mathbb{N}$,  the symmetric difference of $A$ and $B$ is represented as
 $A \oplus B =  (A-B)  \cup (B-A)$,  instead of the usual $\Delta$.  
 Herein symbol  $\oplus$ never represents the {\it join} operation as used in \cite{downey}.

$\Sym(\mathbb{N})$ is the group of all permutations of $\mathbb{N}$ under function composition.
For a subset $S \subseteq \Sym(\mathbb{N})$, $\langle S \rangle$ denotes the subgroup
generated by $S$, that is, the intersection of all subgroups of $\Sym(\mathbb{N})$
that contain $S$.



%% file: bi_immune_lemmas.tex


\section{Preserving Non-Mechanistic Behavior}\label{sect:preserve_indeterminism}

Our goal is to construct operations that combine indeterminism measured from quantum randomness \cite{herrero}  
with computational determinism \cite{turing} so that the non-mechanistic behavior (bi-immunity) is preserved.   

In \cite{post}, Post introduced the notion of an immune set. 
\begin{defn}\label{defn_bi_immune}   \hskip 1pc   Set $A \subset \mathbb{N}$ is    
 {\it immune} if conditions (i) and (ii) hold.  

\smallskip 

\noindent   (i)   $A$  is infinite.

\smallskip  

\noindent    (ii)  For all $R \subset \mathbb{N}$, $(R$ is infinite and 
computably enumerable$) \implies$   $R \cap \overline{A} \ne \emptyset$. 

\medskip

\noindent  Set  $A$  is {\it bi-immune} if  both $A$  and $\overline{A}$ are  immune.  

\end{defn}

A binary sequence $(a_0 a_1 \dots)$ is called bi-immune if its corresponding set $A$ (via the standard identification) is bi-immune.  
Bi-immunity is a strong form of non-mechanistic behavior because no Turing machine can even produce a subsequence of a bi-immune sequence.

\begin{lem}\label{bi_finite_set}

Suppose $A$ is bi-immune.  Let $R$ be a finite set.   Then $ A \cup R$ 
and $A - R$ are both bi-immune.  

\end{lem}   

\begin{proof}
Set  $A^+ =  A \cup  R$ and $A^- =  A -  R$.   
From definition \ref{defn_bi_immune}, $A^+$ and  $A^-$ are still infinite because a finite number of elements have been
added to or removed from $A$, respectively.   For condition (ii), suppose there exists a computably enumerable set $B$ 
such that $B \cap \overline{A^+} = \emptyset$.   Then $B \cap \overline{A^+}$
and $B \cap \overline{A}$ only differ on a finite number of elements which implies that a c.e. set $B'$ 
can be constructed from $B$ such that $B' \cap \overline{A} = \emptyset$.  Contradiction.   
The same argument holds for $A^-$.  
\end{proof}


Lemma \ref{bi_finite_set} does not hold if $R$ is an infinite, computable set.  
$A \cup  \mathbb{E}$ is not immune as  ${\mathbb{E}} \cap (\overline{ A \cup {\mathbb{E}}}) = \emptyset$.
Also,  $\overline{ A - {\mathbb{E}}}$ is not immune    
as $\mathbb{E} \cap (A - \mathbb{E}) = \emptyset$.  
While the isolated operations of union and relative complement do not preserve bi-immunity, 
the symmetric difference operation preserves bi-immunity.

\begin{lem}\label{lem:recursive_xor_bi_immune}

If $R$ is c.e.  and $A$ is bi-immune,  then   $A \oplus R$ is bi-immune.   

\end{lem}

\begin{proof}
{\it Condition (i)}.  Verify that  $A \oplus R$  is infinite.   
For the case that $R$ is finite, let $K$ be the largest element in $R$.   
Then $A \oplus R$ is the disjoint union of the finite set  $\{x \in  A \oplus R:   x \le K \}$ and the infinite set  
$\{x \in A:  x > K\}$ since $A$ is bi-immune.

Otherwise, $R$ is infinite.   By contradiction, suppose $A \oplus R$  is finite.    
$A$ is the disjoint  union of $A-R$ and $A \cap R$, so 
$A$'s bi-immunity implies $A \cap R$ is infinite.  
Also, let $R-A = \{ r_1, r_2  \dots, r_m  \}$ since it is finite.  

\smallskip

Claim:  $A \cap R$ is c.e.  
Consider Turing machine $M$ that enumerates $R$.   
Whenever $M$ halts, concatenate machine $N$ to execute after machine $M$:

\begin{itemize} 

\item{Machine $N$ does not halt if $M$ halts with $r_k$ in $R-A$.}    

\smallskip 

\item{For all other outputs where $M$ halts, after machine $N$ checks that $M$'s output is not in $R-A$, then $N$ immediately halts.}

\end{itemize}  


\noindent Hence, $A \cap R$ is infinite and c.e., contradicting $A$'s bi-immunity, so $A \oplus R$  
must be infinite.

\medskip


\noindent  {\it Condition (ii)}.   Set $Q = A \oplus R$.   

\smallskip

By contradiction, suppose there exists an infinite, c.e. set $B$  with 
$B \cap \overline{Q} = \emptyset$.   Then $B \subset Q$.   Also, $B \cap R$ is c.e.    Now
$B \cap R \subset R - A \subset \overline{A}$ 
which contradicts that $A$ is bi-immune, if   $B \cap R$ 
is infinite. Otherwise, $B \cap R$ is finite.  Set $K = max(B \cap R)$.   
Then $B \cap (A-R)$ is infinite.  Define the infinite, c.e. set $B'  = \{ x \in B: x > K\}$.  
Thus, $B' \subset A$ contradicts $A$'s bi-immunity.  

\smallskip

Similarly, by contradiction, suppose there exists an infinite, c.e. set $B$  with 
$B \cap Q = \emptyset$.  Then $B \subset  \overline{Q}$.  Also, $B \cap R$ is c.e.  Thus,
$B \cap R \subset A \cap R$ which contradicts that $A$ is bi-immune if  $B \cap R$ is infinite. 
Otherwise, $B \cap R$ is finite.  Set $K = max(B \cap R)$.  
Define the infinite, c.e. set $B'  = \{ x \in B: x > K\}$.    
Thus, $B' \subset \overline{A}$, which contradicts $A$'s bi-immunity.    
\end{proof}



%% file: rearrangements.tex
\section{Rearrangements of Subsets of $\mathbb{N}$}\label{sect:rearrangements}

This subsection continues to explore how to preserve non-mechanistic behavior with operations based on 
permutations.  First, some definitions and results are developed about rearrangements of $\mathbb{N}$, 
induced by subsets of $\mathbb{N}$.   


Let $k \in  \mathbb{N}$.  A permutation $\sigma_{(k)}:  \mathbb{N} \rightarrow \mathbb{N}$ is generated from the 
identity permutation $\sigma_{\emptyset} = (0, 1, 2, 3, \dots)$ 
by transposing the $k$th entry and $k+1$th entry of  $\sigma_{\emptyset}$.   Thus,   
$\sigma_{(k)} (x) = x$ when $x \notin \{k, k+1\}$ and $\sigma_{(k)} (k) = k+1$ and  $\sigma_{(k)} (k+1) = k$.   

This can be repeated on $\sigma_{(k)}$ where $\sigma_{(k, k+1)}$ is generated from transposing 
the $k+1$ and $k+2$ entries of  $\sigma_{(k)}$.    Thus, $\sigma_{(k, k+1)}(x) = x$ when $x \notin \{k, k+1, k+2\}$.  
Otherwise, $\sigma_{(k, k+1)} (k) = k+1$, $\sigma_{(k, k+1)} (k+1) = k+2$, and  $\sigma_{(k, k+1)} (k+2) = k$.   
This leads to the simple observation that  $\sigma_{(k)} \circ \sigma_{(k+1)} = \sigma_{(k, k+1)}$.  

Consider the segment $[m, n] = (m, m+1, \dots, n)$ where $m < n$.   
Starting with $\sigma_{\emptyset}$, apply  the aforementioned transposition step $n - m +1$ times.   
Thus,  $\sigma_{(m, m+1, \dots, n)} (x) = x$ when $x < m$ or $x > n+1$.  
Otherwise, $\sigma_{(m, m+1, \dots, n)} (k) = k+1$ when $m \le k \le n$ and   $\sigma_{(m, m+1, \dots, n)} (n+1) = m$.  
From the prior observation,  it is apparent that  
$\sigma_{(m, m+1, \dots, n)} = \sigma_{(m)} \circ \sigma_{(m+1)} \circ \dots \circ \sigma_{(n)}$.
As examples, $\sigma_{(0, 1)} = (1, 2, 0, 3, 4, 5, \dots)$ and  $\sigma_{(4, 5, 6)} = (0, 1, 2, 3, 5, 6, 7, 4, 8, 9, 10, \dots)$.

\smallskip

It will be helpful to represent inverses with the same construction.   First, $\sigma_{(k)}$ is its own inverse.   
$\sigma_{(k)} \circ \sigma_{(k)} (k) = \sigma_{(k)}(k+1) = k$ and  $\sigma_{(k)} \circ \sigma_{(k)} (k+1)$ $= \sigma_{(k)}(k) = k+1$.    
Also $\sigma_{(k)} (x) = x$ when $x \notin \{k, k+1\}$.  Thus, $\sigma_{(k)} \circ \sigma_{(k)} (x) = x$ when $x \notin \{k, k+1\}$.
Since function composition is associative, 
$(\sigma_{(x)} \circ \sigma_{(y)}) \circ (\sigma_{(y)} \circ \sigma_{(x)}) = \sigma_{(x)} \circ (\sigma_{(y)}  \circ \sigma_{(y)}) \circ \sigma_{(x)} = \sigma_{\emptyset}$.   
It was already verified that   $\sigma_{(k)} \circ \sigma_{(k+1)} = \sigma_{(k, k+1)}$.

Observe that   $\sigma_{(0, 1)} = (1, 2, 0, 3, 4, 5, \dots) \ne  (2, 0, 1, 3, 4, 5, \dots) = \sigma_{(1, 0)}$.   
It is helpful to know when $\sigma_{(x)}$ and $\sigma_{(y)}$ commute.    
Lemma \ref{commute_remark} helps explain why commutativity fails for more 
complicated permutations such as $\sigma_{(5, 7, 2, 11,  4, 729)}$.

\begin{lem}\label{commute_remark}

If $|x-y| > 1$, then $\sigma_{(y)} \circ \sigma_{(x)} =  \sigma_{(x)} \circ \sigma_{(y)}$.    

\end{lem}

\begin{proof}     W.L.O.G., suppose $x+1 < y$.    From above, $\sigma_{(y)}$ is the identity map outside of $\{y, y+1\}$ and 
$\sigma_{(x)}$ is the identity map outside of $\{x, x+1\}$.   
Thus, $\sigma_{(y)} \circ \sigma_{(x)} (k) = k = \sigma_{(x)} \circ \sigma_{(y)} (k)$ when $k \notin \{x, x+1, y, y+1\}$. 
When $k = x$, $\sigma_{(y)} \circ \sigma_{(x)} (x) =  \sigma_{(y)}(x+1) = x+1$ because $x+1 < y$.    
Also, $\sigma_{(x)} \circ \sigma_{(y)} (x) = \sigma_{(x)}(x) = x+1$ because $x+1 < y$.

Similarly, $\sigma_{(y)} \circ \sigma_{(x)} (x+1)  = \sigma_{(y)}(x) = x = \sigma_{(x)} (x+1) =  \sigma_{(x)}  \circ \sigma_{(y)} (x+1)$ because $x+1 < y$.  The remaining verifications hold for $y$ and $y+1$ because $x+1 < y$.
\end{proof}

When $x < y$, the previous proof and $\sigma_{(k, k+1)} = \sigma_{(k)} \circ \sigma_{(k+1)}$
together imply that $\sigma_{(x)} \circ \sigma_{(y)} = \sigma_{(x, y)}$.
Although sequences with repeats won't be considered here, 
it is helpful to notice that $\sigma_{(x)} \circ \sigma_{(x)}=  \sigma_{(x, x)} = \sigma_{\emptyset}$.    
In some later constructions, it will be useful to know when
$\sigma_{(x, y)} = \sigma_{(x)} \circ \sigma_{(y)}$, and similarly when
$\sigma_{(a_0, a_1, \dots a_n)} = \sigma_{(a_0)} \circ \sigma_{(a_1)} \circ \dots \circ \sigma_{(a_n)}$.


\begin{lem}\label{lem:increase_order_lemma}

If $x \le y$, then  $\sigma_{(x, y)} =  \sigma_{(x)} \circ \sigma_{(y)}$.   

\smallskip

\noindent  If $a_0 < a_1 < \dots < a_n$, then $\sigma_{(a_0, a_1, \dots a_n)} = \sigma_{(a_0)} \circ \sigma_{(a_1)} \dots  \circ \sigma_{(a_n)}$.

\end{lem}

\begin{proof}   
If $x \le y$, then $\sigma_{(x, y)} =  \sigma_{(x)} \circ \sigma_{(y)}$ was just verified.     
This covers the base case for 
$\sigma_{(a_0, a_1)} =  \sigma_{(a_0)} \circ \sigma_{(a_1)}$.   
By induction, suppose  $\sigma_{(a_0, a_1, \dots a_{k-1})} =$ 
$\sigma_{(a_0)} \circ \sigma_{(a_1)} \dots  \circ \sigma_{(a_{k-1})}.$  Observe that 
$\sigma_{(a_0)} \circ \sigma_{(a_1)} \dots  \circ \sigma_{(a_{k-1})} \circ \sigma_{(a_{k})}$ is generated 
by swapping the $a_k$ entry and the $a_k + 1$ entry in $\sigma_{(a_0)} \circ \sigma_{(a_1)} \dots  \circ \sigma_{(a_{k-1})}.$  
Furthermore,  $a_0 < a_1 < \dots < a_{k-1} < a_k$.   
These two properties imply that $\sigma_{(a_0, a_1, \dots a_{k-1})} \circ \sigma_{(a_{k})}$  
$=  \sigma_{(a_0, a_1, \dots a_k)}.$  Lastly, the induction hypothesis implies that $\sigma_{(a_0, a_1, \dots a_k)} =$  
$\sigma_{(a_0)} \circ \sigma_{(a_1)}  \dots  \circ \sigma_{(a_{k-1})} \circ \sigma_{(a_k)}.$        
\end{proof}


In general,  $\sigma_{(y)} \circ \sigma_{(x)} \ne  \sigma_{(x, y)}$ when $x > y$.    When does equality hold?
Consider the case $x > y + 1$.   Similar to the proof in lemma \ref{commute_remark},  
during the construction of  $\sigma_{(x, y)}$, the first step swaps the $x$ entry and $x+1$ entry.   
When the second step swaps the $y$ and $y+1$ entry,  this swap doesn't move the  $x$ and $x+1$ entries because $x > y + 1$.    
Thus, if $x > y + 1$, then $\sigma_{(y)} \circ \sigma_{(x)} =  \sigma_{(x, y)}$.

\begin{table}[h]
 
\noindent  \caption{ \hskip 0.3pc $\sigma_{(k)} \circ \sigma_{(k+1)}$  \hskip 0.3pc  and  \hskip 0.3pc $\sigma_{(k+1, k)}$  } 
\label{tab:sigma_k+1_k}

\smallskip 

 \begin{tabular}{ p{1.5cm}p{2cm}p{1.75cm}p{2.75cm}p{1.75cm}}
   \hline
       $x$  & $\sigma_{(k+1)}(x)$ &  $\sigma_{(k)}(x)$ & $\sigma_{(k)} \circ \sigma_{(k+1)}(x)$  &  $\sigma_{(k+1, k)}(x)$        \\
\hline     
      $k$        & $k$      & $k+1$    & $k+1$  &  $k+2$   \\        
      $k+1$    & $k+2$  & $k$       &   $k+2$  &  $k$    \\
       $k+2$   &  $k+1$ & $k+2$       & $k$  &  $k+1$    \\
       \end{tabular}

 \end{table}

The remaining case is $\sigma_{(k)} \circ \sigma_{(k+1)}$. Both $\sigma_{(k)} \circ \sigma_{(k+1)}$ and $\sigma_{(k+1, k)}$  
are the identity map outside of $\{k, k+1, k+2\}$.  
However, $\sigma_{(k)} \circ \sigma_{(k+1)} \ne   \sigma_{(k+1, k)}$ on $\{k, k+1, k+2\}$, as 
shown in   table \ref{tab:sigma_k+1_k}.


\medskip 

A sequence of functions $f_n : \mathbb{N} \rightarrow \mathbb{N}$ {\it pointwise converges} 
to a function $f : \mathbb{N} \rightarrow \mathbb{N}$ if for each $m \in \mathbb{N}$, there 
exists $N_m$ such that for all $n \ge N_m$, $f_n(m) = f(m)$.   For this particular $m$, we write 
${\underset{n \to \infty} \lim} f_n(m)  = f$.  When a sequence of functions $\{f_n\}$ pointwise 
converges, we write ${\underset{n \to \infty} \lim} f_n  = f$.

\begin{defn}\label{defn:seq_rearrange}  \hskip 1pc  {\it Sequence Rearrangement}

\noindent  Let $A = (a_0, a_1, \dots, a_k \dots)$ 
be a sequence of elements from $\mathbb{N}$,  where there are no repeats i.e., $j \ne k$ implies that $a_j \ne a_k$.   
To construct $\sigma_{A}:  \mathbb{N} \rightarrow \mathbb{N}$, start with $\sigma_{\emptyset} = (0, 1, 2, 3, \dots)$ where   
$\sigma_{\emptyset}$ is the identity permutation.   Similar to the above,  $\sigma_{A}$ is constructed 
iteratively using each element of $A$ to generate a transposition.
For $a_0$ swap the $a_0$ entry and the $a_0 + 1$ entry of  $\sigma_{\emptyset}$.   Thus, 
$\sigma_{(a_0)}(n) = n$ when $n \notin \{a_0, a_0+1\}$.  Also,  $\sigma_{(a_0)} (a_0) = a_0 + 1$ and  $\sigma_{ (a_0)}(a_0 + 1) = a_0$.  
Inductively, suppose $\sigma_{(a_0, a_1, \dots, a_{k-1})} = (b_0, b_1, b_2, \dots, b_k, \dots)$.   Then to construct $\sigma_{(a_0, a_1, \dots, a_{k-1}, a_k)}$  from  $\sigma_{(a_0, a_1, \dots, a_{k-1})}$ swap $b_{a_k}$ and $b_{a_k+1}$.    

\smallskip 

Define $\sigma_{A} = {\underset{n \to \infty} \lim} \sigma_{(a_0, a_1, \dots, a_{n-1}, a_{n})}$.  We verify that 
$f_n =  \sigma_{(a_0, a_1, \dots, a_{n-1}, a_{n})}$ pointwise converges so that $\sigma_{A}$ is well-defined.    
Since the elements of $A$ are distinct, this implies that for any $m$, there exists an $N$ such that all 
elements $a_j  \in A$ and $a_j \le m$ implies $j \le N$.  In other words, the $a_j \le m$ have already appeared.   
Thus, there will be no more swaps for entries  $\le m$, so 
${\underset{n \to \infty} \lim} \sigma_{(a_0, a_1, \dots, a_{n-1}, a_{n})} (m)$ exists for each $m$.    

\end{defn}


\begin{lem}\label{commute_order_lemma}

Let sequence $A = (a_0, a_1, \dots, a_k \dots)$ with no repeats.  
Construct $(b_0, b_1, \dots b_n)$, such that $b_k < b_{k+1}$, as a  
rearrangement of $(a_0, a_1, \dots a_n)$.  That is, as sets $\{b_0, b_1, \dots b_n\} = \{a_0, a_1, \dots, a_n\}.$ 
If $|a_j - a_k| > 1$ for each pair $j \ne k$, then 
$\sigma_{(a_0, a_1, \dots a_n)} = \sigma_{(a_0)} \circ \sigma_{(a_1)} \dots \circ \sigma_{(a_n)} =$ 
$\sigma_{(b_0)} \circ \sigma_{(b_1)} \dots \circ \sigma_{(b_n)}$ 
$= \sigma_{(b_0, b_1, \dots b_n)}$ for each $n$.  
\end{lem}

\begin{proof}  
The order of $(a_0, a_1, \dots, a_n)$ can be rearranged so that the transpositions  $\sigma_{(b_k)}$ 
are applied in increasing order $b_k <  b_{k+1}$.  Lemma \ref{commute_remark} implies that  
$\sigma_{(a_0)} \circ \sigma_{(a_1)} \dots \circ \sigma_{(a_n)} = \sigma_{(b_0)} \circ \sigma_{(b_1)} \dots \circ \sigma_{(b_n)}$.   
Lemma \ref{lem:increase_order_lemma} implies $\sigma_{(b_0, b_1, \dots b_n)} = \sigma_{(b_0)} \circ \sigma_{(b_1)} \dots \circ \sigma_{(b_n)}$.  
Definition \ref{defn:seq_rearrange} and $|a_j - a_k| > 1$ for each pair $j \ne k$ implies that  
$\sigma_{(a_0, a_1, \dots a_n)} = \sigma_{(a_0)} \circ \sigma_{(a_1)} \dots \circ \sigma_{(a_n)}$.     
\end{proof}



\begin{defn}\label{defn:set_rearrange}  \hskip 1pc  {\it Set Rearrangement}  

\noindent Instead of starting with a sequence,  a set $A \subset {\mathbb{N}}$  is ordered into a 
sequence $(a_0, a_1, \dots, a_k \dots)$  according to $a_k < a_{k+1}$.  
This means $a_0$ is the least element of  $A$; $a_1$ is the least element of $A - \{a_0\}$; 
inductively, $a_{k+1}$ is the least element of $A - \{a_0, a_1, \dots, a_k\}$.    Define  
${\overset{n} {\underset{k = 0} \prod}} \sigma_{a_k}$ 
$=  \sigma_{(a_0)} \circ  \sigma_{(a_1)} \circ  \dots \circ  \sigma_{(a_n)}.$
Define $\sigma_{A} = {\underset {n \to \infty} \lim}$ ${\overset{n} {\underset{k = 0} \prod}} \sigma_{a_k}$.  
For the same reason as in definition \ref{defn:seq_rearrange}, this limit exists and 
hence $\sigma_{A}$ is well-defined.   If $A$ is an infinite set, 
lemma \ref{lem:increase_order_lemma} implies   
$\sigma_{A} = {\underset{n \to \infty} \lim} \sigma_{(a_0, a_1, \dots, a_{n-1}, a_{n})}$.       
\end{defn}  

For example, $\sigma_{\mathbb{E}} = (1, 0, 3, 2, 5, 4, 7, 6,  \dots)$, and 
$\sigma_{\mathbb{O}} = (0, 2, 1, 4, 3, 6, 5, 8, 7, \dots)$. 

\smallskip 

\begin{lem}\label{lem:sigma_map_1_to_1}
For any $A \subset \mathbb{N}$, 
the map $A \longmapsto \sigma_{A}$ is one-to-one.
\end{lem}

\begin{proof}
Suppose $A, B \subset \mathbb{N}$ and $A \ne B$.  
As defined in \ref{defn:set_rearrange}, 
let $A = \{a_1, a_2, \dots \}$, where $a_k < a_{k+1}$ 
and let  $B = \{b_1, b_2, \dots \}$, where $b_k < b_{k+1}$. 
Let $m$ be the smallest index such that $a_m \ne b_m$.  
W.L.O.G., suppose $a_m < b_m$.

Claim:  $\sigma_{A}(a_m) \ne \sigma_B(a_m)$.   
If $m = 1$, then $\sigma_{A}(a_m) = a_m + 1$, 
and $\sigma_{B}(a_m) = a_m$.  
For the other case $m > 1$, $a_{m-1} = b_{m-1}$ and $b_m - b_{m-1} > 1$. 
The three conditions  $a_{m-1} = b_{m-1}$, $b_m - b_{m-1} > 1$ 
and $a_m < b_m$ together imply that $\sigma_B(a_m) \le a_m$. 
In the next paragraph, we verify that $\sigma_{A}(a_m) = a_m + 1$ 
and this completes the proof.

For the case $a_m - a_{m-1} = 1$, when $\sigma_{(a_{m-1})}$ is applied, $a_m$ and $a_{m-1}$ are swapped so 
that $\sigma_{A}(a_m - 1) = a_m$.   After this swap, the values $a_{m-1}$ and $a_{m} + 1$ are swapped 
so that at the $a_m$ index, $\sigma_{A}(a_m) = a_{m} + 1$.  
For the case  $a_m - a_{m-1} > 1$, no swap occurs between indices $a_m$ and $a_{m} - 1$.  
At the $a_m$ index, the values $a_{m}$ and $a_{m} + 1$ are swapped so that 
$\sigma_{A}(a_m) = a_{m} + 1$.  
\end{proof}

$\sigma_{A}: \mathbb{N} \rightarrow \mathbb{N}$ is one-to-one, as $\sigma_{A}$ 
is a composition of transpositions according to definition \ref{defn:set_rearrange}.  
$\sigma_{A}$ is not always onto.  $\sigma_{\mathbb{N}} = (1, 2, 3, 4, 5, 6, \dots)$;
that is,  $\sigma_{\mathbb{N}}(n) = n + 1$.   $A$ is called a {\it tail set} 
if there exists an $N$ such that $m \ge N$ implies $m \in A$.  
During the construction of $\sigma_{A}$ per definition \ref{defn:set_rearrange}, 
if $a_m$ is not in $A$, then no element less than $a_m$ swaps beyond index $a_m$.  
This observation implies remark \ref{rem:tail_set_not_onto}. 

\smallskip 

\begin{rem}\label{rem:tail_set_not_onto}

\noindent $A$ is a tail set if and only if  
$\sigma_{A}: \mathbb{N} \rightarrow \mathbb{N}$ is not onto.   
\end{rem}

\smallskip 

In \cite{post}, Post states that the complement of any finite subset (i.e., tail set) of $\mathbb{N}$ is Turing computable.  
Example \ref{ex:TM_tail_set} describes a Turing machine that computes a tail set $A$; this machine is 
provided for the group theorist who may not be as familiar with Turing machines as a computability theorist. 

\smallskip 

\begin{example}\label{ex:TM_tail_set}  \hskip 1pc  {\it A Turing Machine that computes a Tail Set}

\noindent  Let $n$ consecutive $1$'s on the tape, followed by a blank, correspond to the non-negative integer $n$.   
The machine starts in state $q_0$.   If the machine reads a 1, when in state $q_k$ when $k < M$, then 
it moves one tape square to the right and moves to state $q_{k+1}$.     
If the machine reads a blank in state $q_k$, then if $k \in A$, then it writes a $1$ in this tape square and halts.  
Otherwise, $k \notin A$ and the machine writes a $0$ in the tape square and halts.     
If the machine reaches state $q_M$, then it stays in state $q_M$ while still reading a 1 and moves one tape square to the right.   
If it reads a blank while in state $q_M$, then it writes a 1 and halts.  
\end{example}

%% file: incomputable_permutations.tex
\section{Computable \& Incomputable Permutations  }\label{sect:incomputable_permutations}

Let $\sigma:  \mathbb{N} \rightarrow \mathbb{N}$ be a permutation.   
$\sigma$ is a {\it Turing computable permutation} if there exists a Turing machine $M$ such that
if $M$'s tape has input $n$ (non-negative integer), then $M$ halts with $\sigma(n)$ on its tape.  
Define $\mathcal{C}({\mathbb{N}}) = \{\sigma \in$  Sym$({\mathbb{N}}): \sigma$ is Turing computable$\}$. 

Following \cite{bhattacharjee}, for any $\sigma \in$ Sym$(\mathbb{N})$, 
define the {\it support} of $\sigma$ as  $supp(\sigma) = \{ n \in \mathbb{N}: \sigma(n) \ne n \}$.  
When the support of $\sigma$ is a finite set, we say permutation $\sigma$ is  {\it finitary}. 
This means for each finitary $\alpha$, there exists an  
$n \in \mathbb{N}$ such that $\alpha(k) = k$ when $k \ge n$. 
The {\it finitary permutations} are 
FS$(\mathbb{N}) = \{\sigma \in$ Sym$(\mathbb{N}):$ \hskip 0.1pc  $\sigma $ \verb|is finitary|$\}$. 

\smallskip 

\begin{rem}
The finitary permutations FS$(\mathbb{N})$ are a proper subset of $\mathcal{C}({\mathbb{N}})$.
\end{rem}

\begin{proof}
Let $\alpha$ be a finitary permutation.  A table can be constructed with $n$ columns and 
one row that represents $\alpha$:  
For each $0 \le j < n$, the value of $\alpha(j)$ is stored in the $j$th column of the table.  
$\alpha$ can be computed from this finite table, so any finitary permutation is Turing computable.  

Finally, we verify that FS$(\mathbb{N})$ is a proper subset.  
Define $\sigma(n) = n + 1$ if $n$ is even, and $\sigma(n) = n - 1$ if $n$ is odd.  
$\sigma$ is a Turing computable permutation, and $\sigma$ is not finitary.   
\end{proof}

\begin{lem}
$\mathcal{C}({\mathbb{N}})$ is a countable subgroup of Sym$({\mathbb{N}})$.  
\end{lem}

\begin{proof}

Closure:  suppose $\alpha, \beta$ lie in $\mathcal{C}({\mathbb{N}})$.  Let $M_{\beta}$ be a Turing machine 
that computes $\beta$.  Let $M_{\alpha}$ denote a Turing machine that computes $\alpha$.  
Build machine $M_{\beta, \alpha}$ as follows:  
Execute $M_{\alpha}$ on input $n$ until it halts with $\alpha(n)$ on the tape.  
Then execute $M_{\beta}$ on input  $\alpha(n)$, which halts with $\beta \circ \alpha(n)$ on the tape.

Inverse:  we show that $\beta^{-1}$ is in $\mathcal{C}({\mathbb{N}})$, by constructing a machine $M_{\beta^{-1}}$ 
from machine $M_{\beta}$.  The value $\beta^{-1}(n)$ is computed as follows:  
start with inputs $k = 0, 1, 2, 3, \dots$.  After $M_{\beta}$ halts
with $\beta(k)$, machine $M_{\beta^{-1}}$ checks if $\beta(k)$ equals $n$.  
For some $k_n$, $M_{\beta}$ will halt with $n$ because $\beta$ is a permutation of  $\mathbb{N}$.  
$k_n$ is written on the tape by $M_{\beta^{-1}}$ as the output and $M_{\beta^{-1}}$ halts.

The Turing machines are countable, so $\mathcal{C}({\mathbb{N}})$ is countable. 
\end{proof}

\begin{rem}
The finitary permutations FS$(\mathbb{N})$ are a proper subgroup of $\mathcal{C}({\mathbb{N}})$.
\end{rem}

\begin{proof}  Closure:  If $\sigma$ and $\beta$ are finitary permutations, 
then $\sigma \circ \beta$ and $\beta \circ \sigma$ are finitary. 
Inverse:  If $\sigma$ is finitary, then $\sigma^{-1}$ is finitary.  
\end{proof}

\medskip 

Define the {\it Turing incomputable permutations} as $\mathcal{I}({\mathbb{N}}) =$  Sym$({\mathbb{N}}) -  \mathcal{C}({\mathbb{N}})$.

\medskip

\begin{rem}\label{rem:S_N_disjoint_union}
Sym$({\mathbb{N}}) = \mathcal{C}({\mathbb{N}})$  $\dcup$  $\mathcal{I}({\mathbb{N}})$.   (Disjoint union)
\end{rem}

\medskip 

Sym$({\mathbb{N}})$ is an uncountable group with cardinality $2^{\aleph_0}$.     

\medskip  

\begin{rem}\label{rem:incomputable_uncountable}
$\mathcal{I}({\mathbb{N}})$ is an uncountable set.  
\end{rem}

\begin{proof}
$\Sym({\mathbb{N}})$ is uncountable and $\mathcal{C}({\mathbb{N}})$ is countable.
\end{proof}

\medskip

Let $\langle\mathcal{I}({\mathbb{N}}) \rangle$ be the subgroup of Sym$({\mathbb{N}})$
 generated by $\mathcal{I}({\mathbb{N}})$.  
 More precisely, define $S_{\mathcal{I}({\mathbb{N}})} = \{ H  :  H$  \hskip 0.2pc   \verb|is a|  \hskip 0.2pc 
\verb|subgroup of| Sym$({\mathbb{N}})$  \hskip 0.2pc  \verb|and| \hskip 0.2pc 
$H \supseteq \mathcal{I}({\mathbb{N}})  \}$.   

\begin{equation}
 \langle\mathcal{I}({\mathbb{N}}) \rangle =  {\underset{H \in S_{\mathcal{I}({\mathbb{N}})}} \cap}  H.
\end{equation}

Since $\mathcal{C}({\mathbb{N}})$ is a proper subgroup of Sym$({\mathbb{N}})$, 
and $\mathcal{I}({\mathbb{N}}) = $  Sym$({\mathbb{N}}) -  \mathcal{C}({\mathbb{N}})$, 
then $\langle\mathcal{I}({\mathbb{N}}) \rangle = $ Sym$({\mathbb{N}})$.  
  It is instructive to directly show this, using the next two lemmas.
  

\begin{lem}\label{lem:C_composed_I_incomputable}
If $\upsilon \in \mathcal{C}({\mathbb{N}})$ 
and $\tau \in \mathcal{I}({\mathbb{N}})$, 
then $\upsilon \circ \tau \in \mathcal{I}({\mathbb{N}})$.   
\end{lem}

\begin{proof}
$\upsilon^{-1}$ is in $\mathcal{C}({\mathbb{N}})$.  
If $\upsilon \circ \tau \in \mathcal{C}({\mathbb{N}})$, 
then  $\upsilon^{-1} \circ \upsilon \circ \tau$ $\in$ $\mathcal{C}({\mathbb{N}})$ 
because $\mathcal{C}({\mathbb{N}})$ is a subgroup.  
This implies $\tau \in \mathcal{C}({\mathbb{N}})$,  which is a contradiction.  
Hence, $\upsilon \circ \tau \in \mathcal{I}({\mathbb{N}})$ since $\upsilon \circ \tau$ 
is not in $\mathcal{C}({\mathbb{N}})$. 
\end{proof}

\begin{lem}
$\mathcal{C}({\mathbb{N}}) \subset \langle\mathcal{I}({\mathbb{N}}) \rangle$.
\end{lem}

\begin{proof}
Let $\upsilon \in \mathcal{C}({\mathbb{N}})$. 
For any $\tau \in \mathcal{I}({\mathbb{N}})$, lemma \ref{lem:C_composed_I_incomputable} 
implies $\upsilon \circ \tau = \rho \in \mathcal{I}({\mathbb{N}})$.  
Hence, $\upsilon = \rho \circ \tau^{-1}$ is in $\langle\mathcal{I}({\mathbb{N}}) \rangle$ 
since $\langle\mathcal{I}({\mathbb{N}}) \rangle$ is a group.  
\end{proof}

\begin{thm}\label{thm:incomputable_permutations_eq_sym_N}
$\langle\mathcal{I}({\mathbb{N}}) \rangle = $ Sym$({\mathbb{N}})$.
\end{thm}

\begin{proof}
Sym$({\mathbb{N}}) = \mathcal{C}({\mathbb{N}})$  $\dot \cup$  $\mathcal{I}({\mathbb{N}})$ 
by remark \ref{rem:S_N_disjoint_union},
and both sets lie in $\langle \mathcal{I}(\mathbb{N}) \rangle$.
\end{proof}

Theorem \ref{thm:incomputable_permutations_eq_sym_N} can be interpreted as {\it Turing 
incomputable permutations generate infinite symmetry on} $\mathbb{N}$. 
Hence, theorem \ref{thm:incomputable_permutations_eq_sym_N}  
suggests a relationship between indeterminism and infinite symmetry.  

\bigskip 

A binary sequence is {\it Turing incomputable} if no Turing machine 
can exactly reproduce this infinite sequence of 0’s and 1’s.  
In other words, let $\mathcal{C}$ be the subset of $\{0, 1\}^{\mathbb{N}}$ 
that are Turing computable sequences.    Then  $\mathcal{I} = \{0, 1\}^{\mathbb{N}} - \mathcal{C}$ 
are the Turing incomputable binary sequences.  
Lemma \ref{lem:incomputable_sequences_generate_incomputable_permutations} shows 
that each Turing incomputable binary sequence 
generates (via definition \ref{defn:set_rearrange})   a Turing incomputable permutation.

\begin{lem}\label{lem:incomputable_sequences_generate_incomputable_permutations}
Suppose $A$ $\subset \mathbb{N}$.  As before, 
set $A$ is identified with  binary sequence $(a_0, a_1 \dots )$ as:
$a_i = 1$  if $i \in A$   and   $a_i = 0$ if   $i \notin A$.  
If $A$ is Turing incomputable,  then  (via definition \ref{defn:set_rearrange})  
$\sigma_A$ is a permutation and $\sigma_A$ lies in $\mathcal{I}({\mathbb{N}})$.
\end{lem}


\begin{proof}
From remark  \ref{rem:tail_set_not_onto} and example \ref{ex:TM_tail_set}, 
$\sigma_A$ is in  Sym$({\mathbb{N}}) = \mathcal{I}({\mathbb{N}})$ $\dot \cup$ $\mathcal{C}({\mathbb{N}})$.
If $\sigma_A$ lies in $\mathcal{C}({\mathbb{N}})$ then there exists Turing machine
$M_A$ such that with $n$ on the initial tape, $M_A$ halts with $\sigma_A(n)$ on the tape. 
That is, $M_A(n) = \sigma_A(n)$.  
The transpositions described in definition  \ref{defn:set_rearrange} 
can be used by machine $M_A$ to compute whether each $n$ is in $A$ or not in $A$.  
This contradicts that $A$ is Turing incomputable.  
\end{proof}


\medskip 

Set $\mathfrak{I} = \{ \sigma_{A}: A$ \hskip 0.2pc  \verb|is a| \hskip 0.2pc 
\verb|Turing| \verb|incomputable set|$\}$.   
Set $S_{\mathfrak{I}} = \{ H:$  $H$  \hskip 0.2pc \verb|is a subgroup of| Sym$({\mathbb{N}})$  
\hskip 0.2pc  \verb|and|  \hskip 0.2pc   $H \supseteq \mathfrak{I} \}$. 
The {\it incomputable rearrangements} are the elements of $\mathfrak{I}$.  
$\mathfrak{I}$ generates a permutation subgroup of  Sym$({\mathbb{N}})$:

\begin{equation}
G_{\mathfrak{I}} = {\underset{H \in S_{\mathfrak{I}} } {\cap} } H 
\end{equation}












\bigskip 

\begin{thm}\label{thm:G_I_eqs_Sym_N}
If $\mathcal{I}({\mathbb{N}}) \subset G_{\mathfrak{I}}$, then  $G_{\mathfrak{I}} =$ Sym$({\mathbb{N}})$.
\end{thm}

\begin{proof}
By assumption, $G_{\mathfrak{I}}$ contains  $\mathcal{I}({\mathbb{N}})$
and $G_{\mathfrak{I}}$ is a group.  
Therefore, $\langle\mathcal{I}({\mathbb{N}}) \rangle \subset  G_{\mathfrak{I}}$.  
Theorem \ref{thm:incomputable_permutations_eq_sym_N} completes the proof.    
\end{proof}

\medskip

\noindent  {\bf Open Questions:} \hskip 1pc  Is $\mathcal{I}({\mathbb{N}}) \subset G_{\mathfrak{I}}$?   
             Is  $\mathcal{C}({\mathbb{N}})  \subset G_{\mathfrak{I}}$?  
             These two questions are addressed in section \ref{sec:crossing}.




%% file: bi_immune_permutations.tex
\section{Bi-immune Permutations}\label{sect:bi_immune_permutations}




Via the standard identification between binary sequences and subsets of $\mathbb{N}$, 
define $\mathcal{B} = \{ b \in \{0, 1\}^{\mathbb{N}}: b$ \verb|is bi-immune|  $\}$.

\begin{cor}\label{cor:bi_immune_set_implies_permutation}
\noindent  If $A$ is a bi-immune set, then $\sigma_{A}:  \mathbb{N} \rightarrow \mathbb{N}$ is a permutation.   

\end{cor}

\medskip 

Whenever $A$ is a bi-immune set, $\sigma_{A}$ is called a {\it bi-immune permutation}.

\begin{rem}
$\mathcal{I} - \mathcal{B}$ is an uncountable subset of $\{0, 1\}^{\mathbb{N}}$.  
\end{rem}

\begin{proof}
Let $p_1 = 2$, and let $p_i$ be the $i$th prime number. 
Let $(b_1 b_2 \dots)$ be an element of $\mathcal{B}$.  
Let $(c_1 c_2 \dots)$ be an element of $\mathcal{C}$. 
Define a new binary sequence where $a_i = b_i$ if $i$ is not prime.  If 
$i$ is prime, then  define $a_i = c_k$, where $i = p_k$ for some $k$.  
Binary sequence $(a_1 a_2 \dots)$ is Turing incomputable, but it is not bi-immune because the 
subsequence $(a_{2}, a_{3}, a_{5}, \dots a_{p_i}, a_{p_{i+1}}, \dots )$ 
is a Turing computable subsequence of $(a_1 a_2 \dots)$.    
For each distinct element $(b_1 b_2 \dots)$ in $\mathcal{B}$, the corresponding incomputable 
sequence $(a_1 a_2 \dots)$ is unique.  Since  $\mathcal{B}$ is uncountable, this implies
$\mathcal{I} - \mathcal{B}$ is uncountable.   
\end{proof}


Set $\mathfrak{B} = \{ \sigma_{A}: A$   \hskip 0.2pc \verb|is a bi-immune set|$\}$. 
Set $S_{\mathfrak{B}} = \{ H  :  H$  \hskip 0.2pc   \verb|is a|  \hskip 0.2pc 
\verb|subgroup of| Sym$({\mathbb{N}})$  \hskip 0.2pc  \verb|and| \hskip 0.2pc 
$H \supseteq \mathfrak{B}  \}$. 
The bi-immune permutations $\mathfrak{B}$ generate a permutation group, 
called the {\it bi-immune symmetric group}:      

\begin{equation}
G_{\mathfrak{B}} = {\underset{H \in S_{\mathfrak{B}} } {\cap} } H
\end{equation}

\begin{lem}\label{lem:bi_immune_subgroup_incomputable}
$G_{\mathfrak{B}}$ is a subgroup of $G_{\mathfrak{I}}$.
\end{lem}

\begin{proof}
Every bi-immune binary sequence is a Turing incomputable sequence. 
\end{proof}

Let $\mathfrak{A}$ be a subset of $\mathfrak{B}$.  $\mathfrak{A}$ may contain a countable 
number of bi-immune permutations, or  uncountable number.  $\mathfrak{A}$ may contain a 
finite number of bi-immune permutations.   
Set $S_{\mathfrak{A}} = \{ H:$  \hskip 0.2pc $H$ 
\hskip 0.2pc  \verb|is a subgroup of| Sym$({\mathbb{N}})$ \hskip 0.2pc 
\verb|and| \hskip 0.2pc  $H \supseteq \mathfrak{A} \}$.   
For each $\mathfrak{A} \subset \mathfrak{B}$, define the permutation group 

\begin{equation}
G_{\mathfrak{A}} = {\underset{H \in S_{\mathfrak{A}} } {\cap} } H
\end{equation}

A longer term goal is to answer questions about the subgroups $G_{\mathfrak{A}}$ 
of the bi-immune symmetric group $G_{\mathfrak{B}}$.  
For each  $\mathfrak{A}$,  what is the structure of $G_{\mathfrak{A}}$ as a subgroup of $G_{\mathfrak{B}}$? 
Can $G_{\mathfrak{A}} = G_{\mathfrak{B}}$ when $\mathfrak{A}$ is countably infinite or finite?
What is the structure  of $G_{\mathfrak{B}}$ as a subgroup of Sym$({\mathbb{N}})$?

Lemma \ref{lem:sigma_map_1_to_1} implies $G_{\mathfrak{B}}$ is uncountable 
because the bi-immune sets are uncountable.  Since the number of distinct $\mathfrak{A}$ is uncountable, 
do the various $\mathfrak{A}$ generate an uncountable number of subgroups of the 
bi-immune symmetric group?  
Is $G_{\mathfrak{B}}$ a proper subgroup of $G_{\mathfrak{I}}$? 

First, we show for any $i$ that $\sigma_{(i)}$ lies in the bi-immune symmetric group.   
Recall that $\sigma_{(i)}$ permutes $i$ and $i+1$, 
and all other natural numbers are fixed points of $\sigma_{(i)}$. 
From $\sigma_{(i)}$ in $G_{\mathfrak{B}}$, we show that for any 
$i, j \in \mathbb{N}$, the transposition $(i$ $j)$ lies in  $G_{\mathfrak{B}}$.  
Since the transpositions $\{ (i$ \hskip 0.2pc  $j): 0 \le i < j < n  \}$  generate the finite 
symmetric group Sym$(\{0, 1, \dots, n-1\})$, we see that the finitary symmetric group 
is a subgroup of the bi-immune symmetric group.  
We proceed with the details.

For each $r \in \mathbb{N}$ and $A \subset \mathbb{N}$, 
define $A_{>r} = \{ a \in A: a > r \}$ and $A_{\le r} = \{ a \in A: a \le r \}$.   
Since $A_{\le r}$ is finite, the following remark follows from lemma \ref{bi_finite_set}.

\begin{rem}\label{rem:A_le_r_bi_immune}
If $A$ is a bi-immune set, then $A_{>r} = A - A_{\le r}$ is bi-immune.  
\end{rem}

\medskip 

Remark \ref{rem:A_gt_r_union_i} is an immediate consequence of 
remark \ref{rem:A_le_r_bi_immune} and lemma \ref{bi_finite_set}.

\begin{rem}\label{rem:A_gt_r_union_i}
If $A$ is a bi-immune set,  then the set $A_{>r} \cup \{i\}$ is also bi-immune whenever $i < r$.
\end{rem}

\begin{lem}
For any $i$, permutation $\sigma_{(i)}$ lies in  $G_{\mathfrak{B}}$.  
\end{lem}

\begin{proof}
Fix $i \in \mathbb{N}$.   Let $A$ be a bi-immune set.  Choose $r > i$.  Since $A_{>r}$ is bi-immune, 
$\sigma_{A_{>r}}$ is a permutation in subgroup $G_{\mathfrak{B}}$.   Hence,  ${\sigma}^{-1}_{A_{>r}}$ 
is in $G_{\mathfrak{B}}$, and $\sigma_{A_{>r} \cup \{i\}}$ is in $G_{\mathfrak{B}}$. 
Lastly, $\sigma_{(i)} = \sigma_{A_{>r} \cup \{i\}}  \circ  {\sigma}^{-1}_{A_{>r}}$  is in $G_{\mathfrak{B}}$.
\end{proof}

The following example illustrates why for any $i < j$, 
transposition  $(i$ \hskip 0.2pc $j)$  is in  $G_{\mathfrak{B}}$.  
To show that transposition $(1$ \hskip 0.2pc $4)$ is in $G_{\mathfrak{B}}$, 
we start with identity permutation  $\sigma_{\emptyset} = (0, 1, 2, 3, 4, 5, \dots)$ 
and compose with the appropriate $\sigma_{(i)}$ to move the $4$ to the location of $1$.  
$\sigma_{(3)}  = (0, 1, 2, 4, 3, 5, \dots)$. 
$\sigma_{(3)} \circ \sigma_{(2)}  = (0, 1, 4, 2, 3, 5, \dots)$. 
$\sigma_{(3)} \circ \sigma_{(2)} \circ \sigma_{(1)}  = (0, 4, 1, 2, 3, 5, \dots)$.

Then we compose the appropriate $\sigma_{(i)}$ to move $1$ to the original location of $4$ in $\sigma_{\emptyset}$.  
$\sigma_{(3)} \circ \sigma_{(2)} \circ \sigma_{(1)} \circ \sigma_{(2)}  = (0, 4, 2, 1, 3, 5, \dots)$.
$\sigma_{(3)} \circ  \sigma_{(2)} \circ \sigma_{(1)} \circ \sigma_{(2)} \circ \sigma_{(3)}$ 
$= (0, 4, 2, 3, 1, 5, \dots)$.

\begin{lem}\label{lem:transposition_in_G}
For $i < j$, transposition $(i$ \hskip 0.2pc $j)$ lies in $G_{\mathfrak{B}}$.   
\end{lem}

\begin{proof}
$(i$ \hskip 0.2pc $j) =  \sigma_{(j-1)} \circ \sigma_{(j-2)}$ 
                         \hskip 0.2pc  $\dots$ \hskip 0.2pc  $\sigma_{(i+1)} \circ \sigma_{(i)} \circ \sigma_{(i+1)}$ 
                         \hskip 0.2pc  $\dots$ \hskip 0.2pc  $\sigma_{(j-2)} \circ \sigma_{(j-1)} \circ \sigma_{\emptyset}$. 
\end{proof}

Following \cite{bhattacharjee}, for any $\sigma \in$ Sym$(\mathbb{N})$, 
define the {\it support} of $\sigma$ as 
$supp(\sigma) = \{ n \in \mathbb{N}: \sigma(n) \ne n \}$.  
When the support of $\sigma$ is a finite set, we say $\sigma$ is  {\it finitary}. 
The {\it finitary permutations} are 
FS$(\mathbb{N}) = \{\sigma \in$ Sym$(\mathbb{N}):$ \hskip 0.1pc  $\sigma $ \verb|is finitary|$\}$. 
FS$(\mathbb{N})$ is called the {\it finitary symmetric group} on $\mathbb{N}$.

\begin{lem}\label{lem:BSym}
The finitary symmetric group FS$(\mathbb{N})$ is a proper subgroup of 
the bi-immune symmetric group $G_{\mathfrak{B}}$.  
\end{lem}

\begin{proof}
Observe that $G_{\mathfrak{B}}$ contains permutations that are not finitary.  
A bi-immune set $A$ is infinite, so $\sigma_A$ has infinite support.
The composition of two finitary permutations is finitary, so $\circ$ is closed in  FS$(\mathbb{N})$.
For a large enough $n$, any finitary permutation lies in a finite subgroup of 
FS$(\mathbb{N})$ isomorphic to  Sym$(\{0, 1, \dots, n-1\})$.  
Any finitary permutation chosen from Sym$(\{0, 1, \dots, n-1\})$  is a finite composition of 
transpositions selected from the set $\{ (i$ \hskip 0.2pc  $j): 0 \le i < j < n  \}$.
Lemma \ref{lem:transposition_in_G} completes the proof.
\end{proof}

\begin{defn}
A  $k$-tuple $(a_1, \dots a_k)$ has pairwise distinct elements 
if  $a_i \ne a_j$ whenever $i \ne j$.   A subgroup $H$ of Sym$(\mathbb{N})$ is $k${\it-transitive} 
if for any two $k$-tuples $(a_1, \dots a_k)$ and $(b_1, \dots b_k)$ each with pairwise 
distinct elements, there exists at least one permutation $\sigma$ in $H$ such that 
$\sigma(a_i) = b_i$ for $1 \le i \le k$.  
A subgroup $H$ is {\it highly transitive} if $H$ is $k$-transitive for all $k \in \mathbb{N}$.  
\end{defn}

\begin{rem}
The bi-immune symmetric group is highly transitive.  
\end{rem}

\begin{proof}
The finitary symmetric group is highly transitive.
\end{proof}

%% file: permutation_metric.tex

\subsection{ $G_{\mathfrak{B}}$ is Dense in Sym$(\mathbb{N})$  }

Following \cite{cameron}, the remainder of this section defines a metric $d$ 
on Sym$(\mathbb{N})$ compatible with the pointwise convergence topology.  
For any $\sigma, \tau \in$ Sym$(\mathbb{N})$, define $\rho: $ 
Sym$(\mathbb{N}) \times$  Sym$(\mathbb{N}) \rightarrow \mathbb{R}$ for the  three cases:

\begin{itemize}

\item[(a)]{  If $\sigma = \tau$, set $\rho(\sigma, \tau) = 0$. }

\item[(b)]{  If $\sigma(0) \ne \tau(0)$, set $\rho(\sigma, \tau) = 1$. }

\item[(c)]{  If  $\sigma(0) = \tau(0)$ and $\sigma \ne \tau$, set $\rho(\sigma, \tau) = 2^{-j}$, 
             where $\sigma(i) = \tau(i)$ for all $i$ such that $0 \le i < j$. }

\end{itemize}

For any $\sigma, \tau \in$ Sym$(\mathbb{N})$, 
define $d: $ Sym$(\mathbb{N}) \times$  Sym$(\mathbb{N}) \rightarrow \mathbb{R}$ as
$d(\sigma, \tau) = \max \{\rho(\sigma, \tau),$ $\rho(\sigma^{-1}, \tau^{-1}) \}$.  
It is straightforward to verify that $\big{(}\Sym(\mathbb{N}), d\big{)}$. 
is a complete metric space.

\smallskip 

\begin{rem}
The bi-immune symmetric group   $G_{\mathfrak{B}}$   is dense in   $\big{(}\Sym(\mathbb{N}), d\big{)}$. 
\end{rem}

\begin{proof}
For any $\tau$ in Sym$(\mathbb{N})$, there exists a sequence of $\tau_n$ in FS$(\mathbb{N})$ such that $\tau_n(k) = \tau(k)$ 
and ${\tau_n}^{-1}(k) = \tau^{-1}(k)$ for all $k \le n$.  
\end{proof}

\medskip

\begin{rem}\label{rem:closed_iff_sym_N}
$G_{\mathfrak{B}}$ is closed in $\big{(}$Sym$(\mathbb{N}), d\big{)}$ if and only if  
$G_{\mathfrak{B}}$   $=$   Sym$(\mathbb{N})$.
\end{rem}

%% file: normal_subgroup.tex

\begin{lem}
The finitary symmetric group  is a normal subgroup of the bi-immune symmetric group.  
\end{lem}

\begin{proof}
Let $\tau$ in  FS$(\mathbb{N})$ and $\sigma \in G_{\mathfrak{B}}$.  $\tau$ has finite support so there exists 
$\kappa > 0$ such that $\tau(n) = n$ when $n \ge \kappa$.  
Set $m = \max \{ \tau(0), \tau(1), \dots, \tau(\kappa) \}$.  
We show that $\sigma \tau \sigma^{-1}$ has finite support.

Set $m_{\sigma} = \max$  $\{\sigma(1), \sigma(2) \dots, \sigma(\kappa) \}$.  
If $n \le m_{\sigma}$, then $\tau \circ \sigma^{-1}(n) \in \{0, 1, \dots, m\}$.  
If $n > m_{\sigma}$, then $\sigma^{-1}(n) \notin \{0, 1, \dots, \kappa \}$, so 
$\tau \circ \sigma^{-1}(n) = \sigma^{-1}(n)$.  
If $n > m_{\sigma}$, then $\sigma \circ \tau \circ \sigma^{-1}(n) = n$.  
\end{proof}

%% file: bi_immune_infinite_cycles.tex
\section{Bi-immune Infinite Cycles}
\label{sec:infinite_cycles}

Every generator $\sigma_A$ of $G_{\mathfrak{B}}$ is a product of \emph{disjoint
finite} consecutive cycles, and $\sigma_A^{-1}$ has the same orbits.
Definition \ref{defn:orbit} defines \emph{orbit}.
Lemma \ref{lem:orbits_are_blocks} identifies the orbits of $\sigma_A$.  Thus
$G_{\mathfrak{B}}$ is generated by permutations all of whose orbits are finite.

A group generated by finite-orbit elements need not consist only of finite-orbit
elements.  If every element of $G_{\mathfrak{B}}$ has only finite
orbits, then $G_{\mathfrak{B}}$ is a proper subgroup of $\Sym(\mathbb{N})$,
because $\Sym(\mathbb{N})$ contains permutations with infinite orbits.
Hence, the following question is a necessary condition for
$G_{\mathfrak{B}} = \Sym(\mathbb{N})$:

\begin{center}
\emph{Does $G_{\mathfrak{B}}$ contain a permutation with an infinite orbit?}
\end{center}

Theorem \ref{thm:complement_infinite_cycle} explicitly answers this question:
for \emph{every} bi-immune $A$,  $\sigma_A \circ \sigma_{\overline{A}}$
has  \textit{one infinite cycle} on all of $\mathbb{N}$.
Sections \ref{sec:blocks} and \ref{sec:join}
define the block and join language that help analyze a permutation
generated by sequence rearrangement (Definition \ref{defn:seq_rearrange}).


\subsection{Orbits, Cycles, and Support}
\label{sec:orbits}

\begin{defn}[Orbit]
\label{defn:orbit}
Let $\sigma \in \Sym(\mathbb{N})$.  For $j \in \mathbb{Z}$, let $\sigma^{j}$
be the $j$th power of $\sigma$ in the group $\Sym(\mathbb{N})$.  Hence,
$\sigma^{0}$ is the identity permutation,
$\sigma^{j} = \sigma \circ \sigma^{j-1}$ for $j \ge 1$,
and $\sigma^{-j} = (\sigma^{-1})^{j}$ for $j \ge 1$.
The \textit{orbit} of $k \in \mathbb{N}$ under $\sigma$ is the set
\[
  O_{\sigma}(k) \;=\; \bigl\{\, \sigma^{j}(k) \;:\; j \in \mathbb{Z} \,\bigr\} .
\]
An \textit{orbit of} $\sigma$ is a subset of $\mathbb{N}$
of the form $O_{\sigma}(k)$ for some $k \in \mathbb{N}$.
$O_{\sigma}(k)$ is called \textit{finite} if it has a finite number of elements;
otherwise, $O_{\sigma}(k)$ is \textit{infinite}.
\end{defn}

\begin{rem}[The orbits of $\sigma$ partition $\mathbb{N}$]
\label{rem:orbits_partition}
Define $k$ ${\overset{\sigma}\sim}$ $l$ to mean that $l = \sigma^{j}(k)$ for some
$j \in \mathbb{Z}$.   ${\overset{\sigma}\sim}$ is an equivalence relation on
$\mathbb{N}$:  ${\overset{\sigma}\sim}$  is reflexive because $\sigma^{0}$ is the identity;
symmetric because $l = \sigma^{j}(k)$ implies $k = \sigma^{-j}(l)$;
and transitive because $\sigma^{j} \circ \sigma^{i} = \sigma^{i+j}$.
The equivalence classes of ${\overset{\sigma}\sim}$ are
the orbits of $\sigma$, so the orbits of $\sigma$ partition
$\mathbb{N}$, and $l \in O_{\sigma}(k)$ implies $O_{\sigma}(l) = O_{\sigma}(k)$.

Furthermore $\sigma$ and $\sigma^{-1}$ have the same orbits, because
$(\sigma^{-1})^{j} = \sigma^{-j}$, so the two sets of powers
$\{\sigma^{j} : j \in \mathbb{Z}\}$ and
$\{(\sigma^{-1})^{j} : j \in \mathbb{Z}\}$ are the same set.
Hence $O_{\sigma^{-1}}(k) = O_{\sigma}(k)$ for every $k$.
\end{rem}

\begin{rem}[Finite and infinite orbits]
\label{rem:orbit_dichotomy}
Let $\sigma \in \Sym(\mathbb{N})$ and $k \in \mathbb{N}$.

\medskip
\noindent\textit{(a)}\quad Suppose $O_{\sigma}(k)$ is finite.  Then the points
$\sigma^{j}(k)$, $j \ge 0$, cannot all be distinct, so
$\sigma^{i}(k) = \sigma^{j}(k)$ for some $0 \le i < j$; applying $\sigma^{-i}$
yields $\sigma^{j-i}(k) = k$ with $j - i \ge 1$.  Let $r \ge 1$ be least with
$\sigma^{r}(k) = k$.  Then
\[
  O_{\sigma}(k) \;=\; \bigl\{ k, \; \sigma(k), \; \dots, \;
  \sigma^{r-1}(k) \bigr\} ,
\]
these $r$ points are distinct, and $\sigma$ restricted to $O_{\sigma}(k)$ is the
$r$-cycle $k \mapsto \sigma(k) \mapsto \cdots \mapsto \sigma^{r-1}(k) \mapsto
k$.  In this case $O_{\sigma}(k)$ is called a \textit{cycle of length} $r$, or
an \textit{$r$-cycle}, of $\sigma$.  When $r = 1$ the orbit is the singleton
$\{k\}$ and $k$ is a \textit{fixed point} of $\sigma$.

\medskip
\noindent\textit{(b)}\quad Suppose $O_{\sigma}(k)$ is infinite.  If
$\sigma^{i}(k) = \sigma^{j}(k)$ for some $i \ne j$ in $\mathbb{Z}$, then
$\sigma^{r}(k) = k$ for $r = |j - i| \ge 1$, and the argument of (a) would make
$O_{\sigma}(k)$ finite.  The points in the orbit
\begin{equation}
  \dots, \;\; \sigma^{-2}(k), \;\; \sigma^{-1}(k), \;\; k, \;\; \sigma(k),
  \;\; \sigma^{2}(k), \;\; \dots
  \label{eq:doubly_infinite_list}
\end{equation}
are all distinct.  In this
case $O_{\sigma}(k)$ is called an \textit{infinite cycle} of $\sigma$:  the
orbit is a doubly infinite list \eqref{eq:doubly_infinite_list} in which no
point repeats, and $\sigma$ moves each point in the list to the next one on its right.
\end{rem}

\begin{rem}[Conjugation maps orbits to orbits]
\label{rem:conjugate_orbits}
Let $\sigma, \alpha \in \Sym(\mathbb{N})$.  Set
$\rho = \alpha \circ \sigma \circ \alpha^{-1}$.  Then
$\rho^{j} = \alpha \circ \sigma^{j} \circ \alpha^{-1}$ for every $j \in \mathbb{Z}$, so
\[
  O_{\rho}\bigl(\alpha(k)\bigr) \;=\; \alpha\bigl(O_{\sigma}(k)\bigr)
  \qquad\text{for every } k \in \mathbb{N} .
\]
Thus $\alpha$ maps the orbits of $\sigma$ onto the orbits of $\rho$ and preserves
their cardinalities.  In particular $\sigma$ has only one orbit, equal to
$\mathbb{N}$, and that orbit is an infinite cycle, if and only if the same holds
for $\rho$.

Consequently, for any $\sigma, \eta \in \Sym(\mathbb{N})$ the two compositions
$\sigma \circ \eta$ and $\eta \circ \sigma$ have the same orbit structure,
because
\[
  \eta \circ \sigma \;=\; \sigma^{-1} \circ
  \bigl( \sigma \circ \eta \bigr) \circ \sigma .
\]
\end{rem}

\begin{defn}[Support]
\label{defn:support}
The \textit{support} of $\sigma \in \Sym(\mathbb{N})$ is
\[
  \supp(\sigma) \;=\; \{\, k \in \mathbb{N} \;:\; \sigma(k) \ne k \,\} ,
\]
which by Remark \ref{rem:orbit_dichotomy} is the union of those orbits of
$\sigma$ having more than one element.
\end{defn}

\subsection{Intervals and Block partitions}
\label{sec:blocks}

Throughout the rest of this section $A \subseteq \mathbb{N}$.

\begin{defn}[Interval, maximal interval of $A$]
\label{defn:interval}
An \textit{interval} is a set of consecutive natural numbers
\[
  [m, n] \;=\; \{m, m+1, \dots, n\}, \qquad m \le n .
\]
An interval $[m,n]$ is an \textit{interval of} $A$ if $[m,n] \subseteq A$, and
it is a \textit{maximal interval of} $A$ if in addition $n + 1 \notin A$, and
either $m = 0$ or $m - 1 \notin A$.  The interval's \textit{size} is
$n-m+1$ which is the number of points in $\{m, m+1, \dots, n\}$.
\end{defn}

The maximal intervals of $A$ are pairwise disjoint, are separated from one
another by at least one element of $\overline{A}$, and their union is $A$.

\begin{defn}[Block partition $\mathcal{P}_A$]
\label{defn:block_partition}
Suppose every maximal interval of $A$ is finite.  The \textit{block partition}
$\mathcal{P}_A$ of $\mathbb{N}$ induced by $A$ has two kinds of blocks:
\begin{enumerate}
  \item[(i)] the interval $[m, n+1]$, of size $n-m+2$, for each maximal
  interval $[m,n]$ of $A$;
  \item[(ii)] the singleton $\{k\}$, for each $k$ lying in none of the
  intervals in (i).
\end{enumerate}
\end{defn}

The intervals in (i) are pairwise disjoint, because the maximal interval after
$[m,n]$ begins at some $m' \ge n+2$.  Each of them has one more point than the
maximal interval it comes from: it contains the extra point
$n+1 \in \overline{A}$ on its right.  The points in (ii) are exactly those
$k \in \overline{A}$ with $k = 0$ or $k-1 \in \overline{A}$.  Together the two
kinds cover $\mathbb{N}$ without overlap, so $\mathcal{P}_A$ is a partition of
$\mathbb{N}$, and every block is an interval, so it is an \textit{interval
partition}.

\begin{lem}[Orbits of $\sigma_A$ are the blocks of $\mathcal{P}_A$]
\label{lem:orbits_are_blocks}
Let $[m,n]$ be a maximal interval of $A$.  Then $\sigma_A$ restricted to
$[m, n+1]$ is the $(n-m+2)$-cycle
\[
  \sigma_A(k) = k+1 \;\; (m \le k \le n), \qquad \sigma_A(n+1) = m ,
\]
and $\sigma_A$ fixes every $k$ lying in none of these intervals.  Consequently
the orbits of $\sigma_A$ and of $\sigma_A^{-1}$ are  the blocks of
$\mathcal{P}_A$.
\end{lem}

\begin{proof}
This is immediate from the construction of $\sigma_A$ by successive
transposition of the points in positions $a$ and $a+1$, for $a \in A$ in
increasing order.  The swaps indexed by a single maximal interval
$m, m+1, \dots, n$ of $A$ move the point originally in position $m$ rightward
to position $n+1$, and shift each of the points originally in positions
$m+1, \dots, n+1$ one position to the left.  Swaps indexed by different maximal
intervals of $A$ act on disjoint sets of positions and so do not interfere with each other.
A point in none of these intervals is not moved.
\end{proof}

Lemma \ref{lem:orbits_are_blocks} implies that
$\sigma_A(k) \le k+1$ and $\sigma_A(k) = k+1 \iff k \in A$.

\smallskip

\begin{rem}[When $\mathcal{P}_A$ is defined]
\label{rem:intervals_finite}
If $A$ and $\overline{A}$ are both infinite, then \emph{every} maximal interval
of $A$ and of $\overline{A}$ is finite, and each of $A$ and $\overline{A}$ has
infinitely many maximal intervals.  Indeed, if $A$ had only finitely many
maximal intervals, then $A$, being infinite, would contain an infinite one, so
$A \supseteq [m, \infty)$ for some $m$ and $\overline{A}$ would be finite.  A
bi-immune set and its complement are both immune. Hence both $A$ and $\overline{A}$ are
infinite, so $\mathcal{P}_A$ and $\mathcal{P}_{\overline{A}}$ are defined
and  contain infinitely many non-singleton blocks whenever $A \in \mathfrak{B}$.
\end{rem}

\subsection{Join of Partitions and Blocks of the Join}
\label{sec:join}

\begin{defn}[Refinement]
\label{defn:refinement}
For partitions $\mathcal{P}, \mathcal{Q}$ of a set $X$, $\mathcal{P}$
\textit{refines} $\mathcal{Q}$, written $\mathcal{P} \le \mathcal{Q}$, if every
block of $\mathcal{P}$ is contained in some block of $\mathcal{Q}$.
Equivalently, every block of $\mathcal{Q}$ is a union of blocks of
$\mathcal{P}$.  One then calls $\mathcal{P}$ \emph{finer} and $\mathcal{Q}$
\emph{coarser}.
\end{defn}

The relation $\le$ is a partial order on the set of all partitions of $X$.
Definition \ref{defn:join} defines the smallest partition that is coarser than a given
finite list of partitions.

\begin{defn}[Join]
\label{defn:join}
Let $\mathcal{P}_1, \dots, \mathcal{P}_n$ be partitions of $X$.  Their
\textit{join}
\[
  \mathcal{J} = \mathcal{P}_1 \vee \mathcal{P}_2 \vee \cdots \vee \mathcal{P}_n
\]
is the unique partition $\mathcal{J}$ of $X$ such that
\begin{enumerate}
  \item[(i)] $\mathcal{P}_i \le \mathcal{J}$ for every $i$, and
  \item[(ii)] if $\mathcal{Q}$ is any partition with $\mathcal{P}_i \le
  \mathcal{Q}$ for every $i$, then $\mathcal{J} \le \mathcal{Q}$.
\end{enumerate}
That is, $\mathcal{J}$ is the \emph{finest} partition that is coarser than all
of $\mathcal{P}_1, \dots, \mathcal{P}_n$ simultaneously.  Conditions (i) and
(ii) determine $\mathcal{J}$ uniquely: if $\mathcal{J}$ and $\mathcal{J}'$ both
satisfy them, then $\mathcal{J} \le \mathcal{J}'$ and $\mathcal{J}' \le
\mathcal{J}$, so $\mathcal{J} = \mathcal{J}'$.
\end{defn}

\begin{example}
\label{ex:join}
Let $\mathcal{P}_1$ have blocks $[0,2], [3,5], [6,7], \{8\}, \{9\}, \dots$ and
let $\mathcal{P}_2$ have blocks $\{0\}, [1,3], \{4\}, [5,6], \{7\}, \{8\},
\dots$  Then $[0,2]$ meets $[1,3]$, which meets $[3,5]$, which meets $[5,6]$,
which meets $[6,7]$.  By (i) the block of
$\mathcal{P}_1 \vee \mathcal{P}_2$ containing $0$ must absorb each of these in
turn, and by (ii) it is no larger.  So $[0,7]$ is one block of
$\mathcal{P}_1 \vee \mathcal{P}_2$, and $\{8\}, \{9\}, \dots$ are the others.
\end{example}

Lemma \ref{lem:join_components} shows that such a
$\mathcal{J}$ always exists, and describes it concretely.

\begin{lem}[Blocks of the join are connected components]
\label{lem:join_components}
Define a graph $\Gamma$ on the vertex set $X$ by creating an edge between $x$ and
$y$ whenever $x$ and $y$ lie in a common block of $\mathcal{P}_i$ for some
$i \in \{1,\dots,n\}$.  Then the blocks of
$\mathcal{P}_1 \vee \cdots \vee \mathcal{P}_n$ are exactly the connected
components of $\Gamma$.  Equivalently, $x$ and $y$ lie in the same block of the
join if and only if there is a finite chain
\[
  x = x_0, \; x_1, \; \dots, \; x_r = y
\]
in which each consecutive pair $x_{j}, x_{j+1}$ lies in a common block of some
$\mathcal{P}_{i_j}$.
\end{lem}

\begin{proof}
Let $\mathcal{J}$ be the partition of $X$ into connected components of $\Gamma$.
Any two points of a block of $\mathcal{P}_i$ are joined by an edge of $\Gamma$,
so that block lies inside one component; hence $\mathcal{P}_i \le \mathcal{J}$
for every $i$, which is (i).  For (ii), let $\mathcal{Q}$ satisfy
$\mathcal{P}_i \le \mathcal{Q}$ for all $i$.  If $xy$ is an edge of $\Gamma$,
then $x$ and $y$ share a $\mathcal{P}_i$-block, which lies in a single
$\mathcal{Q}$-block, so $x$ and $y$ share a $\mathcal{Q}$-block.  Lying in a
common $\mathcal{Q}$-block is an equivalence relation, so it is preserved along
chains; hence every connected component of $\Gamma$ lies in a single block of
$\mathcal{Q}$, which says $\mathcal{J} \le \mathcal{Q}$.  So $\mathcal{J}$
satisfies (i) and (ii) of Definition \ref{defn:join}, and by the uniqueness
noted there, $\mathcal{J} = \mathcal{P}_1 \vee \cdots \vee \mathcal{P}_n$.
\end{proof}

For interval partitions the components take a particularly concrete shape.

\begin{lem}[A block of a join of interval partitions is a maximal chain of
overlapping intervals]
\label{lem:join_interval_chain}
Suppose each $\mathcal{P}_i$ is an interval partition of $\mathbb{N}$, and let
$\mathcal{U} = \bigcup_{i=1}^{n} \mathcal{P}_i$ be the collection of all of
their blocks.  Form the graph $\Delta$ on $\mathcal{U}$ with an edge between two
blocks whenever they intersect.  Then every block of
$\mathcal{P}_1 \vee \cdots \vee \mathcal{P}_n$ is the union of the blocks in one
connected component of $\Delta$, and is itself an interval.
\end{lem}

\begin{proof}
A union of intervals that is chained together by intersection is an interval: if
$I$ and $J$ are intervals with $I \cap J \ne \emptyset$, then $I \cup J$ is an
interval, and the general case follows by induction on the length of the chain.
Two points $x$ and $y$ are joined by an edge of the graph $\Gamma$ of Lemma
\ref{lem:join_components} precisely when some $I \in \mathcal{U}$ contains both;
so a $\Gamma$-chain from $x$ to $y$ is exactly a $\Delta$-chain of blocks
running from one that contains $x$ to one that contains $y$.  The
$\Gamma$-components therefore correspond to the $\Delta$-components, and each is
the union of its blocks.
\end{proof}

\begin{rem}[Building a block of the join]
\label{rem:build_join_block}
Lemma \ref{lem:join_interval_chain} calls a block of the join a
\textit{maximal chain of overlapping intervals}.  The following construction
makes that phrase concrete.  Let
$\mathcal{U} = \bigcup_{i=1}^{n} \mathcal{P}_i$ be the collection of all blocks
of all $\mathcal{P}_i$, and let $B \in \mathcal{U}$.  Set $U_0 = B$ and, for
each $r \ge 0$,
\[
  U_{r+1} \;=\; U_r \;\cup\;
  \bigcup \bigl\{\, C \in \mathcal{U} \;:\; C \cap U_r \ne \emptyset \,\bigr\} .
\]
Each step adds every block that meets what has been collected so far.  Then
\[
  U_0 \subseteq U_1 \subseteq U_2 \subseteq \cdots,
  \mbox{\hskip 0.5pc}  \text{and}  \mbox{\hskip 0.5pc}
  \bigcup_{r \ge 0} U_r
  \;\text{ is the block of }\;
  \mathcal{P}_1 \vee \cdots \vee \mathcal{P}_n \;\text{ containing } B .
\]
Every $U_r$ is a finite interval: it is an interval because each block added
meets $U_r$ and a union of overlapping intervals is an interval, and it is
finite because $U_r$ meets only finitely many blocks of each of the $n$
partitions.  Consequently the construction halts, meaning $U_{r+1} = U_r$ for
some $r$, exactly when the block of the join containing $B$ is finite.
\end{rem}

\subsection{Orbits of a Composition Lie Inside Join Blocks}
\label{sec:containment}

\begin{lem}[Orbit containment]
\label{lem:orbit_in_join}
Let $A_1, \dots, A_n \subseteq \mathbb{N}$ each have all maximal intervals
finite, and let
\[
  \tau \;=\; \sigma_{A_{i_1}}^{s_1} \circ
  \sigma_{A_{i_2}}^{s_2} \circ \cdots \circ
  \sigma_{A_{i_r}}^{s_r},
  \qquad i_j \in \{1,\dots,n\}, \;\; s_j \in \{+1,-1\},
\]
be any finite composition of the permutations
$\sigma_{A_1}, \sigma_{A_1}^{-1}, \dots, \sigma_{A_n}, \sigma_{A_n}^{-1}$.
Then every orbit of $\tau$ is contained in a single block of
\[
  \mathcal{J} \;=\; \mathcal{P}_{A_1} \vee \mathcal{P}_{A_2} \vee \cdots
  \vee \mathcal{P}_{A_n}.
\]
\end{lem}

\begin{proof}
Fix a block $D$ of $\mathcal{J}$.  Since $\mathcal{P}_{A_i} \le \mathcal{J}$,
the set $D$ is a union of blocks of $\mathcal{P}_{A_i}$, for each $i$.  By Lemma
\ref{lem:orbits_are_blocks}, $\sigma_{A_i}$ maps each of its own blocks onto
itself; hence $\sigma_{A_i}(D) = D$, and likewise $\sigma_{A_i}^{-1}(D) = D$.
Each $\sigma_{A_{i_j}}^{s_j}$ in $\tau$ therefore maps $D$ onto $D$, for
$j = 1, \dots, r$, so $\tau(D) = D$.  If
$x \in D$, then $\tau^{j}(x) \in D$ for all $j \in \mathbb{Z}$, so the orbit of
$x$ lies in $D$.  As the blocks of $\mathcal{J}$ partition $\mathbb{N}$, every
orbit lies in exactly one of them.
\end{proof}

\begin{rem}[The containment is only an upper bound]
\label{rem:join_is_vacuous}
Lemma \ref{lem:orbit_in_join} bounds orbits from above and says nothing about
lower bounds.  The bound can be vacuous.  Take $A_1 = A$ and
$A_2 = \overline{A}$ for any $A$ with $A$ and $\overline{A}$ both infinite.
The maximal intervals of $A$ and of $\overline{A}$ alternate and partition
$\mathbb{N}$, so if $[m,n]$ is a maximal interval of $A$, then $n+1$ begins a
maximal interval of $\overline{A}$.  The $\mathcal{P}_A$-block $[m,n+1]$ and
the $\mathcal{P}_{\overline{A}}$-block beginning at $n+1$ therefore
\emph{overlap in the point $n+1$}.  The blocks chain end to end and cover
$\mathbb{N}$, so by Lemma \ref{lem:join_interval_chain}
\[
  \mathcal{P}_A \vee \mathcal{P}_{\overline{A}} \;=\; \{\mathbb{N}\} ,
\]
a single block.  For a complement pair $A$ and $\overline{A}$ the containment
lemma is vacuous.  What it does is indicate \emph{where to look}: the overlap
at $n+1$ is the only mechanism by which two generators can move a point out of
a bounded interval.
\end{rem}

\subsection{An Infinite Cycle with Two Generators}
\label{sec:thm}

\begin{thm}[Complementary set compositions generate infinite cycles]
\label{thm:complement_infinite_cycle}
Let $A \subseteq \mathbb{N}$ with $A$ and $\overline{A}$ both infinite.  Then
each of the two compositions
\[
  \sigma_{\overline{A}} \circ \sigma_{A}
  \qquad\text{and}\qquad
  \sigma_{A} \circ \sigma_{\overline{A}}
\]
has only one orbit, namely all of $\mathbb{N}$, and that orbit is an infinite
cycle.  That is, writing $\tau$ for either composition, the points
\[
  \dots, \;\; \tau^{-2}(0), \;\; \tau^{-1}(0), \;\; 0, \;\; \tau(0), \;\;
  \tau^{2}(0), \;\; \dots
\]
are all distinct and together they exhaust $\mathbb{N}$.  In particular $\tau$
has no finite orbit and no fixed point.
\end{thm}

\begin{proof}
By the last part of Remark \ref{rem:conjugate_orbits}, the two compositions
$\sigma_{\overline{A}} \circ \sigma_{A}$ and
$\sigma_{A} \circ \sigma_{\overline{A}}$ have the same orbit structure, so it
suffices to prove the conclusion for one of them.  Also, the hypothesis is
symmetric in $A$ and $\overline{A}$, and interchanging their names interchanges
the two compositions.  Hence it is sufficient to prove:

\begin{center}
if $0 \notin A$, then $\tau = \sigma_{\overline{A}} \circ \sigma_{A}$ has
 one infinite orbit that is all of $\mathbb{N}$.
\end{center}

\noindent
Indeed, if $0 \in A$ then $0 \notin \overline{A}$, and applying this statement
with $\overline{A}$ in place of $A$ shows that
$\sigma_{A} \circ \sigma_{\overline{A}}$ has the required form, since
$\overline{\overline{A}} = A$.

So assume $0 \notin A$.  By Remark \ref{rem:intervals_finite}, every maximal
interval of $A$ and of $\overline{A}$ is finite and there are infinitely many of
each, so both $\sigma_A$ and $\sigma_{\overline{A}}$ are defined and Lemma
\ref{lem:orbits_are_blocks} applies to both.  Since $0 \notin A$, the maximal
intervals alternate, starting with one of $\overline{A}$.  Now
  $\mathbb{N} \;=\; I_0 \dcup J_1 \dcup I_1 \dcup J_2 \dcup I_2 \dcup   \cdots$, where
   $I_i = [p_i, q_i]$  is a maximal interval of  $\overline{A}$, and
   $J_i = [m_i, n_i]$  is a maximal interval of  $A$.

Hence, $p_0 = 0$ and, for all $i \ge 1$,
\begin{equation}
  m_i = q_{i-1} + 1, \qquad p_i = n_i + 1.
  \label{eq:boundaries}
\end{equation}
By Lemma \ref{lem:orbits_are_blocks} and \eqref{eq:boundaries}, the
non-singleton blocks are
\[
  B_i \;=\; [m_i,\, p_i] \;\; (i \ge 1) \quad\text{for } \mathcal{P}_A,
  \qquad
  C_i \;=\; [p_i,\, m_{i+1}] \;\; (i \ge 0) \quad\text{for }
  \mathcal{P}_{\overline{A}},
\]
and on each block the permutation sends $k \mapsto k+1$ except that the right
endpoint goes to the left endpoint.  Note that the only element of $A$ lying in
any $C_i$ is its right endpoint $m_{i+1}$, and the only element of
$\overline{A}$ lying in any $B_i$ is its right endpoint $p_i$.

Recall that $\tau = \sigma_{\overline{A}} \circ \sigma_A$ applies $\sigma_A$
first.  Every $k \in \mathbb{N}$ falls into exactly one of the following five
cases.

\medskip
\noindent\textit{(i) $k \in J_i$ with $k < n_i$.}\quad
Here $\sigma_A(k) = k+1$, and $m_i < k+1 \le n_i$, so $k+1 \in A$ and $k+1$ is
not the right endpoint $m_j$ of any $C_j$.  Hence $\sigma_{\overline{A}}$ fixes
$k+1$ and
\[
  \tau(k) = k+1 .
\]

\noindent\textit{(ii) $k = n_i$.}\quad
Here $\sigma_A(n_i) = n_i + 1 = p_i$ by \eqref{eq:boundaries}.  Since
$p_i \le q_i < m_{i+1}$, the point $p_i$ is not the right endpoint of $C_i$, so
$\sigma_{\overline{A}}(p_i) = p_i + 1$ and
\[
  \tau(n_i) = n_i + 2 .
\]

\noindent\textit{(iii) $k \in I_i$ with $p_i < k \le q_i$.}\quad
Then $k \in \overline{A}$, and $p_i < k \le q_i < p_{i+1}$ shows that $k$ is not
the right endpoint $p_j$ of any $B_j$, so $\sigma_A$ fixes $k$.  As
$k \le q_i$, the point $k$ is not the right endpoint of $C_i$, so
\[
  \tau(k) = k+1 .
\]

\noindent\textit{(iv) $k = p_i$ with $i \ge 1$.}\quad
Here $p_i$ is the right endpoint of $B_i$, so $\sigma_A(p_i) = m_i$; and $m_i$
is the right endpoint of $C_{i-1}$, so $\sigma_{\overline{A}}(m_i) = p_{i-1}$.
Hence
\[
  \tau(p_i) = p_{i-1} \;<\; p_i .
\]

\noindent\textit{(v) $k = p_0 = 0$.}\quad
Here $0 \notin A$, and $0$ is not the right endpoint of any $B_j$, since those
endpoints are $p_j = n_j + 1 \ge 1$; so $\sigma_A(0) = 0$.  And $0 = p_0 \le
q_0$ implies $\sigma_{\overline{A}}(0) = 1$, so $\tau(0) = 1$.

\medskip
Collecting these cases, and writing $P = \{p_i : i \ge 1\}$,
\begin{equation}
  \tau(k) =
  \begin{cases}
    k + 2, & k = n_i \text{ for some } i \ge 1, \\[2pt]
    p_{i-1}, & k = p_i \in P, \\[2pt]
    k + 1, & \text{otherwise.}
  \end{cases}
  \label{eq:tau_formula}
\end{equation}

\medskip
\noindent\emph{Claim: each $p_i \in P$ has exactly one $\tau$-preimage, namely
$p_{i+1}$.}\quad
By \eqref{eq:tau_formula}, $\tau(p_{i+1}) = p_i$.  Conversely, suppose
$\tau(k) = p_i$ with $k \notin P$.  If $\tau(k) = k+1$, then $k = p_i - 1 = n_i$
by \eqref{eq:boundaries}, but case (ii) implies
$\tau(n_i) = n_i + 2 \ne n_i + 1$, a contradiction.  If $\tau(k) = k+2$, then
$k = n_j$ for some $j$ and $p_i = n_j + 2 = p_j + 1$; but
$p_{j+1} \ge q_j + 2 \ge p_j + 2$ while $p_j < p_j + 1$, so $p_j + 1$ is not
$p_i$ for any $i$, a contradiction.  This proves the claim.

\medskip
\noindent\emph{The forward orbit of $0$ is strictly increasing.}\quad
By induction, $\tau^{j}(0) \notin P$ for every $j \ge 0$.  This holds at $j = 0$
since $0 = p_0 \notin P$; and if $\tau^{j}(0) \notin P$, then
$\tau^{j+1}(0) \notin P$, since otherwise $\tau^{j}(0)$ would be a
$\tau$-preimage of some $p_i$ and yet differ from $p_{i+1} \in P$, contradicting
the claim.  Hence \eqref{eq:tau_formula} implies
$\tau^{j+1}(0) \in \{\tau^{j}(0)+1, \; \tau^{j}(0)+2\}$ for every $j$, so the
sequence $\bigl(\tau^{j}(0)\bigr)_{j \ge 0}$ is strictly increasing and the
orbit of $0$ is infinite.

\medskip
\noindent\emph{The orbit is all of $\mathbb{N}$.}\quad
Let $F = \{\tau^{j}(0) : j \ge 0\}$.  It was just shown that
$F \cap P = \emptyset$.  Next, $F = \overline{P}$, by induction on
$k$.  For $k = 0$: indeed $0 \in F$.  Let $k \ge 1$ with $k \notin P$.  If
$k - 1 \notin P$, then $k-1 \in F$ by induction, and $\tau(k-1) = k$ unless
$k - 1 = n_j$ for some $j$; but $k-1 = n_j$ would imply $k = p_j \in P$, which is
excluded.  So $k = \tau(k-1) \in F$.  If instead $k-1 = p_i$ for some
$i \ge 1$, then $k = p_i + 1 = n_i + 2$, and $n_i \in A$ while every element of
$P$ lies in $\overline{A}$, so $n_i \notin P$; by induction $n_i \in F$, and
case (ii) implies $k = n_i + 2 = \tau(n_i) \in F$.  This completes the induction,
so $F = \overline{P}$.  Finally $\tau^{-1}(0) = p_1$ and
$\tau^{-1}(p_i) = p_{i+1}$, so $\{\tau^{-j}(0) : j \ge 1\} = P$.  Therefore the
orbit of $0$ under $\tau$ is $F \cup P = \mathbb{N}$.  Since $\tau^{j}(0)$ is
strictly increasing in $j \ge 0$ and $\tau^{-j}(0) = p_j$ is strictly increasing
in $j \ge 1$, and since every $\tau^{j}(0)$ with $j \ge 0$ lies outside $P$
while every $\tau^{-j}(0)$ with $j \ge 1$ lies in $P$, no point of the orbit
repeats.  So the orbit is an infinite cycle, as described in Remark
\ref{rem:orbit_dichotomy}(b).
Since the orbit $O_{\tau}(0) = \{\, \tau^{j}(0) : j \in \mathbb{Z} \,\}$ is
$\mathbb{N}$, and the orbits of $\tau$ partition $\mathbb{N}$ by Remark
\ref{rem:orbits_partition}, it is the only orbit of $\tau$.
\end{proof}

\begin{cor}[$G_{\mathfrak{B}}$ contains an infinite cycle]
\label{cor:GB_infinite_cycle}
For every bi-immune $A \in \mathfrak{B}$, the permutations
$\sigma_{\overline{A}} \circ \sigma_{A}$ and
$\sigma_{A} \circ \sigma_{\overline{A}}$ both lie in $G_{\mathfrak{B}}$ and are
each one infinite cycle on $\mathbb{N}$.  Hence $G_{\mathfrak{B}}$ contains
permutations with infinite orbits, and indeed permutations with no finite orbit
at all.
\end{cor}

\begin{proof}
A set $A$ is bi-immune if and only if $\overline{A}$ is bi-immune.
Thus, $\sigma_A$ and $\sigma_{\overline{A}}$ are both generators of
$G_{\mathfrak{B}}$, and their compositions
$\sigma_{\overline{A}} \circ \sigma_{A}$ and
$\sigma_{A} \circ \sigma_{\overline{A}}$ both lie in $G_{\mathfrak{B}}$.
Both $A$ and $\overline{A}$ are infinite, so Theorem
\ref{thm:complement_infinite_cycle} applies.
\end{proof}

\begin{example}
\label{ex:odds}
Set $A = \mathbb{O} = \{1,3,5,\dots\}$, so that
$\overline{A} = \mathbb{E} = \{0,2,4,\dots\}$ and $0 \notin A$.  Every maximal
interval is a singleton, so by Lemma \ref{lem:orbits_are_blocks}
\[
  \sigma_{\mathbb{O}} = (1\;2)(3\;4)(5\;6)\cdots,
  \qquad
  \sigma_{\mathbb{E}} = (0\;1)(2\;3)(4\;5)\cdots .
\]
Here $p_i = 2i$ and $n_i = 2i-1$, so equation \eqref{eq:tau_formula} implies
$\tau(2i-1) = 2i+1$ and $\tau(2i) = 2i-2$; and indeed
$\tau = \sigma_{\mathbb{E}} \circ \sigma_{\mathbb{O}}$ is
\[
  \cdots \;\to\; 6 \;\to\; 4 \;\to\; 2 \;\to\; 0 \;\to\; 1 \;\to\; 3 \;\to\; 5
  \;\to\; \cdots,
\]
which is one infinite cycle that is all of $\mathbb{N}$.  Note that
$\mathbb{O}$ is computable, not bi-immune.  The point of this example is that Theorem
\ref{thm:complement_infinite_cycle} says the orbit structure of
$\sigma_{\overline{A}} \circ \sigma_{A}$ is \emph{the same for every} $A$ with
$A$ and $\overline{A}$ both infinite.  The interval lengths change which points
are skipped on the way out and which are visited coming back, but never the fact
that there is only one orbit and that it is an infinite cycle.
\end{example}

\begin{lem}[The infinite cycles form one conjugacy class]
\label{lem:one_class}
Let
\[
  T \;=\; \{\, \rho \in \Sym(\mathbb{N}) \;:\;
  \rho \text{ has only one orbit, equal to } \mathbb{N},
  \text{ and it is an infinite cycle} \,\} .
\]
Then $T$ is one conjugacy class of $\Sym(\mathbb{N})$.
\end{lem}

\begin{proof}
By Remark \ref{rem:conjugate_orbits}, $T$ is closed under conjugation.  For the
converse, let $\rho, \tau \in T$ and choose $a, b \in \mathbb{N}$.  By Remark
\ref{rem:orbit_dichotomy}(b) the lists
$\bigl(\rho^{j}(a)\bigr)_{j \in \mathbb{Z}}$ and
$\bigl(\tau^{j}(b)\bigr)_{j \in \mathbb{Z}}$ each run through every element of
$\mathbb{N}$ exactly once.  Define
\[
  \alpha\bigl(\tau^{j}(b)\bigr) \;=\; \rho^{j}(a) \qquad (j \in \mathbb{Z}) .
\]
Then $\alpha$ is a permutation of $\mathbb{N}$, and
$\alpha^{-1}\bigl(\rho^{j}(a)\bigr) = \tau^{j}(b)$ for every $j$.  Hence
\begin{align*}
  \bigl( \alpha \circ \tau \circ \alpha^{-1} \bigr)\bigl(\rho^{j}(a)\bigr)
  &\;=\; \bigl( \alpha \circ \tau \bigr)
         \Bigl( \alpha^{-1}\bigl(\rho^{j}(a)\bigr) \Bigr)
   \;=\; \bigl( \alpha \circ \tau \bigr)\bigl(\tau^{j}(b)\bigr) \\[2pt]
  &\;=\; \alpha\bigl(\tau^{j+1}(b)\bigr)
   \;=\; \rho^{j+1}(a)
   \;=\; \rho\bigl(\rho^{j}(a)\bigr) .
\end{align*}
Every element of $\mathbb{N}$ equals $\rho^{j}(a)$ for some $j \in \mathbb{Z}$,
so $\alpha \circ \tau \circ \alpha^{-1} = \rho$.  Hence any two elements of $T$
are conjugate.
\end{proof}

\begin{lem}[The centralizer of an infinite cycle]
\label{lem:centralizer}
Let $\tau \in \Sym(\mathbb{N})$ have only one orbit, equal to $\mathbb{N}$,
and let that orbit be an infinite cycle.  Then the centralizer of $\tau$ in
$\Sym(\mathbb{N})$ is the cyclic group $\langle \tau \rangle =
\{\tau^{n} : n \in \mathbb{Z}\}$.
\end{lem}

\begin{proof}
Every power of $\tau$ commutes with $\tau$.  Conversely, suppose
$\alpha \in \Sym(\mathbb{N})$ and $\alpha \circ \tau = \tau \circ \alpha$.  By
Remark \ref{rem:orbit_dichotomy}(b) the map $j \mapsto \tau^{j}(0)$ is a
bijection from $\mathbb{Z}$ onto $\mathbb{N}$, so write
$\alpha\bigl(\tau^{j}(0)\bigr) = \tau^{c(j)}(0)$ for a unique
$c(j) \in \mathbb{Z}$.  From $\alpha \circ \tau = \tau \circ \alpha$,
\[
  \tau^{c(j+1)}(0) \;=\; \alpha\bigl(\tau^{j+1}(0)\bigr)
  \;=\; \tau\bigl(\alpha(\tau^{j}(0))\bigr) \;=\; \tau^{c(j)+1}(0) ,
\]
and since $j \mapsto \tau^{j}(0)$ is injective, $c(j+1) = c(j) + 1$ for all
$j \in \mathbb{Z}$.  Hence $c(j) = c(0) + j$, so
$\alpha\bigl(\tau^{j}(0)\bigr) = \tau^{c(0)}\bigl(\tau^{j}(0)\bigr)$ for every
$j$, which implies $\alpha = \tau^{c(0)}$.
\end{proof}

\begin{cor}[Containing the whole class is equivalent to equality]
\label{cor:class_equiv}
Let $T$ be as in Lemma \ref{lem:one_class}, let $A \in \mathfrak{B}$, and let
$\tau$ be either of the two permutations
$\sigma_{\overline{A}} \circ \sigma_{A}$ or
$\sigma_{A} \circ \sigma_{\overline{A}}$ of Corollary
\ref{cor:GB_infinite_cycle}, so that $\tau \in T \cap G_{\mathfrak{B}}$.  Then
the following three statements are logically equivalent.
\begin{enumerate}
  \item[(i)] $G_{\mathfrak{B}} = \Sym(\mathbb{N})$.
  \item[(ii)] $T \subseteq G_{\mathfrak{B}}$.
  \item[(iii)] $\{\, \alpha \circ \tau \circ \alpha^{-1} \;:\;
  \alpha \in G_{\mathfrak{B}} \,\} = T$.
\end{enumerate}
\end{cor}

\begin{proof}
(i) implies (ii) and (iii) at once, since $T \subseteq \Sym(\mathbb{N})$ and, by
Lemma \ref{lem:one_class}, conjugating $\tau$ by all of $\Sym(\mathbb{N})$
covers all of $T$.

\medskip
\noindent
(ii) implies (i).  Let $N_T$ be the subgroup of $\Sym(\mathbb{N})$ generated by
$T$.  Since $T$ is a conjugacy class by Lemma \ref{lem:one_class}, $N_T$ is a
normal subgroup of $\Sym(\mathbb{N})$.  By Theorem
\ref{thm:complement_infinite_cycle} the element $\tau$ has support
$\mathbb{N}$, so $\tau$ is not finitary and $N_T \not\subseteq FS(\mathbb{N})$.
By the Baer--Schreier--Ulam theorem \cite{baer,schreier_ulam} the normal
subgroups of $\Sym(\mathbb{N})$ are
\[
  1 \;\lhd\; \mathrm{Alt}(\mathbb{N}) \;\lhd\; FS(\mathbb{N})
  \;\lhd\; \Sym(\mathbb{N}) ,
\]
where $\mathrm{Alt}(\mathbb{N})$ is the group of finitary permutations of
$\mathbb{N}$ that are even on their support \cite{cameron}.  Hence
$N_T = \Sym(\mathbb{N})$.  If $T \subseteq G_{\mathfrak{B}}$, then
$\Sym(\mathbb{N}) = N_T \subseteq G_{\mathfrak{B}}$.

\medskip
\noindent
(iii) implies (i).  Let $Z$ be the centralizer of $\tau$ in $\Sym(\mathbb{N})$.
By Lemma \ref{lem:centralizer}, $Z = \langle \tau \rangle$, and
$\tau \in G_{\mathfrak{B}}$, so $Z \subseteq G_{\mathfrak{B}}$.  Let
$\beta \in \Sym(\mathbb{N})$ be arbitrary.  Then
$\beta \circ \tau \circ \beta^{-1} \in T$ by Lemma \ref{lem:one_class}, so by
(iii) there is an $\alpha \in G_{\mathfrak{B}}$ with
$\beta \circ \tau \circ \beta^{-1} = \alpha \circ \tau \circ \alpha^{-1}$.  Then
$\alpha^{-1} \circ \beta$ commutes with $\tau$, so
$\alpha^{-1} \circ \beta \in Z \subseteq G_{\mathfrak{B}}$ and therefore
$\beta = \alpha \circ (\alpha^{-1} \circ \beta) \in G_{\mathfrak{B}}$.
\end{proof}

\begin{rem}[Finite-orbit permutations are not closed under composition]
\label{rem:not_a_subgroup}
Let $\mathcal{F}$ be the set of all $\rho \in \Sym(\mathbb{N})$ every orbit
of which is finite.  Then $\mathcal{F}$ is closed under inverses and contains
every generator $\sigma_A$ of $G_{\mathfrak{B}}$, by Lemma
\ref{lem:orbits_are_blocks}.  However, $\mathcal{F}$ is \emph{not} a subgroup of
$\Sym(\mathbb{N})$: for bi-immune $A$, both $\sigma_A$ and
$\sigma_{\overline{A}}$ lie in $\mathcal{F}$, while by Theorem
\ref{thm:complement_infinite_cycle} their composition
$\sigma_{\overline{A}} \circ \sigma_{A}$ has one infinite orbit
and so $\sigma_{\overline{A}} \circ \sigma_{A}$  lies outside $\mathcal{F}$.
\end{rem}

%% file: dense_bi-immune_orbits.tex

\section{Bi-Immune Dense Orbits in $[0, 1)$}
\label{sec:dense_orbits}

$a = a_0 a_1 \dots a_n \dots$ $\in$ $\{0, 1\}^{\mathbb{N}}$
is called a {\it tail sequence} if there exists $M$ such that $a_n = 1$ for all $n \ge M$.
Based on the standard identification between $A \subset \mathbb{N}$ and the binary sequence $\{a_k\}$,
where $a_k = 1$ if $k \in A$ and $a_k = 0$ if $k \notin A$, each tail sequence uniquely corresponds
to a tail set and vice versa.  Set $\mathcal{T} = \{ a \in \{0, 1\}^{\mathbb{N}} :
a \text{ is a tail sequence} \}$, and set $\mathcal{S} = \{0, 1\}^{\mathbb{N}} - \mathcal{T}$.

Throughout this section a sequence $x = x_0 x_1 x_2 \dots \in \mathcal{S}$ is identified
with the real number
\begin{equation}
  \label{eq:binary_value}
  v(x) \;=\; \sum_{j=0}^{\infty} x_j \, 2^{-(j+1)} \;\in\; [0, 1) .
\end{equation}
The map $v : \mathcal{S} \rightarrow [0, 1)$ is a one-to-one correspondence: every
$t \in [0,1)$ has exactly one binary expansion that is not a tail sequence.
If $a$ is a tail sequence, let $N_0$ be the natural number such that $a_n = 1$ when
$n \ge N_0$ and $a_{N_0-1} = 0$.  In the standard topology on $[0, 1)$, $a$ represents
the same point as the sequence $b$ with $b_i = a_i$ for $0 \le i < N_0-1$,
$b_{N_0-1} = 1$, and $b_n = 0$ for all $n \ge N_0$.

\begin{rem}\label{rem:bi_immune_in_S}
Every bi-immune sequence lies in $\mathcal{S}$.
\end{rem}

\begin{proof}
If $a$ were a tail sequence, its corresponding set $A$ would have a finite complement
$\overline{A}$, so $\overline{A}$ would not be immune.
\end{proof}

\subsection{The Odometer}
\label{sec:odometer}

\begin{defn}[Odometer]\label{defn:odometer}
Let $x = x_0 x_1 x_2 \dots \in \mathcal{S}$.  Since $x$ is not a tail sequence, $x$ has
infinitely many $0$'s; in particular, there is a least $i$ with $x_i = 0$.  Define
$T(x) = y_0 y_1 y_2 \dots$ by
\[
  y_j = 0 \;\; (0 \le j < i), \qquad
  y_i = 1, \qquad
  y_j = x_j \;\; (j > i).
\]
That is, $T$ adds $1$ at position $0$ and carries toward larger indices:
the leading block of $1$'s becomes $0$'s, and the first $0$ becomes a $1$.
For example, $T(1101\dots) = 0011\dots$ and $T(0110\dots) = 1110\dots$.
$T$ is called the \textit{binary odometer}.
\end{defn}

\begin{lem}[The odometer changes finitely many bits]
\label{lem:odometer_finite_change}
Let $x \in \mathcal{S}$, and let $i$ be the least index with $x_i = 0$.  Then
$T(x)$ and $x$ differ exactly at the positions $0, 1, \dots, i$.  Hence
$T(x) \in \mathcal{S}$, and for every $n \ge 1$ the sequence $T^{n}(x)$ is defined,
lies in $\mathcal{S}$, and differs from $x$ at only finitely many positions.
\end{lem}

\begin{proof}
By definition \ref{defn:odometer}, positions $0, \dots, i-1$ change from $1$ to $0$,
position $i$ changes from $0$ to $1$, and every position $j > i$ is unchanged.  Since
$T(x)$ agrees with $x$ beyond position $i$ and $x$ is not a tail sequence, $T(x)$ is not a
tail sequence, so $T(x) \in \mathcal{S}$.  The statement about $T^{n}(x)$ follows by
induction on $n$: if $T^{n-1}(x) \in \mathcal{S}$ differs from $x$ on a finite set
$F_{n-1}$, then $T^{n}(x) \in \mathcal{S}$ differs from $T^{n-1}(x)$ on a finite set
$G_n$, so $T^{n}(x)$ differs from $x$ on a subset of the finite set $F_{n-1} \cup G_n$.
\end{proof}

\begin{lem}[The odometer preserves bi-immunity]
\label{lem:odometer_bi_immune}
If $a$ is a bi-immune sequence, then $T(a)$ is a bi-immune sequence.  Consequently
$T^{n}(a)$ is bi-immune for every $n \in \mathbb{N}$.
\end{lem}

\begin{proof}
By remark \ref{rem:bi_immune_in_S}, $a \in \mathcal{S}$, so $T(a)$ is defined.  Let $A$
be the set corresponding to $a$, and let $i$ be the least index with $a_i = 0$.  By lemma
\ref{lem:odometer_finite_change}, the set corresponding to $T(a)$ is
$A \oplus \{0, 1, \dots, i\}$.  A finite set is computably enumerable, so lemma
\ref{lem:recursive_xor_bi_immune} implies $A \oplus \{0, 1, \dots, i\}$ is bi-immune.
The second statement follows by induction on $n$.
\end{proof}

For $x \in \{0,1\}^{\mathbb{N}}$ and $k \ge 1$, define the natural number
\[
  N_k(x) \;=\; \sum_{j=0}^{k-1} x_j \, 2^{j} \;\in\; \{0, 1, \dots, 2^k - 1\} .
\]
$N_k(x)$ reads the first $k$ bits of $x$ as a binary integer whose
\emph{least} significant bit is $x_0$.  The map $x_0 \dots x_{k-1} \mapsto N_k(x)$
is a one-to-one correspondence between the $2^k$ binary words of length $k$ and
$\{0, 1, \dots, 2^k - 1\}$.

\begin{lem}[The odometer counts modulo $2^k$]
\label{lem:odometer_prefix}
For every $x \in \mathcal{S}$, every $k \ge 1$, and every $n \in \mathbb{N}$,
\[
  N_k\bigl(T^{n}(x)\bigr) \;\equiv\; N_k(x) + n \pmod{2^k} .
\]
\end{lem}

\begin{proof}
It suffices to prove the case $n = 1$; the general case follows by induction on $n$,
since $T^{n}(x) \in \mathcal{S}$ by lemma \ref{lem:odometer_finite_change}.  Let $i$ be the
least index with $x_i = 0$, and set $y = T(x)$.

\smallskip
\noindent \textit{Case $i < k$.}  Positions $0, \dots, i-1$ of $x$ contribute
$2^0 + 2^1 + \dots + 2^{i-1} = 2^{i} - 1$ to $N_k(x)$, and position $i$ contributes $0$.
In $y$, positions $0, \dots, i-1$ contribute $0$ and position $i$ contributes $2^{i}$.
Positions $i+1, \dots, k-1$ are unchanged.  Hence
$N_k(y) = N_k(x) - (2^i - 1) + 2^i = N_k(x) + 1$.

\smallskip
\noindent \textit{Case $i \ge k$.}  Then $x_0 = \dots = x_{k-1} = 1$, so
$N_k(x) = 2^k - 1$, and $y_0 = \dots = y_{k-1} = 0$, so
$N_k(y) = 0 \equiv N_k(x) + 1 \pmod{2^k}$.
\end{proof}

\subsection{Dense Orbits of Bi-Immune Sequences}
\label{sec:odometer_dense}

For $p_0 \in \mathcal{S}$, define the orbit
$\mathcal{O}(T, p_0) = \{p_0\} \cup \{ p_{n} : p_{n} = T(p_{n-1}) \text{ for all } n \ge 1 \}
= \{ T^{n}(p_0) : n \in \mathbb{N} \}$.

\begin{lem}\label{lem:S1_dense_bi_immune}
Identify $\mathcal{S}$ with $[0, 1)$ by equation \eqref{eq:binary_value}.  If $p_0$ is a
bi-immune sequence, then each point in $\mathcal{O}(T, p_0)$, expressed as a binary
sequence, is a bi-immune sequence, and $\mathcal{O}(T, p_0)$ is dense in $[0, 1)$.
\end{lem}

\begin{proof}
Lemma \ref{lem:odometer_bi_immune} implies that each point in $\mathcal{O}(T, p_0)$ is
bi-immune.

\smallskip
\noindent \textit{Every prefix occurs.}  Let $k \ge 1$ and let $w_0 w_1 \dots w_{k-1}$ be any
binary word of length $k$.  Choose $n \in \{0, 1, \dots, 2^k - 1\}$ with
$n \equiv N_k(w) - N_k(p_0) \pmod{2^k}$.  By lemma \ref{lem:odometer_prefix},
$N_k\bigl(T^{n}(p_0)\bigr) \equiv N_k(w) \pmod{2^k}$, and both sides lie in
$\{0, \dots, 2^k - 1\}$, so they are equal.  Since $N_k$ is one-to-one on words of length
$k$, the first $k$ bits of $T^{n}(p_0)$ are $w_0 w_1 \dots w_{k-1}$.

\smallskip
\noindent \textit{Density in $[0,1)$.}  Let $0 \le u < t \le 1$.  Choose $k$ with
$2^{-k} < (t - u)/2$.  Then some interval
$\bigl[\, m 2^{-k}, \; (m+1) 2^{-k} \,\bigr)$ with $m \in \{0, \dots, 2^k - 1\}$ lies
inside $(u, t)$.  Write $m = \sum_{j=0}^{k-1} w_j \, 2^{k-1-j}$ with each
$w_j \in \{0, 1\}$.  For every $x \in \mathcal{S}$ whose first $k$ bits are
$w_0 \dots w_{k-1}$,
\[
  m 2^{-k} \;\le\; v(x) \;=\; m 2^{-k} + \sum_{j \ge k} x_j \, 2^{-(j+1)}
  \;<\; m 2^{-k} + 2^{-k} ,
\]
where the last inequality is strict because $x$ is not a tail sequence, so
$x_j = 0$ for some $j \ge k$.  By the previous paragraph, some $T^{n}(p_0)$ has first
$k$ bits $w_0 \dots w_{k-1}$, so $v\bigl(T^{n}(p_0)\bigr) \in (u, t)$.  Hence
$\mathcal{O}(T, p_0)$ meets every non-empty open interval of $[0,1)$.
\end{proof}

\begin{rem}[The odometer as a map of $[0,1)$]
\label{rem:odometer_interval_map}
Under the identification \eqref{eq:binary_value}, position $0$ is the most significant
bit, so $T$ is not addition of a fixed real number.  
The odometer as a map is the piecewise translation
that, for each $i \ge 0$, maps interval 
$\bigl[\, 1 - 2^{-i}, \; 1 - 2^{-(i+1)} \,\bigr)$ onto
$\bigl[\, 2^{-(i+1)}, \; 2^{-i} \,\bigr)$ by
\[
  v\bigl(T(x)\bigr) \;=\; v(x) - \bigl(1 - 2^{-i}\bigr) + 2^{-(i+1)} .
\]
This map is known as the dyadic odometer.  It replaces the
irrational rotation $x \mapsto (x + \beta) \bmod 1$ used in an earlier version of this
section.  Rotation by $\beta$ adds bits with carries, and the carry sequence depends on
$x$, so lemma \ref{lem:recursive_xor_bi_immune} does not apply to it.  The odometer also
produces carries, but each carry stops after finitely many positions, which is exactly
what lemma \ref{lem:odometer_finite_change} states.
\end{rem}

\begin{rem}[Exclusive-or with a computable sequence]
\label{rem:xor_map}
Let $\beta$ be Turing computable with binary sequence $b = b_0 b_1 \dots$, and set
$g(x) = x \oplus b$.  By lemma \ref{lem:recursive_xor_bi_immune}, $g$ maps bi-immune
sequences to bi-immune sequences.  However $g \circ g$ is the identity, so the orbit of
$p_0$ under $g$ is $\{p_0, \; p_0 \oplus b\}$, which is not dense.  This is why the
odometer is used in lemma \ref{lem:S1_dense_bi_immune}.
\end{rem}

\subsection{Rotations and Martin-L\"of Randomness}
\label{sec:rotation_ML}

A reader familiar with dynamical systems will anticipate irrational rotations, not the
odometer, here.  This subsection shows that rotations of the circle $\mathbb{S}^1$,
identified with $[0,1)$, do preserve a stronger form of randomness, which implies
bi-immunity.  Let $\mu$ denote Lebesgue measure on $[0,1)$.  A real number $\beta$ is
\textit{Turing computable} if there is a Turing computable sequence of rationals
$q_0, q_1, \dots$ with $|\beta - q_s| < 2^{-s}$ for all $s$.

\begin{defn}[Martin-L\"of randomness]
\label{defn:ML_random}
A \textit{Martin-L\"of test} is a sequence $U_0, U_1, U_2, \dots$ of open subsets of
$[0,1)$ such that $\mu(U_n) \le 2^{-n}$ for every $n$, and there is a single Turing
machine that, on input $n$, enumerates rational intervals whose union is $U_n$.  A point
$t \in [0,1)$ is \textit{Martin-L\"of random} if $t \notin \bigcap_{n} U_n$ for every
Martin-L\"of test.  A sequence $a \in \mathcal{S}$ is Martin-L\"of random if $v(a)$ is.
This agrees with the usual definition on $\{0,1\}^{\mathbb{N}}$ \cite{downey}.
\end{defn}

There are only countably many Martin-L\"of tests, and $\bigcap_n U_n$ has measure $0$ for
each test.  Hence the Martin-L\"of random points have measure $1$.  No Turing computable
real $t$ is Martin-L\"of random: the intervals $(q_s - 2^{-s}, \, q_s + 2^{-s})$, for
$s = n+2$, form a test whose intersection contains $t$.  In particular, a Martin-L\"of
random point is not a dyadic rational.

\begin{lem}[Martin-L\"of random sequences are bi-immune]
\label{lem:ML_implies_bi_immune}
If $a \in \mathcal{S}$ is Martin-L\"of random, then $a$ is bi-immune.
\end{lem}

\begin{proof}
Let $A$ be the set corresponding to $a$.  Suppose $R$ is an infinite c.e. set with
$R \subseteq A$.  For each $n$, let $r_0, \dots, r_{n-1}$ be the first $n$ elements in
an enumeration of $R$, set $k = 1 + \max\{r_0, \dots, r_{n-1}\}$, and let $U_n$ be the
union of the open dyadic intervals $\bigl(m 2^{-k}, \, (m+1) 2^{-k}\bigr)$ whose
length-$k$ binary word has a $1$ at each of the positions $r_0, \dots, r_{n-1}$.
Exactly $2^{k-n}$ of the $2^k$ words qualify, so $\mu(U_n) = 2^{-n}$, and the
$U_n$ are uniformly enumerable because $R$ is c.e.  Since $a_r = 1$ for all
$r \in R$, and $v(a)$ is not a dyadic rational, $v(a) \in U_n$ for every $n$.  Hence
$a$ is not Martin-L\"of random, a contradiction.  So $A$ is immune.  The same argument,
with $0$ in place of $1$, shows $\overline{A}$ is immune.
\end{proof}

\begin{lem}[Computable rotations preserve Martin-L\"of randomness]
\label{lem:rotation_ML}
Let $\beta \in [0,1)$ be Turing computable, and let $f(x) = (x + \beta) \bmod 1$.  If
$x \in [0,1)$ is Martin-L\"of random, then $f(x)$ is Martin-L\"of random.
\end{lem}

\begin{proof}
Suppose $f(x) \in \bigcap_n U_n$ for a Martin-L\"of test $(U_n)$.  Set
$V_n = f^{-1}(U_n)$.  Since $f$ is a bijection of $[0,1)$ that translates each of the
two intervals $[0, 1-\beta)$ and $[1-\beta, 1)$, it preserves Lebesgue measure, so
$\mu(V_n) = \mu(U_n) \le 2^{-n}$.  The preimage of a rational interval
$(p, q)$ under $f$ is the part of $(p - \beta, \, q - \beta)$ and of
$(p - \beta + 1, \, q - \beta + 1)$ that lies in $[0,1)$.  Since $\beta$ is
computable, a strict inequality between a rational and $p - \beta$ or $q - \beta$ can
be confirmed in finitely many steps.  Hence one Turing machine can, on input $n$,
enumerate rational intervals whose union is $V_n' = V_n - \{0\}$.  (A piece of the
form $[0, s)$ is enumerated as $(0, s)$.)  So $(V_n')$ is a Martin-L\"of test.  The
point $x$ is not $0$, because $0$ is computable, so $x \in V_n'$ for every $n$.  This
contradicts the randomness of $x$.
\end{proof}

\begin{cor}[Dense rotation orbits of Martin-L\"of random points]
\label{cor:rotation_ML_dense}
Let $\beta$ be a Turing computable irrational number, let
$f(x) = (x + \beta) \bmod 1$, and let $p_0 \in \mathcal{S}$ be Martin-L\"of random.
Then the orbit $\{ f^{n}(p_0) : n \in \mathbb{N} \}$ is dense in $[0,1)$, and every
point of it, expressed as a binary sequence, is Martin-L\"of random and hence
bi-immune.
\end{cor}

\begin{proof}
The orbit is dense due to Jacobi's theorem.  
By lemma \ref{lem:rotation_ML} and induction, every
$f^{n}(p_0)$ is Martin-L\"of random, and lemma \ref{lem:ML_implies_bi_immune} implies it
is bi-immune.
\end{proof}

\begin{rem}[Bi-immunity is coordinate-dependent]\label{rem:coordinate_dependence}
Martin-L\"of randomness is defined by Lebesgue measure, and Lebesgue measure is
invariant under rotation, so rotations preserve it.  Bi-immunity is defined by the
digits of a binary expansion.  It is preserved by operations that act on digits one at
a time, such as $A \oplus R$ (lemma \ref{lem:recursive_xor_bi_immune}) and a computable
rearrangement of positions (lemma \ref{recursive_sigma_preserves_bi_immunity}).  Adding
$\beta$ is not such an operation: the $i$th output bit depends on input bits
$i, i+1, \dots$ up to the place where the carry chain stops, which depends on $x$.  In
the language of physics, bi-immunity is a coordinate-dependent property, and Martin-L\"of randomness is not.

The two notions nevertheless agree almost everywhere.  The bi-immune points have
measure $1$, by lemma \ref{lem:ML_implies_bi_immune}.  Since $f$ preserves Lebesgue
measure, the set of $x$ such that both $x$ and $f(x)$ are bi-immune also has measure
$1$.  Any bi-immune $x$ with $f(x)$ not bi-immune therefore lies in a set of measure
zero.  Corollary \ref{cor:rotation_ML_dense} requires the stronger hypothesis that $p_0$
is Martin-L\"of random; lemma \ref{lem:S1_dense_bi_immune} requires only that $p_0$ is bi-immune. 
\end{rem}

\subsection{Dense Orbits and the Bi-Immune Symmetric Group}

Set $\phi(\mathcal{S}) = \{\sigma_A : A \in \mathcal{S}\}$, where $\sigma_A$ is given by
definition \ref{defn:set_rearrange}.

\begin{rem}
The subgroup of $\Sym(\mathbb{N})$ generated by $\phi(\mathcal{S})$ contains $G_{\mathfrak{B}}$.
\end{rem}

\begin{proof}
The bi-immune permutations generate the bi-immune symmetric group,
and the bi-immune sequences are a subset of $\mathcal{S}$ by remark \ref{rem:bi_immune_in_S}.
\end{proof}

\begin{cor}\label{cor:dense_orbit_in_GB}
Let $p$ be a bi-immune sequence.  Then $\phi\bigl(\mathcal{O}(T, p)\bigr) \subseteq
\mathfrak{B} \subseteq G_{\mathfrak{B}}$.  That is, the bi-immune symmetric group
contains the $\phi$-image of every dense orbit $\mathcal{O}(T, p)$, where $p$ ranges
over all bi-immune sequences in $\mathcal{S}$.
\end{cor}

\begin{proof}
By lemma \ref{lem:S1_dense_bi_immune}, every point of $\mathcal{O}(T, p)$ is a bi-immune
sequence, so its $\phi$-image is a bi-immune permutation, which lies in $\mathfrak{B}$.
\end{proof}

%% file: eventually_commutative.tex
\section{Some Other Properties of $\sigma_{A}$ }

\begin{defn}\label{set_eventually_commutative}  \hskip 1pc  {\it Eventually Commutative}  

\noindent 
Suppose set $A$ is ordered as $A = \{a_0, a_1, \dots  \}$ where $a_k < a_{k+1}$.   
Set $A$ is eventually commutative if there exists $M$ such that for any $j \ne k$
where $j, k \ge M$ implies that $|a_j - a_k| > 1$.   

\end{defn}

Observe that if set $A$ is eventually commutative, 
then $\sigma_{A} \circ \sigma_{A} = \sigma_{(r_0, r_1, \dots, r_n)}$.  
Lemma \ref{commute_order_lemma} implies that the 
$\sigma_j$ and $\sigma_k$ with $j, k \ge M$ can be swapped with each element in $\sigma_{A}$ 
until $\sigma_j$ is next to $\sigma_k$.

\smallskip   

In contrast to lemma \ref{bi_finite_set}, lemma \ref{recursive_sigma_preserves_bi_immunity} shows that 
$\sigma_R(A)$ is bi-immune when $A$ is bi-immune and $R$ is a Turing computable set.  


\begin{lem}\label{recursive_sigma_preserves_bi_immunity}

If $R$ is Turing computable and $A$ is bi-immune, then $\sigma_R(A)$ is bi-immune.   

\end{lem}    

\begin{proof}     
$\sigma_R(A)$ is infinite because $\sigma_R$ is one-to-one.    
Let $Q = \sigma_R(A)$.   It remains to show that $Q$ and $\overline{Q}$ satisfy condition (ii).
By contradiction, suppose $B \cap Q = \emptyset$ for some c.e.  set $B$.    
Note that ${\sigma}^{-1}_R$ is Turing computable because $\sigma_R$ is Turing computable. 
Thus, ${\sigma}^{-1}_R(B)$ is c.e.   
Now  ${\sigma}^{-1}_R(B) \cap A = \emptyset$ 
because  ${\sigma}^{-1}_R$ is one-to-one and ${\sigma}^{-1}_R(Q) = A$.  
This contradicts that $A$ is bi-immune.

A similar argument holds for $\overline{Q}$.   By contradiction,  
suppose $B  \cap \overline{Q} = \emptyset$ for some c.e.  set $B$.  
Then $B  \subset Q$.    Since ${\sigma}^{-1}_R(Q) = A$, 
then ${\sigma}^{-1}_R(B) \subset A$, 
which implies that ${\sigma}^{-1}_R(B) \cap \overline{A} = \emptyset$.   
This contradicts that $A$ is  bi-immune.   
\end{proof}

\begin{defn}\label{defn:consec_cycle_compose}
$\sigma$ is called a {\it consecutive cycle} if it is finitary and $\sigma = (m, m+1, \dots, n)$.   In other words,
$\sigma(k) = k+1$ when $m \le k < n$, $\sigma(n) = m$ and $\sigma(x) = x$ when $x< m$ or $x > n$.  Consider consecutive cycles
$(m_1 \dots n_1)$ and $(m_2 \dots n_2)$.  Suppose  $n_1 < n_2$.  These cycles are {\it disconnected} if $m_2 - n_1 > 1$.   
Consecutive cycles as a composition $ \dots (m_k \dots n_k) \circ \dots \circ (m_2 \dots n_2) \circ (m_1 \dots n_1)$ are 
{\it increasing} if $m_{j+1} > m_j$ for each $j \ge 1$. 

\end{defn}

\begin{lem}\label{consec_cycle}

If $A$ is bi-immune, then $\sigma_{A}$ is an infinite composition of 
disconnected, increasing, consecutive cycles. 

\end{lem}

\begin{proof}  
Lemma  \ref{consec_cycle} follows from the previous definitions and lemmas.   
\end{proof}

\smallskip

\begin{lem}\label{lem:up_iff_in}
For any $A \subseteq \mathbb{N}$ and any $k \in \mathbb{N}$,
\[
  \sigma_A(k) = k+1 \quad\Longleftrightarrow\quad k \in A .
\]
Consequently $\sigma_A(k) \le k+1$ for all $k \in \mathbb{N}$.
\end{lem}

\begin{proof}
By lemma \ref{consec_cycle}, $\sigma_A$ is a composition of disjoint consecutive
cycles.  Let $\{m, \ldots, n\}$ be a maximal interval contained in $A$, so that
$m-1 \notin A$ and $n+1 \notin A$.  The consecutive cycle that $\{m, \ldots, n\}$ generates is
supported on the interval $\{m, \ldots, n+1\}$, one larger, and by definition
\ref{defn:consec_cycle_compose} it sends $j \mapsto j+1$ for $m \le j \le n$ and
sends $n+1 \mapsto m$.

On that interval $\{m, \ldots, n+1\}$, the points moved up by one are exactly $m, \ldots, n$, 
and the single remaining point $n+1$ is moved down.  Taking the union over all maximal
intervals contained in $A$, and observing that $\sigma_A$ fixes every point lying in
no such interval, the points with $\sigma_A(k) = k+1$ are exactly the points of
$A$.  By the disconnectedness in lemma \ref{consec_cycle}, every other point is
either fixed or moved strictly down, so $\sigma_A(k) \le k+1$ for all $k$.
\end{proof}


\begin{rem}
If $A$ is bi-immune, then $\sigma_{A}(\mathbb{E})$ is Borel-1 normal.  
\end{rem}   

\noindent $\mathbb{E}$ equals $(1 0 1 0 \dots)$ when identified as a binary sequence, so  
this remark follows from lemma \ref{consec_cycle}.

%% file: proper_bi_immune_subgroups.tex
\section{Proper Subgroups of the Bi-Immune Subgroup}\label{sect:proper_bi_immune_subgroups}

In this section, non-trivial proper subgroups of the bi-immune subgroup $G_{\mathfrak{B}}$ are established, 
based on a stabilizer rule.   


\begin{defn}
Let $E$ be a subset of $\mathbb{N}$.  
A permutation $\alpha \in G_{\mathfrak{B}}$ is in $H_E$ if and only if $\alpha(E) = E$.
$H_E$ is the set of permutations in $G_{\mathfrak{B}}$ that stabilize $E$.  
\end{defn}

\smallskip

\begin{rem}\label{rem:complement}
$H_E = H_{\overline{E}}$.
\end{rem}

\begin{proof}
Let $\alpha$ be a permutation of $\mathbb{N}$.  If $\alpha(E) = E$, then
$\alpha(\overline{E}) = \alpha\bigl(\mathbb{N} - E\bigr) = \mathbb{N} - \alpha(E)
= \mathbb{N} - E = \overline{E}$.  Exchanging the roles of $E$ and $\overline{E}$
establishes the converse.  Hence $\alpha \in H_E$ if and only if  $\alpha \in H_{\overline{E}}$.
\end{proof}

\smallskip

\begin{lem}[Inside transposition]\label{lem:inside_transposition}
Let $E \subset \mathbb{N}$ and let $m_1 \neq m_2$ be natural numbers.  
If $m_1$ and $m_2$ both lie in $E$, or both lie in $\overline{E}$, then transposition
$(m_1 \ m_2)$ lies in $H_E$.
\end{lem}

\begin{proof}
W.L.O.G., assume $m_1 < m_2$.  By lemma \ref{lem:transposition_in_G}, transposition
$\tau = (m_1 \ m_2)$ lies in $G_{\mathfrak{B}}$.  Let $F$ denote the member of
$\{E, \overline{E}\}$ that contains both $m_1$ and $m_2$.
 Then $\tau(F) = F$ because $\tau$ swaps two elements of
$F$ and fixes every other natural number.

If $F = E$, then $\tau(E) = E$, so $\tau \in H_E$.  If $F = \overline{E}$, then
$\tau(\overline{E}) = \overline{E}$, so $\tau \in H_{\overline{E}} = H_E$ by remark
\ref{rem:complement}.  In both cases, $\tau \in H_E$.
\end{proof}


\begin{lem}[Cross-cut]\label{lem:cross_cut}
Let $E \subset \mathbb{N}$ with $E \neq \emptyset$ and $\overline{E} \neq \emptyset$.
Then there is a transposition $\tau \in G_{\mathfrak{B}}$ with $\tau(E) \neq E$.
\end{lem}

\begin{proof}
Choose $m \in E$ and $n \in \overline{E}$: $m$ and $n$ exist by hypothesis and $m \neq n$.
W.L.O.G., assume $m < n$.  
By lemma \ref{lem:transposition_in_G}, transposition
$\tau = (m \ n)$ lies in $G_{\mathfrak{B}}$.  Since $m \in E$, we have
$n = \tau(m) \in \tau(E)$.  Since $n \notin E$, it follows that $\tau(E) \neq E$.
\end{proof}


\begin{lem}[Non-triviality]
\label{lem:nontriviality}
Let $E \subset \mathbb{N}$ with $E \neq \emptyset$ and $\overline{E} \neq \emptyset$.
Then $H_E$ is infinite.  In particular, $H_E$ contains a non-identity permutation.
\end{lem}

\begin{proof}
Since $E \cup \overline{E} = \mathbb{N}$ and $\mathbb{N}$ is infinite, at least one of
$E$ and $\overline{E}$ is an infinite set.  Let $F$ denote one of the two that is
infinite.  For each pair $m_1 < m_2$ in $F$, lemma \ref{lem:inside_transposition} shows
that transposition $(m_1 \ m_2)$ lies in $H_E$.  Since $F$ is infinite, there are
infinitely many such pairs, and distinct pairs determine distinct transpositions.  Hence
$H_E$ is infinite, and in particular $H_E$ contains a non-identity permutation.
\end{proof}

\begin{thm}\label{thm:proper_bi_immune_subgroup}
Let $E \subset \mathbb{N}$ with $E \neq \emptyset$ and $\overline{E} \neq \emptyset$.
Then $H_E$ is a non-trivial, infinite, proper subgroup of $G_{\mathfrak{B}}$.
\end{thm}

\begin{proof}
First, we prove $H_E$ is a subgroup.  The identity permutation satisfies
$\mathrm{id}(E) = E$, so $H_E \neq \emptyset$.  Let $\alpha, \beta \in H_E$.  Since
$\beta$ is a bijection with $\beta(E) = E$, applying $\beta^{-1}$ to both sides gives
$\beta^{-1}(E) = E$.  Hence
$(\alpha \circ \beta^{-1})(E) = \alpha\bigl(\beta^{-1}(E)\bigr) = \alpha(E) = E$,
so $\alpha \circ \beta^{-1} \in H_E$ and $H_E$ is a subgroup.

Second, $H_E$ is proper.  By the cross-cut lemma \ref{lem:cross_cut} there is a
transposition $\tau \in G_{\mathfrak{B}}$ with $\tau(E) \neq E$.  Thus,
$\tau \in G_{\mathfrak{B}} - H_E$.

Third, $H_E$ is non-trivial and infinite by the non-triviality lemma
\ref{lem:nontriviality}.
\end{proof}


\begin{cor}\label{cor:distinct_stabilizers}
Let $E, F \subset \mathbb{N}$ with $E, \overline{E}, F, \overline{F}$ all non-empty.
Then $H_E = H_F$ if and only if $\{E, \overline{E}\} = \{F, \overline{F}\}$.
\end{cor}

\begin{proof}
If $F = E$, then $H_F = H_E$ trivially.  If $F = \overline{E}$, then
$H_F = H_{\overline{E}} = H_E$ by remark \ref{rem:complement}.  This proves the
converse direction.

For the forward direction, we prove the contrapositive.  Suppose 
$\{E, \overline{E}\} \neq \{F, \overline{F}\}$.  That is, $F \neq E$ and $F \neq \overline{E}$.  

First, we show there exists a pair $(m, n)$ of natural numbers such that $m$ and $n$
lie in the same member of $\{F, \overline{F}\}$, but in different members of
$\{E, \overline{E}\}$.  By reductio ad absurdum, suppose no such pair exists.  Then
each member of $\{F, \overline{F}\}$ is contained in a single member of
$\{E, \overline{E}\}$.  There are four cases.

\begin{itemize}
  \item[(a)] $F \subseteq E$ and $\overline{F} \subseteq E$.  Then
        $\mathbb{N} = F \cup \overline{F} \subseteq E$, so $\overline{E} = \emptyset$.

  \item[(b)] $F \subseteq \overline{E}$ and $\overline{F} \subseteq \overline{E}$.  Then
        $\mathbb{N} \subseteq \overline{E}$, so $E = \emptyset$.

  \item[(c)] $F \subseteq E$ and $\overline{F} \subseteq \overline{E}$.  Taking the set complement 
        of   $\overline{F} \subseteq \overline{E}$  yields $E \subseteq F$, so $E = F$.

  \item[(d)] $F \subseteq \overline{E}$ and $\overline{F} \subseteq E$.  Taking the set complement
          of $F \subseteq \overline{E}$ yields $E \subseteq \overline{F}$, so
          $E = \overline{F}$.  Hence, $F = \overline{E}$.
\end{itemize}

Each case contradicts a hypothesis.  
Hence a pair $(m, n)$ with the stated property exists.

Both $m$ and $n$ lie in $F$, or both lie in $\overline{F}$, so lemma
\ref{lem:inside_transposition}, applied to the set $F$, shows that transposition
$\tau = (m \ n)$ lies in $H_F$.  Exactly one of $m$ and $n$ lies in $E$; W.L.O.G.,
assume $m \in E$ and $n \in \overline{E}$.  Then $n = \tau(m) \in \tau(E)$, and
$n \notin E$, so $\tau(E) \neq E$.  That is, $\tau \notin H_E$.  Therefore
$H_E \neq H_F$.
\end{proof}

\begin{cor}
\label{cor:generates_FS}
Let $E \subset \mathbb{N}$ with $E \neq \emptyset$ and $\overline{E} \neq \emptyset$, and
let $\tau \in G_{\mathfrak{B}}$ be a transposition with $\tau(E) \neq E$.  Then the
subgroup of $G_{\mathfrak{B}}$ generated by $H_E \cup \{\tau\}$ contains the finitary
symmetric group FS$(\mathbb{N})$.
\end{cor}

\begin{proof}
Write $\tau = (m \ n)$.  If $m$ and $n$ both lie in $E$, or both lie in $\overline{E}$,
then $\tau \in H_E$ by lemma \ref{lem:inside_transposition}; that is, $\tau(E) = E$,
contrary to the hypothesis.  Hence exactly one
of $m, n$ lies in $E$.  W.L.O.G., assume $m \in E$ and $n \in \overline{E}$.

Let $i \neq j$ be natural numbers.  If $i, j$ both lie in $E$, or both lie in
$\overline{E}$, then $(i \ j) \in H_E$ by lemma \ref{lem:inside_transposition}.
Otherwise, W.L.O.G., assume $i \in E$ and $j \in \overline{E}$.

Set $\alpha_1 = (m \ i)$ when $m \neq i$ and $\alpha_1 = \mathrm{id}$ when $m = i$. 
Set $\alpha_2 = (n \ j)$ when $n \neq j$ and $\alpha_2 = \mathrm{id}$ when $n = j$.  
Each of $\alpha_1, \alpha_2$ lies in $H_E$, by lemma \ref{lem:inside_transposition} when
it is a transposition and because $\mathrm{id}(E) = E$ otherwise.
Thus, $\alpha = \alpha_1 \circ \alpha_2 \in H_E$.  
The supports of $\alpha_1$ and $\alpha_2$ are disjoint, so $\alpha(m) = i$ and
$\alpha(n) = j$.  Therefore
\[
  \alpha \circ \tau \circ \alpha^{-1} \;=\; \bigl(\alpha(m) \ \ \alpha(n)\bigr) \;=\; (i \ j).
\]

Every transposition lies in the subgroup generated by $H_E \cup \{\tau\}$, and by lemma
\ref{lem:BSym} the transpositions generate FS$(\mathbb{N})$.
\end{proof}


\begin{rem}\label{rem:computable_family}
Theorem \ref{thm:proper_bi_immune_subgroup} does not require that set $E$ is Turing
computable.  A Turing computable $E$ enables a family of subgroups $\{H_E\}$ that is
effectively indexed, each member specified by a Turing computable characteristic
function $\chi_E$.  By remark \ref{rem:complement} and corollary
\ref{cor:distinct_stabilizers}, the assignment $\{E, \overline{E}\} \mapsto H_E$ is
two-to-one on characteristic functions and one-to-one on complementary pairs.  Hence
the Turing computable sets index countably infinitely many \emph{distinct} non-trivial,
infinite, proper subgroups of $G_{\mathfrak{B}}$.
\end{rem}

\smallskip 

\begin{example}
Let $\chi_P: \mathbb{N} \rightarrow \{0, 1\}$ be such that $\chi_P(n) = 1$ if and only
if $n$ is prime.  Then $P = \{ n \in \mathbb{N} \mid \chi_P(n) = 1 \}$ is a Turing
computable set.  $P$ and $\overline{P}$ are both infinite sets.  Therefore, $H_P$ is a
non-trivial, infinite, proper subgroup of $G_{\mathfrak{B}}$.
\end{example}

%% file: bi_immune_density.tex
\section{Density of Bi-Immune Permutations}\label{sect:bi_immune_density}

\begin{defn}[Word]\label{defn:word}
Recall that $\mathfrak{B} = \{\sigma_A : A \text{ is bi-immune}\}$ 
is the set of bi-immune rearrangements.  
A \emph{word of length} $n$ in $\mathfrak{B}$ is a permutation
\[
  \tau \;=\; g_n \circ \cdots \circ g_2 \circ g_1
\]
in which each \emph{letter} $g_i$ is $\sigma_{A_i}$ or $\sigma_{A_i}^{-1}$ for some
bi-immune $A_i$, and the letters are applied in the order
$g_1, g_2, \ldots, g_n$, following the convention of definition
\ref{defn:seq_rearrange}.  The identity permutation is the empty word, of
length $0$.  Let $W_n$ be the set of words of length at most $n$.  
\end{defn}

\begin{defn}[Positive word]\label{defn:positive_word}
A word $\tau = g_n \circ \cdots \circ g_2 \circ g_1$ is \emph{positive} 
if no inverse occurs among its letters.  
That is, $g_i = \sigma_{A_i}$ for every $i$.  
The positive words contain the identity and are closed under composition, 
so they form a submonoid of
$G_{\mathfrak{B}}$, which is denoted $G_{\mathfrak{B}}^{+}$.
\end{defn}

\begin{rem}[Inverse closure]\label{rem:inverse_closure}
The set $\bigcup_{n \in \mathbb{N}} W_n$ is a subgroup of $\Sym(\mathbb{N})$:  the
empty word is the identity; a word of length $m$ composed with a word of length $n$
is a word of length $m+n$; and
$(g_n \circ \cdots \circ g_1)^{-1} = g_1^{-1} \circ \cdots \circ g_n^{-1}$ is again
a word of length $n$, since each $g_i^{-1}$ is $\sigma_{A_i}$ or $\sigma_{A_i}^{-1}$.  
This subgroup contains $\mathfrak{B}$, and every subgroup containing $\mathfrak{B}$ 
contains every word.  Hence
\[
  G_{\mathfrak{B}} \;=\; \bigcup_{n \in \mathbb{N}} W_n ,
\]
so $G_{\mathfrak{B}}$ is the smallest subgroup of $\Sym(\mathbb{N})$ containing
$\mathfrak{B}$, and $G_{\mathfrak{B}}^{+} \subseteq G_{\mathfrak{B}}$.
\end{rem}

\begin{rem}\label{rem:length_not_intrinsic}
A permutation is represented by many different words, so length is a property of
the word and not of the permutation it represents.  Accordingly $\tau \in W_n$ means
that $\tau$ admits \emph{some} representation of length at most $n$, and
$\tau \in G_{\mathfrak{B}}^{+}$ means that $\tau$ admits \emph{some} positive
representation.
\end{rem}

\begin{lem}\label{lem:degrade}
Let $A_1, \ldots, A_n \subseteq \mathbb{N}$ and set
$\tau = \sigma_{A_n} \circ \cdots \circ \sigma_{A_1}$; that is, let $\tau$ be a word
of length $n$ in the rearrangements, using no inverses.  Then $\tau(k) \le k+n$ for
all $k \in \mathbb{N}$.
\end{lem}

\begin{proof}
Induction on $n$.  The case $n = 1$ is lemma \ref{lem:up_iff_in}.  Suppose $\tau'$
is a word of length $n-1$ with $\tau'(k) \le k + n - 1$ for all $k$.  Applying lemma
\ref{lem:up_iff_in} again,
$\sigma_{A_n}\bigl(\tau'(k)\bigr) \le \tau'(k) + 1 \le k + n$.
\end{proof}

\begin{defn}
Let $X$ be a topological space.
A set $E \subseteq X$ is \textit{nowhere dense} if $\mathrm{int}\bigl(\mathrm{cl}(E)\bigr) = \emptyset$.  
That is, the closure of $E$ has empty interior.
A set $A \subseteq X$ is \textit{meager} (or \emph{of the first category}) if $A$ is
contained in a countable union of nowhere dense sets. $A$ is \textit{comeager} if
its complement $X - A$ is meager.
\end{defn}

\begin{lem}\label{lem:Dn_nwd}
For each $n \in \mathbb{N}$, the set
$D_n = \{\sigma \in \Sym(\mathbb{N}) : \sigma(j) \le j+n \text{ for all } j\}$
is nowhere dense in $\Sym(\mathbb{N})$.
\end{lem}

\begin{proof}
$D_n = \bigcap_{j} \{\sigma : \sigma(j) \le j+n\}$, and each set in the intersection
is determined by the single value $\sigma(j)$, so it is both closed and open.  In
particular it is closed, and an intersection of closed sets is closed, so
$\mathrm{cl}(D_n) = D_n$.

For an empty interior, let $U$ be a non-empty basic open set, determined by the values
of $\sigma$ and $\sigma^{-1}$ on $\{0, \ldots, k\}$.  Choose $j > k$ and let $\tau$
be any permutation agreeing with those finitely many constraints and satisfying
$\tau(j) = j + n + 1$; this is possible because the values on $\{k+1, k+2, \ldots\}$
may be assigned freely subject only to bijectivity.  Then $\tau \in U$ and
$\tau \notin D_n$.  Hence $\mathrm{int}(D_n) = \emptyset$.
\end{proof}

\begin{thm}\label{thm:semigroup_meager}
$G_{\mathfrak{B}}^{+}$ is meager in $\Sym(\mathbb{N})$.
\end{thm}

\begin{proof}
Every element of $G_{\mathfrak{B}}^{+}$ is a positive word of some finite length $n$, so
by lemma \ref{lem:degrade} it lies in $D_n$. Hence
\[
   G_{\mathfrak{B}}^{+} \;\subseteq\; \bigcup_{n \in \mathbb{N}} D_n ,
\]
a countable union of nowhere dense sets by lemma \ref{lem:Dn_nwd}. A set contained
in a countable union of nowhere dense sets is meager.
\end{proof}

%% file: bi_immune_crossing_number.tex

\section{ $G_{\mathfrak{B}}$ is a Proper Subgroup of $\Sym(\mathbb{N})$}\label{sec:crossing}

This section proves that $G_{\mathfrak{B}} \ne \Sym(\mathbb{N})$.  The proof
assigns to each permutation of $\mathbb{N}$ a number, called its
\textit{crossing number}, that is at most $1$ on every generator $\sigma_A$ and
on every $\sigma_A^{-1}$, and that grows by at most $1$ with each composition.
Hence every permutation in $G_{\mathfrak{B}}$ has a finite crossing number.
Then a permutation with an infinite crossing number is constructed.

\subsection{The Crossing Number}
\label{sec:crossing_defn}

\begin{defn}[Crossing count, crossing number]
\label{defn:crossing}
Let $\alpha \in \Sym(\mathbb{N})$ and let $c \ge 1$ be a natural number.  The
\textit{crossing count of $\alpha$ at $c$} is the number of points below $c$
that $\alpha$ moves to $c$ or above:
\[
  \chi_c(\alpha) \;=\; \bigl|\, \{\, k \in \mathbb{N} \;:\; k < c
  \text{ and } \alpha(k) \ge c \,\} \,\bigr| .
\]
Since there are only $c$ points below $c$, $\; 0 \le \chi_c(\alpha) \le c$.
The \textit{crossing number} of $\alpha$ is
\[
    \chi(\alpha) \;=\; \sup \,\{\, \chi_c(\alpha) \;:\; c \ge 1 \,\} ,
\]
which is either a natural number or $\infty$.
\end{defn}

\begin{lem}[As many points cross downward as upward]
\label{lem:up_equals_down}
For every $\alpha \in \Sym(\mathbb{N})$ and every $c \ge 1$,
\[
  \bigl|\, \{\, k < c \;:\; \alpha(k) \ge c \,\} \,\bigr|
  \;=\;
  \bigl|\, \{\, k \ge c \;:\; \alpha(k) < c \,\} \,\bigr| .
\]
Consequently $\chi_c(\alpha^{-1}) = \chi_c(\alpha)$ for every $c \ge 1$, and
$\chi(\alpha^{-1}) = \chi(\alpha)$.
\end{lem}

\begin{proof}
Let $a$ be the number of $k < c$ with $\alpha(k) < c$.  The interval $[0, c-1]$
has $c$ points, so the left side equals $c - a$.  Since $\alpha$ is a bijection,
exactly $c$ points $k \in \mathbb{N}$ satisfy $\alpha(k) < c$; $a$ of them are
below $c$, so $c - a$ of them are at $c$ or above.  That is the right side.

For the second statement, $k \ge c$ and $\alpha(k) < c$ if and only if the point
$j = \alpha(k)$ satisfies $j < c$ and $\alpha^{-1}(j) \ge c$.  Thus the right side
equals $\chi_c(\alpha^{-1})$, while the left side is $\chi_c(\alpha)$.
\end{proof}

\begin{lem}[The crossing count is subadditive]
\label{lem:crossing_subadditive}
For all $\alpha, \rho \in \Sym(\mathbb{N})$ and every $c \ge 1$,
\[
  \chi_c(\alpha \circ \rho) \;\le\; \chi_c(\alpha) + \chi_c(\rho) .
\]
\end{lem}

\begin{proof}
Let $k < c$ with $\alpha\bigl(\rho(k)\bigr) \ge c$.  
Either $\rho(k) \ge c$, or $\rho(k) < c$.

Case 1: $\rho(k) \ge c$.  Such a $k$ is one of the points $k < c$ with
$\rho(k) \ge c$, and there are at most $\chi_c(\rho)$ of these points.  Some of
them may have $\alpha\bigl(\rho(k)\bigr) < c$ and so do not count toward
$\chi_c(\alpha \circ \rho)$, so $\chi_c(\rho)$ can only overestimate the
number of points in this case.

Case 2: $\rho(k) < c$.  Then the point $j = \rho(k)$ satisfies $j < c$ and
$\alpha(j) \ge c$.  There are at most $\chi_c(\alpha)$ such points $j$, and
since $\rho$ is one-to-one, at most $\chi_c(\alpha)$ such points $k$.

Combining cases 1 and 2 yields
$\chi_c(\alpha \circ \rho) \le \chi_c(\alpha) + \chi_c(\rho)$.
\end{proof}

\begin{lem}[A generator crosses each $c$ at most once]
\label{lem:generator_crossing}
Let $A \subseteq \mathbb{N}$ have every maximal interval finite.  Then
$\chi_c(\sigma_A) \le 1$ and $\chi_c(\sigma_A^{-1}) \le 1$ for every $c \ge 1$.
\end{lem}

\begin{proof}
By Lemma \ref{lem:orbits_are_blocks}, $\sigma_A(k) \le k+1$ for every
$k \in \mathbb{N}$.  If $k < c$ and $\sigma_A(k) \ge c$, then
$c \le \sigma_A(k) \le k + 1 \le c$, so $k = c-1$.  Hence at most one point
below $c$ is moved to $c$ or above, which is $\chi_c(\sigma_A) \le 1$.  Lemma
\ref{lem:up_equals_down} implies $\chi_c(\sigma_A^{-1}) = \chi_c(\sigma_A) \le 1$.
\end{proof}

\subsection{Every Element of $G_{\mathfrak{B}}$ has a Finite Crossing Number}
\label{sec:crossing_bound}

\begin{thm}[Crossing bound]
\label{thm:crossing_bound}
Let $A_1, \dots, A_n \subseteq \mathbb{N}$ each have every maximal interval
finite, and let
\[
  \tau \;=\; \sigma_{A_{i_1}}^{s_1} \circ
  \sigma_{A_{i_2}}^{s_2} \circ \cdots \circ
  \sigma_{A_{i_r}}^{s_r},
  \qquad i_j \in \{1,\dots,n\}, \;\; s_j \in \{+1,-1\} .
\]
Then $\chi_c(\tau) \le r$ for every $c \ge 1$.  Hence $\chi(\tau) \le r$.
\end{thm}

\begin{proof}
By Lemma \ref{lem:generator_crossing}, $\chi_c\bigl(\sigma_{A_{i_j}}^{s_j}\bigr)
\le 1$ for $j = 1, \dots, r$.  Applying Lemma \ref{lem:crossing_subadditive}
$r - 1$ times implies $\chi_c(\tau) \le r$.
\end{proof}

\begin{cor}
\label{cor:GB_finite_crossing}
Every $\tau \in G_{\mathfrak{B}}$ satisfies $\chi(\tau) < \infty$.
\end{cor}

\begin{proof}
Every element of $G_{\mathfrak{B}}$ is a finite composition of permutations
$\sigma_A$ and $\sigma_A^{-1}$ with $A$ bi-immune, and every maximal interval
of a bi-immune set is finite.  Theorem \ref{thm:crossing_bound} applies.
\end{proof}

\begin{defn}[The group of finite crossing number]
\label{defn:finite_crossing_group}
Set
$\; \mathcal{W} \;=\; \{\, \alpha \in \Sym(\mathbb{N}) \;:\; \chi(\alpha) < \infty \,\}$.
\end{defn}

\begin{lem}
\label{lem:W_is_subgroup}
$\mathcal{W}$ is a subgroup of $\Sym(\mathbb{N})$, and
$FS(\mathbb{N}) \subseteq G_{\mathfrak{B}} \subseteq \mathcal{W}$.
\end{lem}

\begin{proof}
The identity has crossing number $0$.  Lemma \ref{lem:crossing_subadditive}
implies $\chi(\alpha \circ \rho) \le \chi(\alpha) + \chi(\rho)$, and Lemma
\ref{lem:up_equals_down} implies $\chi(\alpha^{-1}) = \chi(\alpha)$, so $\mathcal{W}$ is
closed under composition and inverses.  The first containment is Lemma
\ref{lem:BSym} and the second is Corollary \ref{cor:GB_finite_crossing}.
\end{proof}

\subsection{A Permutation with an Infinite Crossing Number}
\label{sec:crossing_witness}

\begin{defn}[The half-swap permutation $\lambda$]
\label{defn:half_swap}
Define $\lambda : \mathbb{N} \to \mathbb{N}$ by $\lambda(0) = 0$,
$\lambda(1) = 1$, and for every $m \ge 1$ and every $i$ with
$0 \le i < 2^{m-1}$,
\[
  \lambda\bigl( 2^{m} + i \bigr) \;=\; 2^{m} + 2^{m-1} + i ,
  \qquad
  \lambda\bigl( 2^{m} + 2^{m-1} + i \bigr) \;=\; 2^{m} + i .
\]
For each $m \ge 1$, $\lambda$ interchanges the lower half
$[\,2^{m},\; 2^{m} + 2^{m-1} - 1\,]$ and the upper half
$[\,2^{m} + 2^{m-1},\; 2^{m+1} - 1\,]$ of the interval
$[\,2^{m},\, 2^{m+1} - 1\,]$.  These intervals partition
$\{k \in \mathbb{N} : k \ge 2\}$, so $\lambda \in \Sym(\mathbb{N})$, and
$\lambda \circ \lambda$ is the identity.
\end{defn}

\begin{lem}
\label{lem:half_swap_crossing}
For every $m \ge 1$, $\; \chi_c(\lambda) \ge 2^{m-1}$ at
$c = 2^{m} + 2^{m-1}$.  Hence $\chi(\lambda) = \infty$ and
$\lambda \notin \mathcal{W}$.
\end{lem}

\begin{proof}
Let $c = 2^{m} + 2^{m-1}$.  For each $i$ with $0 \le i < 2^{m-1}$, the point
$k = 2^{m} + i$ satisfies $k < c$ and $\lambda(k) = c + i \ge c$.  These are
$2^{m-1}$ distinct points.
\end{proof}

\begin{thm}[$G_{\mathfrak{B}}$ is a proper subgroup of $\Sym(\mathbb{N})$]
\label{thm:GB_proper}
\[
  G_{\mathfrak{B}} \;\subseteq\; \mathcal{W} \;\subsetneq\; \Sym(\mathbb{N}) .
\]
In particular $\lambda \notin G_{\mathfrak{B}}$.  More generally, the subgroup
of $\Sym(\mathbb{N})$ generated by \emph{all} permutations $\sigma_A$, where
$A$ ranges over every subset of $\mathbb{N}$ whose maximal intervals are
finite, is contained in $\mathcal{W}$ and is a proper subgroup of
$\Sym(\mathbb{N})$.
\end{thm}

\begin{proof}
Lemma \ref{lem:W_is_subgroup} implies $G_{\mathfrak{B}} \subseteq \mathcal{W}$,
and Lemma \ref{lem:half_swap_crossing} implies
$\lambda \in \Sym(\mathbb{N})$ and $\lambda \notin \mathcal{W}$.  The last
statement holds because Theorem \ref{thm:crossing_bound} does not use
bi-immunity.
\end{proof}

\begin{cor}
\label{cor:GI_in_W}
\[
  G_{\mathfrak{B}} \;\subseteq\; G_{\mathfrak{I}}  \;\subseteq\; \mathcal{W}
  \;\subsetneq\; \Sym(\mathbb{N}) .
\]
In particular, $\mathcal{C}(\mathbb{N}) \not\subseteq G_{\mathfrak{I}}$ and
$\mathcal{I}(\mathbb{N}) \not\subseteq G_{\mathfrak{I}}$.
\end{cor}

\begin{proof}
Lemma \ref{lem:bi_immune_subgroup_incomputable} states that
$G_{\mathfrak{B}}$ is a subgroup of $G_{\mathfrak{I}}$.

Let $A$ be Turing incomputable.  If a maximal interval of $A$ were infinite,
then $A$ would contain $[m, \infty)$ for some $m$, so $A$ would be cofinite
and therefore computable.  Hence every maximal interval of $A$ is finite, and
Theorem \ref{thm:crossing_bound} applies to every finite composition of
permutations $\sigma_A$ and $\sigma_A^{-1}$ with $A$ incomputable.  Thus, every
element of $G_{\mathfrak{I}}$ has a finite crossing number, which is
$G_{\mathfrak{I}} \subseteq \mathcal{W}$.  Theorem \ref{thm:GB_proper} states
that $\mathcal{W} \subsetneq \Sym(\mathbb{N})$.

The permutation $\lambda$ of Definition \ref{defn:half_swap} is computable,
and $\lambda \notin \mathcal{W}$ by Lemma \ref{lem:half_swap_crossing}.  So
$\lambda \in \mathcal{C}(\mathbb{N})$ and $\lambda \notin G_{\mathfrak{I}}$.
Let $A$ be Turing incomputable.  Lemma \ref{lem:C_composed_I_incomputable}
implies $\lambda \circ \sigma_A \in \mathcal{I}(\mathbb{N})$.  If
$\lambda \circ \sigma_A$ were in $G_{\mathfrak{I}}$, then
$\lambda = (\lambda \circ \sigma_A) \circ \sigma_A^{-1}$ would be in
$G_{\mathfrak{I}}$. Hence $\lambda \circ \sigma_A \notin G_{\mathfrak{I}}$.
\end{proof}

\begin{cor}
The bi-immune subgroup $G_{\mathfrak{B}}$ is not closed in (Sym$(\mathbb{N}),d)$.  
That is,  $G_{\mathfrak{B}}$ is not closed with respect 
to the pointwise convergence topology.  
\end{cor}

\begin{proof}
The corollary immediately follows from Remark \ref{rem:closed_iff_sym_N}.
\end{proof}


\subsection{Consequences of Theorem  \ref{thm:GB_proper} }\label{sec:crossing_consequences}

\begin{thm}[The Index of $G_{\mathfrak{B}}$]\label{thm:index_dichotomy}

\smallskip 

\noindent  
$[\,\Sym(\mathbb{N}) : G_{\mathfrak{B}}\,] = 2^{\aleph_0}$.
\end{thm}

\begin{proof}
Suppose $[\,\Sym(\mathbb{N}) : G_{\mathfrak{B}}\,] < 2^{\aleph_0}$.  By the
theorem of Dixon, Neumann and Thomas \cite{dixon_neumann_thomas}, every subgroup
of $\Sym(\mathbb{N})$ of index less than $2^{\aleph_0}$ is open in the topology
of pointwise convergence.  For a finite set $D \subset \mathbb{N}$, the
\textit{pointwise stabilizer} of $D$ is
\[
  \Sym(\mathbb{N})_{(D)} \;=\;
  \{\, \alpha \in \Sym(\mathbb{N}) \;:\; \alpha(d) = d
  \text{ for all } d \in D \,\} .
\]
In the topology of pointwise convergence, every open set that contains the
identity permutation contains $\Sym(\mathbb{N})_{(D)}$ for some finite $D$.
An open subgroup contains the identity, so
$\Sym(\mathbb{N})_{(D)} \subseteq G_{\mathfrak{B}}$ for some finite
$D \subset \mathbb{N}$.  Fix such a $D$.

By Lemma \ref{lem:BSym}, $FS(\mathbb{N}) \subseteq G_{\mathfrak{B}}$.  Let
$\alpha \in \Sym(\mathbb{N})$ be arbitrary.  The sets $D$ and $\alpha(D)$ are
finite and have the same number of points, so there is a finitary permutation
$\rho \in FS(\mathbb{N})$ with $\rho\bigl(\alpha(d)\bigr) = d$ for every
$d \in D$.  Then $\rho \circ \alpha$ fixes every point of $D$, so
$\rho \circ \alpha \in \Sym(\mathbb{N})_{(D)} \subseteq G_{\mathfrak{B}}$, and
therefore
\[
  \alpha \;=\; \rho^{-1} \circ (\rho \circ \alpha) \;\in\; G_{\mathfrak{B}} .
\]

Hence $G_{\mathfrak{B}} = \Sym(\mathbb{N})$, which contradicts Theorem
\ref{thm:GB_proper}.  So the index is not less than $2^{\aleph_0}$.  The index
is at most $|\Sym(\mathbb{N})| = 2^{\aleph_0}$, so it equals $2^{\aleph_0}$.
\end{proof}

\begin{cor}\label{cor:proper_consequences}

\verb|  |

\smallskip 

  \noindent (i) \hskip 0.5pc  $G_{\mathfrak{B}}$ is not a normal subgroup of
  $\Sym(\mathbb{N})$.

\smallskip 

  \noindent (ii) \hskip 0.3pc  $G_{\mathfrak{B}}$ is dense in $\Sym(\mathbb{N})$, and is
  neither closed nor open.

\smallskip 

  \noindent (iii)  \hskip 0.05pc  There is a permutation $\tau$ of $\mathbb{N}$ with only one orbit,
  equal to $\mathbb{N}$, and $\tau$ $\notin G_{\mathfrak{B}}$.

\smallskip 

  \noindent (iv) \hskip 0.3pc   If $\tau \in G_{\mathfrak{B}}$ and $\chi(\tau) = r$, then every
  expression of $\tau$ as a finite composition of permutations $\sigma_A$ and
  $\sigma_A^{-1}$ uses at least $r$ of them.

\end{cor}

\begin{proof}

\smallskip
\noindent
(i) By Corollary \ref{cor:GB_infinite_cycle}, $G_{\mathfrak{B}}$ contains a
permutation that is not finitary.  By the Baer--Schreier--Ulam theorem
\cite{baer,schreier_ulam}, the only normal subgroup of $\Sym(\mathbb{N})$
containing a permutation that is not finitary is $\Sym(\mathbb{N})$.  Theorem
\ref{thm:GB_proper} implies $G_{\mathfrak{B}}$ is not normal.

\smallskip
\noindent
(ii) $G_{\mathfrak{B}}$ is dense because it contains $FS(\mathbb{N})$.  A
dense closed subgroup is the whole group, so $G_{\mathfrak{B}}$ is not closed.
An open subgroup is closed, so $G_{\mathfrak{B}}$ is not open.

\smallskip
\noindent
(iii) Corollary \ref{cor:class_equiv} and Theorem \ref{thm:GB_proper}: if every
such permutation were in $G_{\mathfrak{B}}$, then $G_{\mathfrak{B}}$ would
equal $\Sym(\mathbb{N})$.

\smallskip
\noindent
(iv) Theorem \ref{thm:crossing_bound}.
\end{proof}

\begin{cor}[$G_{\mathfrak{B}}$ is meager]
\label{cor:GB_meager}
$\mathcal{W}$ is a meager subset of $\Sym(\mathbb{N})$.  Hence
$G_{\mathfrak{B}}$ is a meager subset of $\Sym(\mathbb{N})$.
\end{cor}

\begin{proof}
For each $r \in \mathbb{N}$ set
$W_r = \{\, \alpha \in \Sym(\mathbb{N}) : \chi_c(\alpha) \le r
\text{ for every } c \ge 1 \,\}$, so that
$\mathcal{W} = \bigcup_{r \in \mathbb{N}} W_r$.  It suffices to show that each
$W_r$ is closed and nowhere dense.

\smallskip
\noindent
\textit{$W_r$ is closed.}  Let $\alpha \notin W_r$, so $\chi_c(\alpha) > r$ for some
$c \ge 1$.  The crossing count $\chi_c$ depends only on the values of a
permutation on $[0, c-1]$.  Hence every $\rho$ with $\rho(k) = \alpha(k)$ for all
$k < c$ satisfies $\chi_c(\rho) = \chi_c(\alpha) > r$.  These $\rho$ form an open set
that contains $\alpha$ and is disjoint from $W_r$.

\smallskip
\noindent
\textit{$W_r$ is nowhere dense.}  Let $\alpha \in \Sym(\mathbb{N})$ and
$M \ge 1$, and let $U$ be the open set of all $\rho$ with
$\rho(k) = \alpha(k)$ for all $k < M$.  Choose $L \ge M$ with $L > \alpha(k)$ for
all $k < M$, and set $c = L + r + 1$.  The conditions
\[
  \rho(k) = \alpha(k) \;\; (k < M),
  \qquad
  \rho(L + i) = c + i \;\; (0 \le i \le r)
\]
assign distinct values to finitely many distinct points, so some
$\rho \in \Sym(\mathbb{N})$ satisfies them.  Every permutation that agrees
with $\rho$ on $[0, c-1]$ lies in $U$, because $c > M$, and moves the $r+1$
points $L, L+1, \dots, L+r$, which are below $c$, to $c$ or above.  These
permutations form a nonempty open subset of $U$ that is disjoint from $W_r$.
\end{proof}

\begin{rem}[Bi-immunity is not used]
\label{rem:no_computability_in_obstruction}
The proof of Theorem \ref{thm:GB_proper} uses only the inequality
$\sigma_A(k) \le k + 1$.  It holds for every $A$ whose maximal intervals are
finite, computable or not.  Hence the reason that $G_{\mathfrak{B}}$ is proper is
the way a rearrangement $\sigma_A$ is built from transpositions of adjacent
positions, and not a property of bi-immune sets.
\end{rem}

\begin{rem}[Two questions that replace the original one]
\label{rem:new_questions}
Theorem \ref{thm:GB_proper} locates $G_{\mathfrak{B}}$ between two explicit
groups:
\[
  FS(\mathbb{N}) \;\subsetneq\; G_{\mathfrak{B}} \;\subseteq\; \mathcal{W}
  \;\subsetneq\; \Sym(\mathbb{N}) .
\]
The first containment is proper by Corollary \ref{cor:GB_infinite_cycle}.
Two questions are open.
\begin{enumerate}
  \item[(a)] Is $G_{\mathfrak{B}} = \mathcal{W}$?  That is, is every
  permutation of finite crossing number a finite composition of bi-immune
  rearrangements and their inverses?
  \item[(b)] Is the subgroup generated by all $\sigma_A$, with no condition of
  bi-immunity on $A$, equal to $\mathcal{W}$?
\end{enumerate}
Question (b) contains no computability.  Question (a) is where bi-immunity
may matter: it asks whether restricting the sets $A$ to bi-immune sets removes 
any permutations with a finite crossing number.
\end{rem}

%% file: bi_immune_dyadic_subgroup.tex

\section{Dyadic Disjoint-Pair Rearrangement Map}\label{sec:dyadic}

By lemma \ref{lem:generator_crossing}, every generator $\sigma_A$ has crossing count
$\chi_c(\sigma_A) \le 1$ at every $c$, because $\sigma_A$ is built from transpositions
of adjacent positions.  This section replaces adjacent transpositions by disjoint
transpositions whose distances grow across dyadic scales.

Throughout this section, $\mathcal{B}$ denotes the set of bi-immune subsets of
$\mathbb{N}$, identified with the bi-immune binary sequences.

\begin{defn}[Dyadic levels]\label{defn:dyadic_index}

For $m \ge 1$, set
\[
  K_m \;=\; \bigl[\, 2^{m-1} - 1, \; 2^m - 2 \,\bigr], \qquad |K_m| = 2^{m-1} .
\]
The levels $K_1, K_2, \dots$ partition $\mathbb{N}$.  Each $k \in \mathbb{N}$ lies in
the level $m = \lfloor \log_2(k+1) \rfloor + 1$ and has the offset
$i = k - (2^{m-1} - 1)$, with $0 \le i < 2^{m-1}$.  Define the primitive recursive
bijection $\pi : \mathbb{N} \to \{(m, i) : m \ge 1, \; 0 \le i < 2^{m-1}\}$ by
$\pi(k) = (m, i)$.
\end{defn}

\begin{defn}[Dyadic disjoint-pair permutations $\tau_A$]
\label{defn:tau_A}
Set $I_0 = \{0, 1\}$, and for $m \ge 1$ set $I_m = [\, 2^m, \; 2^{m+1} - 1 \,]$, an
interval of $2^m$ points.  The intervals $I_0, I_1, I_2, \dots$ partition
$\mathbb{N}$, and $[\,0, \, 2^m - 1\,] = I_0 \cup I_1 \cup \dots \cup I_{m-1}$.
For each $k \in \mathbb{N}$ with $\pi(k) = (m, i)$, define the transposition
$t_k = (x_k \;\; y_k)$ inside $I_m$ by
\[
  x_k = 2^m + i, \qquad y_k = 2^m + 2^{m-1} + i .
\]
For $A \subseteq \mathbb{N}$, define the \emph{dyadic rearrangement}
$\tau_A : \mathbb{N} \to \mathbb{N}$ by
\[
  \tau_A(n) = \begin{cases}
    y_k & \text{if } n = x_k \text{ and } k \in A, \\
    x_k & \text{if } n = y_k \text{ and } k \in A, \\
    n   & \text{otherwise.}
  \end{cases}
\]
\end{defn}

\begin{lem}[Algebraic properties]
\label{lem:tau_basic}

\leavevmode

For all $A, B \subseteq \mathbb{N}$:
\begin{enumerate}
  \item $\tau_A \in \Sym(\mathbb{N})$ and $\tau_A \circ \tau_A$ is the identity.
  \item $A \neq B$ implies $\tau_A \neq \tau_B$.
  \item $\tau_A \circ \tau_B = \tau_{A \oplus B}$.  In particular $\tau_A$ and $\tau_B$
        commute.
  \item $\tau_{\mathbb{N}} = \lambda$, the half-swap permutation of definition
        \ref{defn:half_swap}.
  \item $\tau_A(I_m) = I_m$ for every $m \ge 0$.
\end{enumerate}
\end{lem}

\begin{proof}
The pairs $\{x_k, y_k\}$, $k \in \mathbb{N}$, are pairwise disjoint and their union is
$\mathbb{N} - \{0,1\}$, so every $n$ is moved by at most one $t_k$.  Hence $\tau_A$ is a
bijection and $\tau_A \circ \tau_A$ is the identity, which is (1); no limit is needed.

(2) If $k \in A \oplus B$, then either  $\tau_A(x_k)$ or $\tau_B(x_k)$ equals $y_k$,
and the other one equals $x_k$.

(3) Both sides fix $0$ and $1$ and map each pair $\{x_k, y_k\}$ to itself.  On that pair,
$\tau_A \circ \tau_B$ applies $t_k$ once for each of the conditions $k \in B$ and
$k \in A$ that holds.  Since $t_k \circ t_k$ is the identity, the pair is swapped if and
only if one condition holds, that is, if and only if $k \in A \oplus B$.

(4) When $A = \mathbb{N}$, every $t_k$ is applied, so for each $m \ge 1$ the lower half
$[\,2^m, \, 2^m + 2^{m-1} - 1\,]$ of $I_m$ is interchanged with the upper half
$[\,2^m + 2^{m-1}, \, 2^{m+1} - 1\,]$ by $n \mapsto n + 2^{m-1}$.  This is definition
\ref{defn:half_swap}.

(5) Each $t_k$ with $\pi(k) = (m,i)$ maps $I_m$ to itself.
\end{proof}

\section{Preservation of Turing Degrees}\label{sec:dyadic_turing}

Lemma \ref{lem:incomputable_sequences_generate_incomputable_permutations} shows that
$\sigma_A$ is Turing incomputable whenever $A$ is.  The dyadic map
$A \mapsto \tau_A$ satisfies the stronger property that $A$ and $\tau_A$ have the same
Turing degree, for every subset $A \subseteq \mathbb{N}$.

\begin{thm}[$A \equiv_T \tau_A$]
\label{thm:tau_turing}
For every $A \subseteq \mathbb{N}$, $A \le_T \tau_A$ and $\tau_A \le_T A$.  Hence
$A \equiv_T \tau_A$.
\end{thm}

\begin{proof}
$A \le_T \tau_A$: given $k$, compute $(m, i) = \pi(k)$, $x_k = 2^m + i$ and
$y_k = 2^m + 2^{m-1} + i$.  Query the oracle for $\tau_A$ at $x_k$.  Then
$k \in A$ if and only if $\tau_A(x_k) = y_k$.

$\tau_A \le_T A$: given $n$, if $n \in \{0,1\}$ then $\tau_A(n) = n$.  For $n \ge 2$,
compute $m = \lfloor \log_2 n \rfloor$, so that $n \in I_m$.  If
$n < 2^m + 2^{m-1}$, set $i = n - 2^m$ and $k = (2^{m-1} - 1) + i$, so $n = x_k$;
query the oracle for $A$ at $k$, and output $n + 2^{m-1}$ if $k \in A$ and $n$
otherwise.  If $n \ge 2^m + 2^{m-1}$, set $i = n - 2^m - 2^{m-1}$ and
$k = (2^{m-1}-1) + i$, so $n = y_k$; output $n - 2^{m-1}$ if $k \in A$ and $n$
otherwise.
\end{proof}

\section{Escaping the Crossing-Number Barrier $\mathcal{W}$}\label{sec:dyadic_escape}

Recall the crossing count $\chi_c$, the crossing number $\chi$, and the subgroup
$\mathcal{W} = \{\alpha \in \Sym(\mathbb{N}) : \chi(\alpha) < \infty\}$ of definition
\ref{defn:crossing} and section \ref{sec:crossing}.  Set
\[
  G_{\mathrm{dyadic}} \;=\; \bigl\langle \{\tau_A : A \in \mathcal{B}\} \bigr\rangle .
\]

\begin{lem}[The crossing number of $\tau_A$]
\label{lem:tau_crossing}
For every $A \subseteq \mathbb{N}$,
$\; \chi(\tau_A) = \sup_{m \ge 1} |A \cap K_m|$.
\end{lem}

\begin{proof}
($\ge$) Fix $m \ge 1$ and set $c_m = 2^m + 2^{m-1}$.  For every $k \in A \cap K_m$
with $\pi(k) = (m, i)$, the point $x_k = 2^m + i$ satisfies $x_k < c_m$ and
$\tau_A(x_k) = y_k = c_m + i \ge c_m$.  Distinct $k$ give distinct $x_k$, so
$\chi_{c_m}(\tau_A) \ge |A \cap K_m|$.

($\le$) Fix $c \ge 1$.  If $n < c$ and $\tau_A(n) \ge c$, then $n$ is moved, so
$n = x_k$ and $\tau_A(n) = y_k$ for some $k \in A$, and $x_k < c \le y_k$.  By lemma
\ref{lem:tau_basic}(5), $x_k$ and $y_k$ lie in the same interval $I_m$, where
$\pi(k) = (m, i)$.  Then $c - 1 \in [x_k, y_k - 1] \subseteq I_m$, so all such $k$ have
the same $m$.  Hence there are at most $|A \cap K_m|$ such points, and
$\chi_c(\tau_A) \le \sup_m |A \cap K_m|$.
\end{proof}

\begin{defn}[Hyperimmune, bi-hyperimmune]
\label{defn:bi_hyperimmune}
A \textit{computable disjoint strong array} is a sequence $F_0, F_1, \dots$ of pairwise
disjoint finite subsets of $\mathbb{N}$ such that a Turing machine, on input $n$, outputs
the finite set $F_n$.  An infinite set $S \subseteq \mathbb{N}$ is
\textit{hyperimmune} if there is no computable disjoint strong array with
$F_n \cap S \ne \emptyset$ for every $n$.  A set $A$ is \textit{bi-hyperimmune} if both
$A$ and $\overline{A}$ are hyperimmune.
\end{defn}

Every hyperimmune set is immune, since an infinite c.e. set $W \subseteq S$ contains an
infinite computable set, whose elements form a computable disjoint strong array of
singletons.  Hence every bi-hyperimmune set is bi-immune.  Every $1$-generic set is
bi-hyperimmune.  A Martin-L\"of random set of hyperimmune-free degree, which exists by
the hyperimmune-free basis theorem, is bi-immune but not hyperimmune \cite{downey}.

\begin{thm}[Hyperimmune complements force unbounded crossing]
\label{thm:escape_W}
If $\overline{A}$ is hyperimmune (in particular, if $A$ is bi-hyperimmune), then
$K_m \subseteq A$ for infinitely many $m$.  Hence $\chi(\tau_A) = \infty$ and
$\tau_A \notin \mathcal{W}$.
\end{thm}

\begin{proof}
For each $N$, the levels $K_N, K_{N+1}, \dots$ form a computable disjoint strong array.
Since $\overline{A}$ is hyperimmune, some $K_m$ with $m \ge N$ satisfies
$K_m \cap \overline{A} = \emptyset$, that is, $K_m \subseteq A$.  As $N$ is arbitrary,
this holds for infinitely many $m$, and then $|A \cap K_m| = 2^{m-1}$ is unbounded.
Lemma \ref{lem:tau_crossing} completes the proof.
\end{proof}

\begin{lem}[$\lambda$ lies in $G_{\mathrm{dyadic}}$]
\label{lem:lambda_in_Gdyadic}
For every $A \in \mathcal{B}$, $\; \tau_A \circ \tau_{\overline{A}} = \lambda$.  Hence
$\lambda \in G_{\mathrm{dyadic}}$ and $G_{\mathrm{dyadic}} \not\subseteq \mathcal{W}$.
\end{lem}

\begin{proof}
A set is bi-immune if and only if its complement is, so $\tau_A$ and
$\tau_{\overline{A}}$ are both generators of $G_{\mathrm{dyadic}}$.  By lemma
\ref{lem:tau_basic}(3) and (4),
$\tau_A \circ \tau_{\overline{A}} = \tau_{A \oplus \overline{A}} = \tau_{\mathbb{N}}
= \lambda$.  By lemma \ref{lem:half_swap_crossing}, $\lambda \notin \mathcal{W}$.
\end{proof}

Lemma \ref{lem:lambda_in_Gdyadic} shows that $G_{\mathrm{dyadic}}$ escapes $\mathcal{W}$
however sparse the individual bi-immune sets are.  In fact $G_{\mathrm{dyadic}}$
contains $\tau_C$ for \emph{every} $C \subseteq \mathbb{N}$.  The proof uses the Baire
category theorem on $\{0,1\}^{\mathbb{N}}$ with the product topology, whose basic open
sets are the cylinders $[w]$ of sequences that begin with a finite binary word $w$.

\begin{lem}[Bi-immune sets are comeager]
\label{lem:BI_comeager}
$\mathcal{B}$ is a comeager subset of $\{0,1\}^{\mathbb{N}}$.
\end{lem}

\begin{proof}
Let $W_0, W_1, W_2, \dots$ be an enumeration of all c.e. sets.  A set $X$ fails to be
bi-immune only if $X$ is finite, or $\overline{X}$ is finite, or for some $e$ with
$W_e$ infinite, $W_e \subseteq X$ or $W_e \subseteq \overline{X}$.  The finite and
cofinite sets form a countable set, which is meager because each single point is
closed and nowhere dense.  For an infinite $W_e$, the set
$N_e = \{X : W_e \subseteq X\} = \bigcap_{w \in W_e} \{X : w \in X\}$ is closed.  It
is nowhere dense: given a cylinder $[w]$ with $w$ of length $\ell$, choose
$u \in W_e$ with $u \ge \ell$, and extend $w$ to a word $w'$ of length $u+1$ whose last
bit is $0$.  Then $[w'] \subseteq [w]$ and $[w'] \cap N_e = \emptyset$.  The same
argument applies to $\{X : W_e \subseteq \overline{X}\}$, with last bit $1$.  Hence
the complement of $\mathcal{B}$ is contained in a countable union of nowhere dense sets.
\end{proof}

\begin{thm}[Bi-immune products of disjoint transpositions]\label{thm:disjoint_pairs_all}
Let $(u_k \;\; v_k)$, $k \in \mathbb{N}$, be transpositions of $\mathbb{N}$ whose
supports $\{u_k, v_k\}$ are pairwise disjoint.  For $C \subseteq \mathbb{N}$, let
$\rho_C$ be the permutation that interchanges $u_k$ and $v_k$ for every $k \in C$ and
fixes every other point.  Then
\[
  \bigl\langle \{\rho_A : A \in \mathcal{B}\} \bigr\rangle
  \;=\; \{\, \rho_C : C \subseteq \mathbb{N} \,\} ,
\]
and $C \mapsto \rho_C$ is an isomorphism from the group
$\bigl(\mathcal{P}(\mathbb{N}), \oplus\bigr)$ onto this group.  More precisely, for every
$C \subseteq \mathbb{N}$ there is $A \in \mathcal{B}$ with $A \oplus C \in \mathcal{B}$,
and then $\rho_C = \rho_A \circ \rho_{A \oplus C}$.  In particular
$G_{\mathrm{dyadic}} = \{\tau_C : C \subseteq \mathbb{N}\}$, an elementary abelian
$2$-group of cardinality $2^{\aleph_0}$.
\end{thm}

\begin{proof}
The proof of lemma \ref{lem:tau_basic}(1)--(3) uses only the disjointness of the supports,
so each $\rho_C$ is an involution in $\Sym(\mathbb{N})$, $C \mapsto \rho_C$ is
one-to-one, and $\rho_A \circ \rho_B = \rho_{A \oplus B}$.  Hence
$\{\rho_C : C \subseteq \mathbb{N}\}$ is a subgroup, $C \mapsto \rho_C$ is an
isomorphism onto it, and it contains the subgroup generated by
$\{\rho_A : A \in \mathcal{B}\}$.

Conversely, fix $C \subseteq \mathbb{N}$.  The map $h_C(X) = X \oplus C$ is a
homeomorphism of $\{0,1\}^{\mathbb{N}}$, because it changes each bit of $X$ according to
the corresponding bit of $C$ alone, and $h_C \circ h_C$ is the identity.  By lemma
\ref{lem:BI_comeager}, $\mathcal{B}$ and $h_C^{-1}(\mathcal{B})$ are both comeager.  By
the Baire category theorem, two comeager subsets of the compact metric space
$\{0,1\}^{\mathbb{N}}$ intersect, so there is a bi-immune $A$ such that $A \oplus C$ is
bi-immune.  Since $A \oplus (A \oplus C) = C$,
$\rho_C = \rho_A \circ \rho_{A \oplus C}$.
For $G_{\mathrm{dyadic}}$, take $(u_k \; v_k) = (x_k \; y_k)$.
\end{proof}

\begin{rem}[Bi-immunity is not used]\label{rem:dyadic_no_computability}
As in remark \ref{rem:no_computability_in_obstruction}, the group $G_{\mathrm{dyadic}}$
does not depend on bi-immunity: theorem \ref{thm:disjoint_pairs_all} shows that it is
generated by all $\tau_C$, $C \subseteq \mathbb{N}$.  The proof uses only that
$\mathcal{B}$ is comeager, so the same holds with $\mathcal{B}$ replaced by any comeager
family of sets, for instance the Turing incomputable sets or the immune sets.  Choosing
$\mathcal{B}$ guarantees that every \emph{generator} is bi-immune.
\end{rem}

\begin{rem}[A sparse bi-immune set]
\label{rem:sparse_bi_immune}
Bi-immunity alone does not force an individual $\tau_A$ out of $\mathcal{W}$: there is a
bi-immune $A$ with $|A \cap K_m| \le 1$ for every $m$, so that $\chi(\tau_A) \le 1$ by
lemma \ref{lem:tau_crossing}.  Build $A$ in stages, using levels $K_m$ that increase
from stage to stage.  At stage $e$, let $\ell$ exceed every level used so far.  If $W_e$
is finite, put one point of $K_{\ell+1}$ into $A$.  If $W_e$ is infinite, choose
$u \in W_e$ in a level $m_u > \ell$ and put $u$ into $A$; then choose $v \in W_e$ in a
level $m_v > m_u$, put no point of $K_{m_v}$ into $A$, and use only levels above $m_v$
at later stages.  Then $A$ and $\overline{A}$ are infinite, $A$ meets each level in at
most one point, and every infinite c.e. set $W_e$ meets both $A$ (at $u$) and
$\overline{A}$ (at $v$).  Hence $A$ and $\overline{A}$ are immune.  By theorem
\ref{thm:escape_W}, $\overline{A}$ is not hyperimmune.
\end{rem}

\newpage 

\section{Joint Subgroup $G_{\mathrm{joint}}$ and Group $\mathcal{W}_{\mathrm{dyad}}$}
\label{sec:joint}

Define the joint subgroup generated by both kinds of rearrangement:
\[
  G_{\mathrm{joint}} \;=\;
  \bigl\langle \{\sigma_A, \; \tau_A : A \in \mathcal{B}\} \bigr\rangle .
\]
By lemma \ref{lem:tau_basic}(5), $G_{\mathrm{dyadic}}$ is intransitive.  The group
$G_{\mathrm{joint}}$ contains $G_{\mathfrak{B}}$, so it contains $FS(\mathbb{N})$ by
lemma \ref{lem:BSym}; hence it is highly transitive and dense in $\Sym(\mathbb{N})$.

\begin{defn}[Dyadic crossing number]\label{defn:dyadic_crossing}
For $\alpha \in \Sym(\mathbb{N})$, define the \emph{dyadic crossing number}
\[
  \chi_{\mathrm{dyad}}(\alpha) \;=\; \sup_{m \ge 1} \chi_{2^m}(\alpha) ,
\]
and set
$\mathcal{W}_{\mathrm{dyad}} = \{\alpha \in \Sym(\mathbb{N}) :
\chi_{\mathrm{dyad}}(\alpha) < \infty\}$.
\end{defn}

\begin{lem}\label{lem:Wdyad_subgroup}
$\mathcal{W}_{\mathrm{dyad}}$ is a subgroup of $\Sym(\mathbb{N})$, and
$\mathcal{W} \subseteq \mathcal{W}_{\mathrm{dyad}}$.
\end{lem}

\begin{proof}
By lemmas \ref{lem:up_equals_down} and \ref{lem:crossing_subadditive},
$\chi_c(\alpha^{-1}) = \chi_c(\alpha)$ and
$\chi_c(\alpha \circ \rho) \le \chi_c(\alpha) + \chi_c(\rho)$ for every $c \ge 1$.
Taking the supremum over $c = 2^m$ yields 
$\chi_{\mathrm{dyad}}(\alpha^{-1}) = \chi_{\mathrm{dyad}}(\alpha)$ and
$\chi_{\mathrm{dyad}}(\alpha \circ \rho) \le
\chi_{\mathrm{dyad}}(\alpha) + \chi_{\mathrm{dyad}}(\rho)$.  The identity has dyadic
crossing number $0$.  Finally $\chi_{\mathrm{dyad}}(\alpha) \le \chi(\alpha)$.
\end{proof}

\begin{thm}[$G_{\mathrm{joint}}$ lies inside $\mathcal{W}_{\mathrm{dyad}}$]
\label{thm:Gjoint_in_Wdyad}
\[
  FS(\mathbb{N}) \;\subsetneq\; G_{\mathfrak{B}} \;\subsetneq\; G_{\mathrm{joint}}
  \;\subseteq\; \mathcal{W}_{\mathrm{dyad}} \;\subsetneq\; \Sym(\mathbb{N}) .
\]
\end{thm}

\begin{proof}
The first containment is lemma \ref{lem:BSym}.

$G_{\mathfrak{B}} \subsetneq G_{\mathrm{joint}}$: by lemma
\ref{lem:lambda_in_Gdyadic}, $\lambda \in G_{\mathrm{dyadic}} \subseteq
G_{\mathrm{joint}}$, while $\lambda \notin \mathcal{W}$ and
$G_{\mathfrak{B}} \subseteq \mathcal{W}$ by theorem \ref{thm:GB_proper}.

$G_{\mathrm{joint}} \subseteq \mathcal{W}_{\mathrm{dyad}}$: let $A \in \mathcal{B}$ and
$m \ge 1$.  By lemma \ref{lem:tau_basic}(5), $\tau_A$ maps
$[\,0, \, 2^m - 1\,] = I_0 \cup \dots \cup I_{m-1}$ onto itself, so
$\chi_{2^m}(\tau_A) = 0$.  By lemma \ref{lem:generator_crossing},
$\chi_{2^m}(\sigma_A) \le 1$ and $\chi_{2^m}(\sigma_A^{-1}) \le 1$.  By lemma
\ref{lem:crossing_subadditive}, every word of length $r$ in
$\sigma_A, \sigma_A^{-1}, \tau_A$, $A \in \mathcal{B}$, has dyadic crossing number at most $r$.

$\mathcal{W}_{\mathrm{dyad}} \subsetneq \Sym(\mathbb{N})$: define the involution $\xi$
that fixes $0$ and $1$, and for every $j \ge 1$ interchanges
$I_{2j-1} = [\,2^{2j-1}, \, 2^{2j} - 1\,]$ with the lower half
$[\,2^{2j}, \, 2^{2j} + 2^{2j-1} - 1\,]$ of $I_{2j}$ by $n \leftrightarrow n + 2^{2j-1}$,
and fixes the upper half of $I_{2j}$.  At $c = 2^{2j}$, all $2^{2j-1}$ points of
$I_{2j-1}$ lie below $c$ and are mapped to $c$ or above, so
$\chi_{2^{2j}}(\xi) \ge 2^{2j-1}$.  Hence $\xi \notin \mathcal{W}_{\mathrm{dyad}}$.
\end{proof}

\begin{thm}[Factorization of $\mathcal{W}_{\mathrm{dyad}}$]\label{thm:Wdyad_factorization}


Let
\[
  \mathcal{P}_{\mathrm{dyad}} \;=\; \{\beta \in \Sym(\mathbb{N}) : \beta(0) = 0, \;
  \beta(1) = 1, \text{ and } \beta(I_m) = I_m \text{ for all } m \ge 1\} ,
\]
which is isomorphic to the unrestricted direct product
$\prod_{m=1}^{\infty} \Sym(I_m)$ of the symmetric groups on the intervals $I_m$.  Then


\[
  \mathcal{W}_{\mathrm{dyad}} \;=\; \mathcal{P}_{\mathrm{dyad}} \circ \mathcal{W}
  \;=\; \{\beta \circ \omega : \beta \in \mathcal{P}_{\mathrm{dyad}}, \;
  \omega \in \mathcal{W}\} .
\]
\end{thm}

\begin{proof}
($\supseteq$) Every $\beta \in \mathcal{P}_{\mathrm{dyad}}$ maps
$[\,0, \, 2^m - 1\,]$ onto itself, so $\chi_{\mathrm{dyad}}(\beta) = 0$.  Hence
$\mathcal{P}_{\mathrm{dyad}} \subseteq \mathcal{W}_{\mathrm{dyad}}$, and
$\mathcal{W} \subseteq \mathcal{W}_{\mathrm{dyad}}$ by lemma \ref{lem:Wdyad_subgroup}.
Since $\mathcal{W}_{\mathrm{dyad}}$ is a group, it contains every $\beta \circ \omega$.

($\subseteq$) Let $\alpha \in \mathcal{W}_{\mathrm{dyad}}$ and
$R = \chi_{\mathrm{dyad}}(\alpha) < \infty$.  For $m \ge 1$ set
\[
  E_m = \{n \in I_m : \alpha(n) \notin I_m\}, \qquad
  F_m = \{y \in I_m : \alpha^{-1}(y) \notin I_m\},
\]
$S_m = I_m - E_m$ and $T_m = I_m - F_m$.

\smallskip
\noindent\textit{Step 1: $\alpha(S_m) = T_m$ and $|E_m| = |F_m| \le 2R$.}  A point
$y$ lies in $\alpha(S_m)$ if and only if $y \in I_m$ and $\alpha^{-1}(y) \in I_m$, that
is, if and only if $y \in T_m$.  Hence $|S_m| = |T_m|$ and $|E_m| = |F_m|$.  A point
$n \in E_m$ satisfies either $\alpha(n) < 2^m \le n$ or $n < 2^{m+1} \le \alpha(n)$.
By lemma \ref{lem:up_equals_down}, there are at most $\chi_{2^m}(\alpha) \le R$ points
of the first kind, and by definition there are at most
$\chi_{2^{m+1}}(\alpha) \le R$ points of the second kind.

\smallskip
\noindent\textit{Step 2: the permutation $\omega$.}  Set
$X = I_0 \cup \bigcup_{m \ge 1} E_m$ and $Y = I_0 \cup \bigcup_{m \ge 1} F_m$.  Then
$\mathbb{N}$ is the disjoint union of $X$ and the sets $S_m$, and also the disjoint
union of $Y$ and the sets $T_m$.  Since $\alpha$ is a bijection and
$\alpha(S_m) = T_m$ for every $m$, it follows that $\alpha(X) = Y$.  Define
\[
  \omega(n) = \begin{cases}
    \alpha(n) & n \in X, \\
    \text{the image of $n$ under the increasing bijection } S_m \to T_m
      & n \in S_m .
  \end{cases}
\]
$\omega$ maps $X$ bijectively onto $Y$ and each $S_m$ bijectively onto $T_m$, so
$\omega \in \Sym(\mathbb{N})$.

\smallskip
\noindent\textit{Step 3: $\chi(\omega) \le 5R + 1$.}  If $c = 1$, then
$\chi_1(\omega) \le 1$.  If $c = 2$, then $\omega = \alpha$ on $I_0 = [0,1]$, so
$\chi_2(\omega) = \chi_2(\alpha) \le R$.  Let $c \ge 3$, and choose $m \ge 1$ with
$2^m < c \le 2^{m+1}$.  Every $n < c$ with $\omega(n) \ge c$ is in one of three cases.
\begin{enumerate}
  \item $n < 2^m$.  If $n \in S_j$ for some $j < m$, then $\omega(n) \in I_j$, so
  $\omega(n) < 2^m < c$.  Hence $n \in X$, and $\alpha(n) = \omega(n) \ge c > 2^m$.
  There are at most $\chi_{2^m}(\alpha) \le R$ such points.
  \item $n \in E_m$.  There are at most $|E_m| \le 2R$ such points.
  \item $n \in S_m$.  Let $r$ be the number of points of $S_m$ below $n$, so
  $r \le n - 2^m$.  Then $\omega(n)$ is the element of $T_m$ with  $r$
  elements of $T_m$ below it.  At most $|F_m|$ points of $I_m$ are missing from $T_m$,
  so $\omega(n) \le 2^m + r + |F_m| \le n + 2R$.  Thus $\omega(n) \ge c$ forces
  $c - 2R \le n < c$, which allows at most $2R$ points.
\end{enumerate}
Hence $\chi_c(\omega) \le 5R$ for every $c \ge 3$, and $\omega \in \mathcal{W}$.

\smallskip
\noindent\textit{Step 4: $\beta = \alpha \circ \omega^{-1} \in
\mathcal{P}_{\mathrm{dyad}}$.}  If $y \in Y$, then $\omega^{-1}(y) \in X$, where
$\omega = \alpha$, so $\beta(y) = y$.  In particular $\beta(0) = 0$ and $\beta(1) = 1$.
If $y \in T_m$, then $\omega^{-1}(y) \in S_m$ and $\alpha(S_m) = T_m$, so
$\beta(T_m) = T_m$.  Since $I_m = T_m \cup F_m$ and $\beta$ fixes $F_m$ pointwise,
$\beta(I_m) = I_m$.  Therefore $\beta \in \mathcal{P}_{\mathrm{dyad}}$ and
$\alpha = \beta \circ \omega$.
\end{proof}

\begin{cor}[Properties of $G_{\mathrm{joint}}$]\label{cor:Gjoint_properties}
$G_{\mathrm{joint}}$ is dense in $\Sym(\mathbb{N})$, is neither closed nor open, is
not a normal subgroup of $\Sym(\mathbb{N})$, is meager in $\Sym(\mathbb{N})$, and
$[\,\Sym(\mathbb{N}) : G_{\mathrm{joint}}\,] = 2^{\aleph_0}$.
\end{cor}

\begin{proof}
$G_{\mathrm{joint}} \supseteq FS(\mathbb{N})$ is dense, and
$G_{\mathrm{joint}} \ne \Sym(\mathbb{N})$ by theorem \ref{thm:Gjoint_in_Wdyad}.  The
proofs of theorem \ref{thm:index_dichotomy} and of corollary
\ref{cor:proper_consequences}(ii) use only $FS(\mathbb{N}) \subseteq G$ and
$G \ne \Sym(\mathbb{N})$, so they apply verbatim with $G_{\mathrm{joint}}$ in place of
$G_{\mathfrak{B}}$.  The proof of corollary \ref{cor:proper_consequences}(i) also needs
a permutation in $G$ that is not finitary; $\lambda \in G_{\mathrm{joint}}$ is one.

Meager: it suffices to show $\mathcal{W}_{\mathrm{dyad}}$ is meager.  For
$r \in \mathbb{N}$ let $V_r = \{\alpha : \chi_{2^m}(\alpha) \le r \text{ for all }
m \ge 1\}$, so $\mathcal{W}_{\mathrm{dyad}} = \bigcup_r V_r$.  Each $V_r$ is closed,
since $\chi_{2^m}(\alpha)$ depends only on $\alpha$ restricted to $[\,0, \, 2^m - 1\,]$.
Each $V_r$ is nowhere dense: let $U$ be the open set of permutations agreeing with a
given $\alpha$ on $[\,0, \, M-1\,]$.  Choose $L \ge M$ with $L > \alpha(k)$ for all
$k < M$, and choose $m$ with $2^m > L + r$.  Some $\rho \in \Sym(\mathbb{N})$ agrees
with $\alpha$ on $[\,0, \, M-1\,]$ and satisfies $\rho(L + i) = 2^m + i$ for
$0 \le i \le r$.  Every permutation agreeing with $\rho$ on $[\,0, \, 2^m - 1\,]$ lies
in $U$ and has $\chi_{2^m} \ge r+1$, so these permutations form a non-empty open
subset of $U$ disjoint from $V_r$.
\end{proof}

\begin{cor}[What $G_{\mathrm{joint}} = \mathcal{W}_{\mathrm{dyad}}$ requires]\label{cor:Gjoint_equals_Wdyad}
$G_{\mathrm{joint}} = \mathcal{W}_{\mathrm{dyad}}$ if and only if
$\mathcal{W} \subseteq G_{\mathrm{joint}}$ and
$\mathcal{P}_{\mathrm{dyad}} \subseteq G_{\mathrm{joint}}$.
\end{cor}

\begin{proof}
If $G_{\mathrm{joint}} = \mathcal{W}_{\mathrm{dyad}}$, both containments hold because
$\mathcal{W}$ and $\mathcal{P}_{\mathrm{dyad}}$ lie in $\mathcal{W}_{\mathrm{dyad}}$.
Conversely, if both hold, then
$\mathcal{W}_{\mathrm{dyad}} = \mathcal{P}_{\mathrm{dyad}} \circ \mathcal{W}
\subseteq G_{\mathrm{joint}}$ by theorem \ref{thm:Wdyad_factorization}, and
$G_{\mathrm{joint}} \subseteq \mathcal{W}_{\mathrm{dyad}}$ by theorem
\ref{thm:Gjoint_in_Wdyad}.
\end{proof}

\begin{rem}[Two open questions]\label{rem:joint_questions}
Corollary \ref{cor:Gjoint_equals_Wdyad} reduces $G_{\mathrm{joint}} =
\mathcal{W}_{\mathrm{dyad}}$ to two questions.
\begin{enumerate}
  \item[(a)] Is $\mathcal{W} \subseteq G_{\mathrm{joint}}$?  A positive answer to
  question (a) of remark \ref{rem:new_questions}, $G_{\mathfrak{B}} = \mathcal{W}$,
  would imply it.
  \item[(b)] Is $\mathcal{P}_{\mathrm{dyad}} \subseteq G_{\mathrm{joint}}$?  By
  theorem \ref{thm:disjoint_pairs_all} and lemma \ref{lem:BSym}, $G_{\mathrm{joint}}$
  already contains every $\tau_C$, $C \subseteq \mathbb{N}$, and every finitary element
  of $\mathcal{P}_{\mathrm{dyad}}$.
\end{enumerate}
\end{rem}

\section{Obstructions for Fixed Interval Partitions}\label{sec:interval_partitions}

A natural way to remove the obstruction $\chi_{\mathrm{dyad}} < \infty$ is to add a
second family of disjoint-pair involutions whose intervals straddle the points $2^m$.
For example, when $m \ge 2$ let
\[
  J_m \;=\; \bigl[\, 3 \cdot 2^{m-1}, \; 3 \cdot 2^m - 1 \,\bigr] ,
\]
an interval of $3 \cdot 2^{m-1}$ points, and pair each point $n$ of the lower half of
$J_m$ with $n + 3 \cdot 2^{m-2}$ in the upper half.  Index these disjoint pairs by
$k \in \mathbb{N}$ in increasing order of the lower point, and let $\tilde{\tau}_C$ be
the permutation $\rho_C$ of theorem \ref{thm:disjoint_pairs_all} for these pairs.  The
point $2^{m+1}$ lies in the lower half of $J_m$, and the $2^{m-1}$ points
$n \in [\,3 \cdot 2^{m-1}, \, 2^{m+1} - 1\,]$ satisfy $n < 2^{m+1}$ and
$n + 3 \cdot 2^{m-2} \ge 2^{m+1}$.  Hence
$\chi_{2^{m+1}}(\tilde{\tau}_{\mathbb{N}}) \ge 2^{m-1}$ and
$\tilde{\tau}_{\mathbb{N}} \notin \mathcal{W}_{\mathrm{dyad}}$.  By theorem
\ref{thm:disjoint_pairs_all}, $\tilde{\tau}_{\mathbb{N}}$ lies in the group generated
by $\{\tilde{\tau}_A : A \in \mathcal{B}\}$.  So adding these generators to
$G_{\mathrm{joint}}$ produces a group not contained in $\mathcal{W}_{\mathrm{dyad}}$.

Theorem \ref{thm:interval_partitions_proper} shows that block-preserving permutations from finitely many
\emph{fixed} interval partitions never generate $\Sym(\mathbb{N})$.  A partition
$\mathcal{Q}$ of $\mathbb{N}$ into finite intervals is \textit{computable} if the set of
left endpoints of its blocks is computable.  For such $\mathcal{Q}$, let
$H(\mathcal{Q})$ be the group of all $\alpha \in \Sym(\mathbb{N})$ with
$\alpha(B) = B$ for every block $B \in \mathcal{Q}$, so that
$H(\mathcal{Q}) \cong \prod_{B \in \mathcal{Q}} \Sym(B)$.

\begin{thm}[Finitely many interval partitions generate a proper subgroup]\label{thm:interval_partitions_proper}
Let $\mathcal{Q}_1, \dots, \mathcal{Q}_q$ be partitions of $\mathbb{N}$ into finite
intervals, and let $f(n)$ be the largest right endpoint among the $q$ blocks that
contain $n$.  Set $F(n) = \max\{f(k) : k \le n\}$.  Then every word $\gamma$ of length
$r$ in $H(\mathcal{Q}_1) \cup \dots \cup H(\mathcal{Q}_q)$ satisfies
$\gamma(n) \le F^{(r)}(n)$ and $\gamma^{-1}(n) \le F^{(r)}(n)$ for all $n$, where
$F^{(r)}$ is the $r$-fold iterate of $F$.  Consequently
$\langle H(\mathcal{Q}_1) \cup \dots \cup H(\mathcal{Q}_q) \rangle$ is a proper subgroup
of $\Sym(\mathbb{N})$.  If every $\mathcal{Q}_i$ is computable, then $F$ is computable.
\end{thm}

\begin{proof}
$F$ is nondecreasing and $F(n) \ge n$.  If $g \in H(\mathcal{Q}_i)$, then $g(n)$ lies in
the block of $\mathcal{Q}_i$ containing $n$, so $g(n) \le f(n) \le F(n)$.  Let
$\gamma = g_r \circ \dots \circ g_1$ with each $g_j$ in some $H(\mathcal{Q}_i)$.  By
induction on $r$, using that $F$ is nondecreasing,
$\gamma(n) \le F\bigl(F^{(r-1)}(n)\bigr) = F^{(r)}(n)$.  Since
$\gamma^{-1} = g_1^{-1} \circ \dots \circ g_r^{-1}$ is again such a word, the same bound
holds for $\gamma^{-1}$.

Define $a_1 = 0$, $b_j = F^{(j)}(a_j) + 1$ and $a_{j+1} = b_j + 1$ for $j \ge 1$, so
$a_1 < b_1 < a_2 < b_2 < \cdots$.  Let $\psi$ interchange $a_j$ and $b_j$ for every
$j \ge 1$ and fix every other point.  Suppose $\psi = \gamma$ for a word of length $r$.
Since $F(n) \ge n$ and $F$ is nondecreasing, $F^{(j)}(n) \ge F^{(r)}(n)$ for $j \ge r$.
For $j \ge r$,
\[
  \psi(a_j) \;=\; b_j \;>\; F^{(j)}(a_j) \;\ge\; F^{(r)}(a_j) \;\ge\; \gamma(a_j) ,
\]
a contradiction.  If every $\mathcal{Q}_i$ is computable, then $f$, and hence $F$, is
computable.
\end{proof}

\begin{cor}\label{cor:dyadic_staggered_proper}
Let $\mathcal{Q}_1 = \{I_m : m \ge 0\}$, and let $\mathcal{Q}_2$ consist of
$[\,0, \, 5\,]$ and the intervals $J_m$, $m \ge 2$.  Then
$\langle H(\mathcal{Q}_1) \cup H(\mathcal{Q}_2) \rangle$ is a proper subgroup of
$\Sym(\mathbb{N})$.  It contains $G_{\mathrm{dyadic}}$, $\mathcal{P}_{\mathrm{dyad}}$,
and every $\tilde{\tau}_C$.  For these partitions $F(n) \le 2n + 5$, so every word of
length $r$ satisfies $\gamma(n) < 2^r (n + 5)$.
\end{cor}

\begin{proof}
Theorem \ref{thm:interval_partitions_proper} with $q = 2$.  Every $\tau_C$ and every
element of $\mathcal{P}_{\mathrm{dyad}}$ lies in $H(\mathcal{Q}_1)$, and every
$\tilde{\tau}_C$ lies in $H(\mathcal{Q}_2)$.  A block of either partition containing $n$
has right endpoint at most $2n + 1$ when $n \ge 2$, and at most $5$ when $n \le 5$, so
$F(n) \le 2n + 5$.  By induction, $F^{(r)}(n) + 5 \le 2^r (n + 5)$.
\end{proof}

\begin{rem}[The rearrangements $\sigma_A$ are not covered]\label{rem:sigma_not_covered}
Theorem \ref{thm:interval_partitions_proper} does not apply to groups containing the
rearrangements $\sigma_A$.  By lemma \ref{lem:up_iff_in}, $\sigma_A(n) \le n + 1$, but
the inverse is not bounded in this way: if $[\,m, \, n\,]$ is a maximal interval of $A$,
then $\sigma_A^{-1}(m) = n + 1$, and for some bi-immune $A$ the ratio $(n+1)/m$ is
unbounded.  The following question remains open: is the group generated by
$G_{\mathfrak{B}} \cup H(\mathcal{Q}_1) \cup H(\mathcal{Q}_2)$ a proper subgroup of
$\Sym(\mathbb{N})$?
\end{rem}

\newpage 

\section{An Adaptive Three-Involution Factorization}\label{sec:adaptive}

Theorem \ref{thm:interval_partitions_proper} shows that fixed partitions cannot reach
every permutation.  This section shows that partitions adapted to the permutation can reach every permutation.

\begin{defn}[Adapted cutoff sequences]
\label{defn:adapted_sequence}
Let $\pi \in \Sym(\mathbb{N})$.  A strictly increasing sequence
$0 = a_0 < a_1 < a_2 < \cdots$ is \textit{adapted to $\pi$} if every gap
$a_{r+1} - a_r$ is even, and for every $r \ge 1$
\begin{equation}
  \label{eq:adapted}
  \pi\bigl([\,0, \, a_r - 1\,]\bigr) \subseteq [\,0, \, a_{r+1} - 1\,]
  \quad\text{and}\quad
  \pi^{-1}\bigl([\,0, \, a_r - 1\,]\bigr) \subseteq [\,0, \, a_{r+1} - 1\,] .
\end{equation}
For such a sequence set $L_r = [\,a_r, \, a_{r+1} - 1\,]$ for $r \ge 0$, an interval of
even cardinality at least $2$, and $L_{-1} = \emptyset$.  Define two partitions of
$\mathbb{N}$ into finite intervals of even cardinality:
$ \mathcal{P}_1 = \{ P_k : k \ge 0 \}, \quad P_k = L_{2k} \cup L_{2k+1} ;
  \qquad
  \mathcal{P}_2 = \{ Q_k : k \ge 0 \},$ 
   $ Q_0 = L_0, \;\;  Q_k = L_{2k-1} \cup L_{2k} \;\; (k \ge 1)$. 
\end{defn}

\begin{lem}[The canonical adapted sequence]\label{lem:canonical_adapted}
For every $\pi \in \Sym(\mathbb{N})$, the sequence defined by $a_0 = 0$, $a_1 = 2$, and,
for $r \ge 1$,
\[
  M_r = 1 + \max\bigl( \{a_r\} \cup \{\pi(x), \, \pi^{-1}(x) : x < a_r\} \bigr),
\]

\[
  a_{r+1} = \begin{cases} M_r & \text{if } M_r - a_r \text{ is even,} \\
  M_r + 1 & \text{if } M_r - a_r \text{ is odd,} \end{cases}
\]
is adapted to $\pi$ and is computable from $\pi$.  If $f$ is a nondecreasing function
with $\pi(x) \le f(x)$ and $\pi^{-1}(x) \le f(x)$ for all $x$, then the same definition
with $M_r = 1 + \max(a_r, f(a_r))$ also gives a sequence adapted to $\pi$, which is
computable from $f$.
\end{lem}

\begin{proof}
$M_r > a_r$, so $a_{r+1} - a_r \ge 1$ is made even, hence at least $2$; and
$a_1 - a_0 = 2$.  For $x < a_r$, $\pi(x)$ and $\pi^{-1}(x)$ are below $M_r \le a_{r+1}$,
which is \eqref{eq:adapted}.  The sequence is computable from $\pi$ because $\pi^{-1}$
is computable from $\pi$ by search.  In the second version,
$\pi(x), \pi^{-1}(x) \le f(x) \le f(a_r) < M_r$ for $x < a_r$.
\end{proof}

\begin{lem}[One-step containment and interface balance]\label{lem:interface_balance}
Let $(a_r)$ be adapted to $\pi$.  Then for every $r \ge 0$,
$\pi(L_r) \subseteq L_{r-1} \cup L_r \cup L_{r+1}$.  Moreover, for every $k \ge 0$ the
sets
\[
  U_k = \{ x \in L_{2k} : \pi(x) \ge a_{2k+1} \}, \qquad
  D_k = \{ x \in L_{2k+1} : \pi(x) < a_{2k+1} \}
\]
satisfy $|U_k| = |D_k|$, and they are the points that cross the cutoff
$a_{2k+1}$ upward and downward.
\end{lem}

\begin{proof}
Let $x \in L_r$.  Since $x < a_{r+1}$ and $r + 1 \ge 1$, \eqref{eq:adapted} gives
$\pi(x) < a_{r+2}$.  If $r \ge 2$ and $\pi(x) < a_{r-1}$, then \eqref{eq:adapted} for
$\pi^{-1}$ at index $r-1 \ge 1$ gives $x = \pi^{-1}(\pi(x)) < a_r$, which is false.  If
$r \le 1$, then $a_{r-1} \le 0$ and there is nothing to prove.  Hence
$a_{r-1} \le \pi(x) < a_{r+2}$, which is the first statement.

Let $c = a_{2k+1}$.  By lemma \ref{lem:up_equals_down},
$|\{x < c : \pi(x) \ge c\}| = |\{x \ge c : \pi(x) < c\}|$.  If $x < a_{2k}$, then
$\pi(x) < a_{2k+1} = c$ by \eqref{eq:adapted} (for $k = 0$ there is no such $x$), so the
left side counts  $U_k$.  If $\pi(x) < c = a_{2k+1}$, then
$x < a_{2k+2}$ by \eqref{eq:adapted}, so the right side counts $D_k$.
\end{proof}

\begin{lem}[Finite permutations are products of two involutions]\label{lem:finite_two_involutions}
Every permutation of a finite set is a product of two involutions of that set.
\end{lem}

\begin{proof}
It suffices to treat one cycle, since involutions on disjoint cycles can be combined.
Label the cycle as $n \mapsto n + 1$ on $\mathbb{Z}/\ell$.  The maps $s(n) = -n$ and
$t(n) = 1 - n$ are involutions of $\mathbb{Z}/\ell$, and $t(s(n)) = 1 + n$.
\end{proof}


\begin{thm}[Adaptive three-involution factorization]\label{thm:adaptive_factor}
Let $\pi \in \Sym(\mathbb{N})$ and let $(a_r)$ be adapted to $\pi$.  Then
$\pi = \tau_3 \circ \tau_2 \circ \tau_1$, where $\tau_1, \tau_2, \tau_3$ are involutions,
$\tau_1 \in H(\mathcal{P}_1)$, and $\tau_2, \tau_3 \in H(\mathcal{P}_2)$.  If $(a_r)$ is
the canonical sequence of lemma \ref{lem:canonical_adapted}, then $\mathcal{P}_1$,
$\mathcal{P}_2$, $\tau_1$, $\tau_2$ and $\tau_3$ are computable from $\pi$.
\end{thm}

\noindent
The factorization $\pi = \beta \circ \tau_1$, with $\tau_1 \in H(\mathcal{P}_1)$ and
$\beta \in H(\mathcal{P}_2)$, is Lemma 9 of Bergman and Shelah
\cite{bergman_shelah}, who prove that every permutation of $\mathbb{N}$ is a product of
two local permutations by the same construction.  The new content of theorem
\ref{thm:adaptive_factor} is that each factor can be taken to be an involution, with
$\beta$ split into two involutions on the finite blocks of $\mathcal{P}_2$, and that the
partitions and the three involutions are computable from $\pi$.  This is what lemma
\ref{lem:two_bi_split}, corollary \ref{cor:universal_gen} and theorem
\ref{thm:oracle_constructivity} use.  The proof is included for completeness.

\begin{proof}
Following \cite[Lemma 9]{bergman_shelah}: for each $k \ge 0$, lemma
\ref{lem:interface_balance} yields $|U_k| = |D_k|$ with
$U_k \subseteq L_{2k}$ and $D_k \subseteq L_{2k+1}$.  Let $\varphi_k : U_k \to D_k$ be
the increasing bijection, and let $\tau_1$ interchange $x$ and $\varphi_k(x)$ for every
$k$ and every $x \in U_k$, fixing all other points.  Then $\tau_1$ is an involution and
$\tau_1(P_k) = P_k$ for every $k$.

Set $\beta = \pi \circ \tau_1$.  We show that $\tau_1(x)$ and $\pi(x)$ lie in the same
block of $\mathcal{P}_2$ for every $x$.  Note that $L_{2k} \subseteq Q_k$ and
$L_{2k+1} \subseteq Q_{k+1}$.  Let $x \in P_k$.
\begin{itemize}
  \item $x \in U_k$.  Then $\tau_1(x) \in D_k \subseteq L_{2k+1}$.  Also
  $\pi(x) \ge a_{2k+1}$ and $\pi(x) \in L_{2k-1} \cup L_{2k} \cup L_{2k+1}$, so
  $\pi(x) \in L_{2k+1}$.  Both lie in $Q_{k+1}$.
  \item $x \in L_{2k} - U_k$.  Then $\tau_1(x) = x \in L_{2k}$, and
  $\pi(x) < a_{2k+1}$ with $\pi(x) \in L_{2k-1} \cup L_{2k}$.  Both lie in $Q_k$.
  \item $x \in D_k$.  Then $\tau_1(x) \in U_k \subseteq L_{2k}$.  Also
  $\pi(x) < a_{2k+1}$ and $\pi(x) \in L_{2k} \cup L_{2k+1} \cup L_{2k+2}$, so
  $\pi(x) \in L_{2k}$.  Both lie in $Q_k$.
  \item $x \in L_{2k+1} - D_k$.  Then $\tau_1(x) = x \in L_{2k+1}$, and
  $\pi(x) \ge a_{2k+1}$ with $\pi(x) \in L_{2k+1} \cup L_{2k+2}$.  Both lie in
  $Q_{k+1}$.
\end{itemize}
Writing $y = \tau_1(x)$, we have $\beta(y) = \pi(x)$, so $\beta$ maps each $y$ into the
block of $\mathcal{P}_2$ containing $y$.  Hence $\beta(Q_k) \subseteq Q_k$, and since
$Q_k$ is finite and $\beta$ is one-to-one, $\beta(Q_k) = Q_k$.  By lemma
\ref{lem:finite_two_involutions}, applied on each finite block $Q_k$, there are
involutions $\tau_2, \tau_3 \in H(\mathcal{P}_2)$ with $\beta = \tau_3 \circ \tau_2$.
Then $\pi = \beta \circ \tau_1^{-1} = \tau_3 \circ \tau_2 \circ \tau_1$.

If $(a_r)$ is computable from $\pi$, then so are the blocks, the sets $U_k, D_k$, and
$\tau_1$; and $\tau_2, \tau_3$ are found by a finite search on each block $Q_k$.
\end{proof}


\section{Universal Generation}\label{sec:universal}

\begin{defn}[Block pairings and bi-immune block involutions]\label{defn:block_involution}
Let $\mathcal{P}$ be a partition of $\mathbb{N}$ into finite intervals of even
cardinality.  A \textit{pairing} of $\mathcal{P}$ is a partition $\mathfrak{p}$ of each
block of $\mathcal{P}$ into $2$-element sets.  List the $2$-element sets of
$\mathfrak{p}$ as $t_0, t_1, t_2, \dots$ in increasing order of their smaller elements,
and for $S \subseteq \mathbb{N}$ let $\tau_{\mathfrak{p}, S}$ interchange the two points
of $t_n$ for every $n \in S$ and fix every other point.  If $S \in \mathcal{B}$, then
$\tau_{\mathfrak{p}, S}$ is a \textit{bi-immune block involution} on $\mathcal{P}$.  Let
$\mathcal{I}_{\mathrm{bi}}$ be the set of bi-immune block involutions on all partitions
$\mathcal{P}$ and all pairings $\mathfrak{p}$, and let
$\mathcal{I}_{\mathrm{bi}}^{\mathrm{comp}}$ be the subset in which $\mathcal{P}$ is
computable (the pairing $\mathfrak{p}$ is arbitrary).
\end{defn}

\begin{lem}[Every block involution is a product of two bi-immune ones]
\label{lem:two_bi_split}
Let $\mathcal{P}$ be a partition of $\mathbb{N}$ into finite intervals of even
cardinality, and let $\tau \in H(\mathcal{P})$ be an involution.  Then there are a
pairing $\mathfrak{p}$ of $\mathcal{P}$, a set $R \subseteq \mathbb{N}$, and
$A, B \in \mathcal{B}$ such that
$\tau = \tau_{\mathfrak{p}, R} = \tau_{\mathfrak{p}, A} \circ \tau_{\mathfrak{p}, B}$.
If $\tau$ and $\mathcal{P}$ are computable, then $\mathfrak{p}$ and $R$ can be chosen
computable, and $B = A \oplus R$ is bi-immune for \emph{every} $A \in \mathcal{B}$.
\end{lem}

\begin{proof}
On each block $D$ of $\mathcal{P}$, $\tau$ restricted to $D$ is an involution.  It
moves an even number of points, and $|D|$ is even, so it fixes an even number of points.
Let $\mathfrak{p}$ consist of the $2$-cycles of $\tau$ together with the fixed points of
$\tau$ in each block, paired in increasing order.  Let $R$ be the set of indices $n$ for
which $t_n$ is a $2$-cycle of $\tau$.  Then $\tau = \tau_{\mathfrak{p}, R}$.  By
theorem \ref{thm:disjoint_pairs_all}, applied to the transpositions $t_n$, there is
$A \in \mathcal{B}$ with $B = A \oplus R \in \mathcal{B}$, and
$\tau_{\mathfrak{p}, R} = \tau_{\mathfrak{p}, A} \circ \tau_{\mathfrak{p}, B}$.  If
$\tau$ and $\mathcal{P}$ are computable, then $\mathfrak{p}$ and $R$ are computable, and
lemma \ref{lem:recursive_xor_bi_immune} shows that $A \oplus R$ is bi-immune for every
bi-immune $A$.
\end{proof}

\begin{defn}[The computably bounded subgroup]
\label{defn:Bcomp}
Let $\mathcal{B}_{\mathrm{comp}}$ be the set of $\pi \in \Sym(\mathbb{N})$ for which
there is a total computable function $f$ with $\pi(n) \le f(n)$ and
$\pi^{-1}(n) \le f(n)$ for all $n$.
\end{defn}

\begin{lem}
\label{lem:Bcomp_subgroup}
$\mathcal{B}_{\mathrm{comp}}$ is a proper subgroup of $\Sym(\mathbb{N})$.  It contains
every Turing computable permutation and every permutation of hyperimmune-free degree.
\end{lem}

\begin{proof}
Replacing $f$ by $n \mapsto \max\{f(k) : k \le n\}$, we may assume $f$ is
nondecreasing.  If $f$ bounds $\pi$ and $g$ bounds $\rho$, both nondecreasing, then
$\pi(\rho(n)) \le f(g(n))$ and $\rho^{-1}(\pi^{-1}(n)) \le g(f(n))$, so
$\max(f \circ g, \, g \circ f)$ bounds $\pi \circ \rho$; and $f$ bounds $\pi^{-1}$.  So
$\mathcal{B}_{\mathrm{comp}}$ is a subgroup.  If $\pi$ is computable, then so is
$\pi^{-1}$, and $f(n) = \max(\pi(n), \pi^{-1}(n))$ is computable.  If $\pi$ has
hyperimmune-free degree, then the $\pi$-computable function
$n \mapsto \max(\pi(n), \pi^{-1}(n))$ is bounded by a total computable function.

Proper: let $\varphi_0, \varphi_1, \dots$ list the partial computable functions, and let
$g(n) = \max\{\varphi_e(n) : e \le n \text{ and } \varphi_e(n) \text{ is defined}\}$
(with $\max \emptyset = 0$).  Every total computable $\varphi_e$ satisfies
$\varphi_e(n) \le g(n)$ for $n \ge e$.  Define $a_1 = 0$, $b_j = g(a_j) + 1$,
$a_{j+1} = b_j + 1$, and let $\psi$ interchange $a_j$ and $b_j$ for every $j$.  If a
total computable $\varphi_e$ bounded $\psi$, then for $a_j \ge e$ we would have
$\varphi_e(a_j) \ge \psi(a_j) = b_j > g(a_j) \ge \varphi_e(a_j)$.
\end{proof}

\newpage

\begin{cor}[Universal generation]\label{cor:universal_gen}

\leavevmode

\begin{enumerate}
  \item Every $\pi \in \Sym(\mathbb{N})$ is a product of at most six elements of
  $\mathcal{I}_{\mathrm{bi}}$.  Hence $\langle \mathcal{I}_{\mathrm{bi}} \rangle =
  \Sym(\mathbb{N})$.
  \item $\langle \mathcal{I}_{\mathrm{bi}}^{\mathrm{comp}} \rangle =
  \mathcal{B}_{\mathrm{comp}}$, and every $\pi \in \mathcal{B}_{\mathrm{comp}}$ is a
  product of at most six elements of $\mathcal{I}_{\mathrm{bi}}^{\mathrm{comp}}$.
\end{enumerate}
\end{cor}

\begin{proof}
(1) Let $(a_r)$ be the canonical sequence of lemma \ref{lem:canonical_adapted}.
Theorem \ref{thm:adaptive_factor} yields $\pi = \tau_3 \circ \tau_2 \circ \tau_1$ with
each $\tau_i$ an involution in $H(\mathcal{P}_1)$ or $H(\mathcal{P}_2)$, and both
partitions have blocks of even cardinality.  Lemma \ref{lem:two_bi_split} writes each
$\tau_i$ as a product of two elements of $\mathcal{I}_{\mathrm{bi}}$.

(2) Each element of $\mathcal{I}_{\mathrm{bi}}^{\mathrm{comp}}$ lies in
$H(\mathcal{P})$ for a computable $\mathcal{P}$.  A word in finitely many of them is
bounded, together with its inverse, by the computable function $F^{(r)}$ of theorem
\ref{thm:interval_partitions_proper}, so it lies in $\mathcal{B}_{\mathrm{comp}}$.
Conversely, let $\pi \in \mathcal{B}_{\mathrm{comp}}$ be bounded by a computable
nondecreasing $f$.  The second sequence of lemma \ref{lem:canonical_adapted} is
computable and adapted to $\pi$, so $\mathcal{P}_1$ and $\mathcal{P}_2$ are computable.
Theorem \ref{thm:adaptive_factor} and lemma \ref{lem:two_bi_split} write $\pi$ as a
product of six bi-immune block involutions on $\mathcal{P}_1$ and $\mathcal{P}_2$.
\end{proof}

\begin{thm}[How much computation the factorization needs]
\label{thm:oracle_constructivity}

\leavevmode

\begin{enumerate}
  \item Let $\pi$ be a Turing computable permutation and let $A \in \mathcal{B}$ be
  \emph{any} bi-immune set.  Then $\pi$ is a product of six bi-immune block involutions
  $\tau_{\mathfrak{p}_i, A} \circ \tau_{\mathfrak{p}_i, A \oplus R_i}$, $i = 1, 2, 3$,
  in which the partitions, the pairings $\mathfrak{p}_i$ and the sets $R_i$ are
  computable.  All six factors are computable from the oracle $A$.
  \item Let $\pi \in \Sym(\mathbb{N})$ be arbitrary.  Then the six bi-immune sets in a
  factorization of $\pi$ can be chosen computable from $\pi'$, the halting problem
  relative to $\pi$.
\end{enumerate}
\end{thm}

\begin{proof}
(1) Since $\pi$ and $\pi^{-1}$ are computable, the canonical adapted sequence, the
partitions $\mathcal{P}_1, \mathcal{P}_2$ and the involutions $\tau_1, \tau_2, \tau_3$
of theorem \ref{thm:adaptive_factor} are computable.  By lemma \ref{lem:two_bi_split},
the pairings $\mathfrak{p}_i$ and sets $R_i$ with
$\tau_i = \tau_{\mathfrak{p}_i, R_i}$ are computable, and $A \oplus R_i$ is bi-immune.
By theorem \ref{thm:disjoint_pairs_all},
$\tau_i = \tau_{\mathfrak{p}_i, A} \circ \tau_{\mathfrak{p}_i, A \oplus R_i}$.  Each
factor is computed from $A$ by computing $\mathfrak{p}_i$ and $R_i$ and querying $A$.

(2) The sets $R_i$ are computable from $\pi$.  A set $A$ that is $1$-generic relative to
$\pi$ can be built computably in $\pi'$ by the usual finite-extension construction: at
stage $e$, ask $\pi'$ whether the current finite word has an extension in the $e$-th
set of words c.e. in $\pi$, and if so extend to one.  Such an $A$ is $1$-generic, hence
bi-immune, and $A \oplus R_i$ is $1$-generic relative to $\pi$ as well, because
$X \mapsto X \oplus R_i$ is a $\pi$-computable homeomorphism.  So $A$ and the three sets
$A \oplus R_i$ are bi-immune and computable from $\pi'$.
\end{proof}

\begin{rem}[Where the constructions sit in the arithmetic hierarchy]
\label{rem:arithmetic_hierarchy}
For a set $X \subseteq \mathbb{N}$, the \textit{Turing jump} $X'$ is the halting problem
relative to $X$: the set of $e$ such that the $e$-th Turing machine with oracle $X$
halts on input $e$.  A set is $\Delta^0_1$ if it is computable, and $\Delta^0_2$ if it
is computable from $\emptyset'$; relative to $X$, these are $\Delta^0_1(X)$ and
$\Delta^0_2(X)$.  A set $G$ is \textit{$1$-generic relative to $X$} if, for every set
$S$ of finite binary words that is c.e. relative to $X$, some initial segment of $G$
either lies in $S$ or has no extension in $S$.

No bi-immune set is $\Delta^0_1$, since a computable infinite set is its own infinite
c.e. subset.  Bi-immune sets do exist in $\Delta^0_2$: the finite-extension
construction in the proof of theorem \ref{thm:oracle_constructivity}(2), carried out
with $X = \emptyset$, produces a $1$-generic set computable from $\emptyset'$.  The
results of this section therefore separate into three levels.
\begin{enumerate}
  \item[(a)] For a computable permutation $\pi$, the partitions, pairings and sets
  $R_i$ are $\Delta^0_1$, and the only non-computable ingredient is the bi-immune
  oracle $A$ (theorem \ref{thm:oracle_constructivity}(1)).  This step uses lemma
  \ref{lem:recursive_xor_bi_immune}.
  \item[(b)] For an arbitrary permutation $\pi$, the partitions and the involutions
  $\tau_1, \tau_2, \tau_3$ are $\Delta^0_1(\pi)$, and the bi-immune sets can be chosen
  $\Delta^0_2(\pi)$ (theorem \ref{thm:oracle_constructivity}(2)).
  \item[(c)] The Baire category argument of theorem \ref{thm:disjoint_pairs_all} only proves
  existence.  It does not locate any level of the hierarchy at which the witnesses $A$ and
  $A \oplus R$ can be found.
\end{enumerate}

A bi-immune source is non-mechanistic: it is not $\Delta^0_1$.  
Level (a) shows that such a source, combined with purely mechanistic ($\Delta^0_1$) data, 
suffices for every computable permutation.
\end{rem}

\begin{rem}[What these results use]\label{rem:sym_diff_covering}
Corollary \ref{cor:universal_gen} uses bi-immunity only via  theorem
\ref{thm:disjoint_pairs_all}:  namely, that every set $R$ is a
symmetric difference $A \oplus B$ of two bi-immune sets.
 The same conclusion holds for
any comeager family in place of $\mathcal{B}$.  However, this argument is not
constructive.  The Baire category theorem shows that some bi-immune $A$ has
$A \oplus R$ bi-immune, but it does not say how to find one, and no comeager family of
this kind has a computable member.  In general, producing such an $A$ requires an oracle
such as $\pi'$, as in theorem \ref{thm:oracle_constructivity}(2).

Lemma \ref{lem:recursive_xor_bi_immune} removes this obstacle for computable
permutations.  When $R$ is computable, \emph{every} bi-immune $A$ satisfies
$A \oplus R \in \mathcal{B}$, so no search is needed, the same $A$ works for every
computable $R$, and the factors are computed from $A$ by a Turing machine without any
jump oracle.  This is what makes theorem \ref{thm:oracle_constructivity}(1) effective
and uniform: one bi-immune source $A$, for instance a quantum random one, combined with
computable data, produces every computable permutation as a product of six bi-immune
block involutions.  Thus generation of $\Sym(\mathbb{N})$ follows from comeagerness
alone, while constructive generation from a single bi-immune source rests on the
preservation of bi-immunity by $\oplus$.
\end{rem}

%% file: bi_immune_references.tex

\bibliographystyle{plain}